\documentclass[10pt,reqno]{amsart}
\usepackage[margin=0.8in]{geometry}
\usepackage{amsthm, amsmath,amsfonts,amssymb,euscript,graphics,color,slashed}
\usepackage{graphicx}
\usepackage{mathrsfs}
\usepackage{comment}
\usepackage{import}
\usepackage{latexsym}
\usepackage[makeroom]{cancel}

\def\norm#1#2{\|#1\|_{#2}}

\def\inte#1{
\displaystyle\mathop{#1\kern0pt}^\circ }

\let\e=\varepsilon
\let\z=\zeta

\let\f=\frac

\let\p=\psi

\let\Lam=\Lambda

\let\wt=\widetilde

\def\cA{{\mathcal A}}

\def\cJ{{\mathcal J}}
\def\cK{{\mathcal K}}

\def\cN{{\mathcal N}}

\def\cQ{{\mathcal Q}}

\def\cS{{\mathcal S}}
\def\cT{{\mathcal T}}

\def\dH{\dot{H}}
\def\dB{\dot{B}}

\def\virgp{\raise 2pt\hbox{,}}
\def\cdotpv{\raise 2pt\hbox{;}}

\def\eqdefa{\buildrel\hbox{\footnotesize def}\over =}

\def\C{\mathop{\mathbb C\kern 0pt}\nolimits}
\def\DD{\mathop{\mathbb D\kern 0pt}\nolimits}
\def\EE{\mathop{{\mathbb E \kern 0pt}}\nolimits}
\def\K{\mathop{\mathbb K\kern 0pt}\nolimits}
\def\N{\mathop{\mathbb N\kern 0pt}\nolimits}
\def\Q{\mathop{\mathbb Q\kern 0pt}\nolimits}
\def\R{\mathop{\mathbb R\kern 0pt}\nolimits}
\def\SS{\mathop{\mathbb S\kern 0pt}\nolimits}
\def\ZZ{\mathop{\mathbb Z\kern 0pt}\nolimits}
\def\TT{\mathop{\mathbb T\kern 0pt}\nolimits}
\def\P{\mathop{\mathbb P\kern 0pt}\nolimits}

\def\na{\nabla}
\def\p{\partial}

\newcommand{\beq}{\begin{equation}}
\newcommand{\eeq}{\end{equation}}
\newcommand{\ben}{\begin{eqnarray}}
\newcommand{\een}{\end{eqnarray}}
\newcommand{\beno}{\begin{eqnarray*}}
\newcommand{\eeno}{\end{eqnarray*}}

\newcommand{\vv}[1]{\boldsymbol{#1}}

\def\div{\text{div}\,}

\newtheorem*{Main Theorem}{Main Theorem}
\newtheorem{theorem}{Theorem}[section]
\newtheorem{lemma}[theorem]{Lemma}
\newtheorem{proposition}[theorem]{Proposition}
\newtheorem{corollary}[theorem]{Corollary}
\newtheorem{definition}[theorem]{Definition}
\newtheorem{remark}[theorem]{Remark}

\numberwithin{equation}{section}

\begin{document}
\title[Well-posedness]{Long time existence for a class of   weakly transverse Boussinesq systems}

\author[Qi Li]{Qi Li}
\address{School of Mathematical Sciences, Beihang University\\  100191 Beijing, China}
\email{Ricci@buaa.edu.cn}

\author{Jean-Claude Saut}
\address{Laboratoire de Math\' ematiques, UMR 8628\\
Universit\' e Paris-Saclay et CNRS\\ 91405 Orsay, France}
\email{jean-claude.saut@universite-paris-saclay.fr}

\author[Li XU]{Li Xu}
\address{School of Mathematical Sciences, Beihang University\\  100191 Beijing, China}
\email{xuliice@buaa.edu.cn}

\maketitle

\vspace{1cm}
\textit{Abstract}. We prove the existence on long time scales of the solutions to the Cauchy problem for a  version of weakly transverse Boussinesq systems arising in the modeling of surface water waves.  This system is much more complicated than the isotropic Boussinesq systems because  dispersion is only present in the x-direction,  leading to anisotropic eigenvalues in the linearized system. This anisotropic character leads to loss of  y-derivatives for the solutions. To overcome this main difficulty our strategy is to symmetrize the system by introducing suitable good unknowns in the sense of \cite{Alin}.

\vspace{1cm}

\setcounter{tocdepth}{1}
 \tableofcontents

\setcounter{equation}{0}
\section{Introduction}
\subsection{The general setting}

This paper is concerned with systems arising in the modeling of weakly nonlinear surface waves. More precisely we consider the {\it weakly transverse} (or KP) regime which  can be specified in terms of the relevant characteristics of the wave, namely, its typical amplitude $a$, the mean depth $h$, the typical wavelength $\lambda$  along the longitudinal direction (say, the $x$ axis) and $\mu$ , the wavelength along the transverse direction (say, the $y$ axis):




$$\frac{a}{h}=\varepsilon,\quad \frac{\lambda^2}{h^2}=\frac{S_1}{\varepsilon},\quad \frac{\mu^2}{h^2}=\frac{S_2}{\varepsilon^2},$$
where $\varepsilon\ll 1$ is a small dimensionless parameter, while $S_1\sim 1$ and $S_2\sim 1$ (so that the Stokes number is $S_1$ along the longitudinal direction and $S_2/{\epsilon}$ along the transverse one). For simplicity, we set $S_1 = S_2 = 1$ throughout this paper.

In this regime one usually derives the Kadomtsev-Petviashvili (KP) equation:

\begin{equation}\label{KP}
u_t+u_x+\varepsilon uu_x+\varepsilon u_{xxx}+\varepsilon(T-\frac{1}{3})\partial_x^{-1}u_{yy}=0,
\end{equation}
where $T\geq 0$ is the surface tension parameter.

However, the non-physical zero mass condition inherent to the structure of the KP equation implies the poor precision of the KP approximation. In fact, see \cite{La1, LS} this precision is $o(1)$ while for the isotropic Boussinesq systems the optimal rate is $O(\varepsilon^2t),$ see \cite{BCL}. 

In order to overcome this shortcoming, a class of {\it weakly transverse Boussinesq systems} was introduced in \cite{LS}. These new systems provide a much more precise approximation than the KP equation and do not require any zero mass assumption. Actually the error is the same that in the case of isotropic Boussinesq systems, namely $O(\varepsilon^2 t).$

Nevertheless,  to fully justify those systems one needs to prove the existence of solutions to the Cauchy problem on a {\it large} time scale, that is $O(1/\varepsilon).$ This is by no way an easy task, even for the "classical" abcd Boussinesq systems derived in \cite{BCS1, BCS2} (see \cite{Bu, Bu2, MSZ, SX, SX2, SX3, SWX, KS}).

Actually even the local well-posedness of the Cauchy problem for the weakly transverse Boussinesq systems is not trivial and we will have to introduce an equivalent system, also consistent with the water wave system and for which we will prove local and long time existence.

\begin{remark}
The systems we will study are different from that in  the class of {\it fully symmetric} weakly transverse Boussinesq systems introduced in \cite{LS} for which the existence in long time is relatively standard.
\end{remark}

\vspace{0.3cm}

We thus consider in this paper  the class of weakly transverse Boussinesq systems derived in  \cite{LS})
\beq\label{WTB 1}\left\{\begin{aligned}
&(1-b\e \p_x^2)v_t+(1+a\e \p_x^2)\z_x+\e(vv_x+\f12ww_x)+\f{\e^{\f32}}{2}wv_y=0,\quad t>0,\, (x,y)\in\R^2,\\
&(1-e\e \p_x^2)w_t+\e^{\f12}(1+f\e \p_x^2)\z_y+\f\e2vw_x+\e^{\f32}(ww_y+\f12vv_y)=0,\\
&(1-d\e \p_x^2)\z_t+(1+c\e \p_x^2)v_x+\e^{\f12}(1+g\e\p_x^2)w_y+\e(\z v)_x+\e^{\f32}(\z w)_y=0,
\end{aligned}\right.\eeq
where $\z(t,x,y)$ is the elevation of the surface wave and $(v(t,x,y),w(t,x,y))$ is an $O(\e^2)$ approximation of the horizontal velocity.
Here $a, b, c, d, e, f, g$ are modeling parameters which satisfy the
constraint \footnote {We neglect surface tension effects}
\beq\label{constraint}
a+b+c+d=\f13,\quad d+e+f+g=\f23.
\eeq
We remark that \eqref{constraint} follows from the explicit expressions of the parameters chosen in \cite{LS}. In this article, we use the constraint \eqref{constraint} instead of the explicit expressions stated in \cite{LS}.

According to the weakly transverse regime, in \eqref{WTB 1},  the small parameter $\varepsilon$ is defined by
$$\varepsilon=\f{a_0}{h}\sim\f{h^2}{\lambda^2}\sim\f{h}{\mu},$$
where $h$ denotes the mean depth of the fluid, $a_0$ denotes a typical amplitude of the wave and $\lambda,\mu$ denote a typical wavelength along $x$-axis and $y$-axis respectively. Notice that the scales of the variables $x$ and $y$ are $1$ and $\e^{\f12}$ in \eqref{WTB 1}.

The dispersion relation for the linearization of \eqref{WTB 1} at $(0, 0, 0)$ is
\beq\label{DR}
\Lam^2(\xi_1,\xi_2)=\xi_1^2\cdot\f{(1-a\e \xi_1^2)(1-c\e \xi_1^2)}{(1+b\e \xi_1^2)(1+d\e \xi_1^2)}+\e\xi_2^2\cdot\f{(1-f\e \xi_1^2)(1-g\e \xi_1^2)}{(1+e\e \xi_1^2)(1+d\e \xi_1^2)}.
\eeq

Therefore, the linearization of \eqref{WTB 1} around the null solution is well-posed provided that 
\beq\label{cases}\begin{aligned}
&(i)\, b\geq0,\,d\geq 0,\,e\geq0,\,a\leq0,\,c\leq0,\,f\leq0,\,g\leq0;\\
\text{or}\quad&(ii)\, b\geq0,\,d\geq 0,\,e\geq0,\,a\leq0,\,c\leq0,\,f=g;\\
\text{or}\quad&(iii)\, b\geq0,\,d\geq 0,\,e\geq0,\,a=c,\,f\leq0,\,g\leq0;\\
\text{or}\quad&(iv)\, b\geq0,\,d\geq 0,\,e\geq0,\,a=c,\,f=g.\\ 
\end{aligned}\eeq

In view of \eqref{constraint} and \eqref{cases},  achieving local well-posedness for the Cauchy problem of \eqref{WTB 1} presents challenges due to the loss of $y$-derivative. To our knowledge, there is no local well-posedness result for the Cauchy problem of \eqref{WTB 1}. The goal of this paper is to derive an equivalent system, which is also consistent with the water wave system, and for which we can establish the long time well-posedness of the Cauchy problem. Inspired by the methodology introduced in \cite{BCL}, by introducing a nonlinear transformation for $\zeta$, we formulate a new system that has the same accuracy as the original one and turns out to be symmetrizable in some special cases under the curl-free condition of the velocity field. Subsequently, for this new symmetrizable system, we ascertain the existence of solutions to the Cauchy problem over extended time intervals of order  $1/\varepsilon$ that is the {\it hyperbolic} time scale..

In this paper, we only consider the  cases 
\beq\label{case a}
b=e\geq0,\quad a=f=g\leq 0,\quad c\leq0,\quad d\geq0,
\eeq
with $(a,b,c,d,e,f,g)$ being satisfying the constraint \eqref{constraint}. For such cases, the curl-free condition
\beq\label{curl free}
  \varepsilon^{\frac{1}{2}} v_y= w_x.
\eeq
is preserved for all time $t>0$ in the lifespan, if it holds for $t=0$. The curl-free assumption \eqref{curl free} is crucial to derive the symmetrizable new system.

\subsection{Derivation of the new system} 
Assume that $(a,b,c,d,e,f,g)$ satisfies \eqref{constraint}. We introduce the nonlinear change of variable
\begin{equation}\label{nonlinear transform}
	\tilde\zeta = \zeta - \frac{\e}{4} |\zeta|^2.
\end{equation}

 Thanks to the equation for $\z$ in \eqref{WTB 1}, we have
\beq\label{A 1}\begin{aligned}
&(1-d\e\p_x^2)\tilde\z_t=(1-d\e\p_x^2)\z_t-\frac{\e}{2} \z\z_t+O(\e^2),\\
=&-(1+c\e\p_x^2)v_x-\e^{\f12}(1+g\e\p_x^2)w_y-\e\z v_x-\e^{\f32}\z w_y-\e v\z_x
-\e^{\f32}w\z_y-\frac{\e}{2} \z\z_t+O(\e^2).
\end{aligned}\eeq
Here and in what follows, it should be noted that the precision $O(\e^2)$ includes both linear and nonlinear terms that encompass $(v, w, \zeta)$ and their derivatives with respect to $\partial_t$, $\partial_x$, and $\e^{\frac{1}{2}}\partial_y$. A similar interpretation applies to the precision $O(\e)$ as well.

Due to \eqref{WTB 1} and \eqref{nonlinear transform}, we get 
\beno
\p_t\z=-v_x-\e^{\f12}w_y+O(\e)\quad\text{and}\quad\tilde\z=\z+O(\e),
\eeno
which along with\eqref{A 1} implies 
\beq\label{A 2}
(1-d\e \p_x^2 )\tilde\z_t + (1+c\e \p_x^2 ) v_x + \e^{\frac{1}{2}}(1+g \e\p_x^2 )w_y + \e v\tilde\z_x  + \e^{\frac{3}{2}} w\tilde\z_y+  \frac{\e}{2}\tilde\z v_x + \frac{\e^{\frac{3}{2}}}{2} \tilde\z w_y = O(\e^2).
\eeq

Following a similar derivation process as for \eqref{A 2}, we can derive the equations for $v$ and $w$. Subsequently, we arrive at the following system for $(v,w,\wt\z)$ :
\begin{equation}\label{approx of WTB1}
	\left\{\begin{aligned}
	&(1-b\e \p_x^2) v_t + (1+a \e \p_x^2) \tilde\z_x + \e (vv_x+ \frac{1}{2} ww_x) + \frac{\e^{\frac{3}{2}}}{2}  wv_y + \frac{\e}{2} \tilde\z\tilde\z_x = O(\e^2),  \\
	&(1-e\e \p_x^2) w_t + \e^{\frac{1}{2}}(1+f \e \p_x^2) \tilde\z_y +  \frac{\e }{2}vw_x + \e^{\frac{3}{2}} (ww_y+\frac{1}{2}vv_y+ \frac{1}{2}\tilde\z\tilde\z_y) = O(\e^2),\\
	&(1-d\e \p_x^2 )\tilde\z_t + (1+c\e \p_x^2 ) v_x + \e^{\frac{1}{2}}(1+g \e\p_x^2 )w_y + \e v\tilde\z_x  + \e^{\frac{3}{2}} w\tilde\z_y +  \frac{\e}{2}\tilde\z v_x+ \frac{\e^{\frac{3}{2}}}{2} \tilde\z w_y = O(\e^2).
	\end{aligned}\right.	
\end{equation}

Thus,  by neglecting the $O(\e^2)$ terms in \eqref{approx of WTB1}, we derive the following system (omitting the tildes $\tilde{}$ )
\beq\label{WTB case a}\left\{\begin{aligned}
&(1-b\e \p_x^2)v_t+(1+a\e\p_x^2)\z_x+\e (vv_x+ \frac{1}{2} ww_x) + \frac{\e^{\frac{3}{2}}}{2}  wv_y + \f\e2 \z\z_x=0,\quad t>0,\, (x,y)\in\R^2,\\
&(1-e\e \p_x^2)w_t+\e^{\f12}(1+f\e\p_x^2)\z_y+\frac{\e }{2}vw_x + \e^{\frac{3}{2}} (ww_y+\frac{1}{2}vv_y+ \frac{1}{2}\z\z_y)=0,\\
&(1-d\e \p_x^2)\z_t+(1+c\e\p_x^2)v_x+\e^{\f12}(1+g\e\p_x^2)w_y+\e v\z_x+\e^{\f32} w\z_y+ \f\e2 \z v_x + \f{\e^{\f32}}{2}\z w_y=0.
\end{aligned}\right.\eeq

\begin{definition}[see Definition 1 of \cite{BCL}]\label{def of consistent} The system \eqref{WTB 1} is consistent with a system $S$ of $3$ equations for 
$(v, w,\tilde\z)$, if for all sufficiently smooth solutions $(v,w,\z)\in C([0,T/\e];H^s_\e(\R^2))$ of \eqref{WTB 1}, the triplet $(v,w,\wt\z\eqdefa\z-\f{\e}{4}|\z|^2)\in C([0,T/\e];H^N_\e(\R^2))$ solves $S$ up to a small residual $\e^2 R$ called precision, where $R$ is bounded in $L^\infty([0,T/\e];H^N_\e(\R^2))$ $(N\geq 2)$, uniformly with respect to $\e\in(0,1)$. Here   $H^s_\e(\R^2)$ is the Sobolev spaces $H^s(\R^2)$ equipped with the following norm 
\beno
\|f\|_{H^s_\e}^2\eqdefa\e^{-\f12}\int_{\R^2}(|\xi_1|^2+\e|\xi_2|^2)^s|\hat{f}(\xi)|^2d\xi.
\eeno
\end{definition}

Thanks to the formal derivation and standard product estimates, we reach the following consistency result.

\begin{proposition}\label{consistent prop}
Assuming that $(a,b,c,d,e,f,g)$ satisfies \eqref{constraint}, the system \eqref{WTB case a} is consistent with the system \eqref{WTB 1} in the sense of Definition \ref{def of consistent} with the precision $O(\e^2)$.
\end{proposition}

\begin{remark}
\begin{enumerate}
	\item
	Since \eqref{WTB case a} is consistent with \eqref{WTB 1} and \eqref{WTB 1} is consistent with the water wave system, \eqref{WTB case a} is also consistent with the  water wave system.
	\item If $b=e\geq 0$, $a=f$, the curl-free condition $\e^{\f12}v_y=w_x$ will be preserved for all time $t>0$ in the lifespan of the solutions to \eqref{WTB 1} , and also of the solution to \eqref{WTB case a}, if it holds at $t=0$. With the curl-free condition $\e^{\f12}v_y=w_x$, the new system \eqref{WTB case a} turns out to  be symmetrizable. 
\end{enumerate}
\end{remark}

\subsection{Main result}

In this paper, we aim to demonstrate the large time existence on time scale $1/\e$ for the Cauchy problem of \eqref{WTB case a} with the curl-free condition \eqref{curl free} for cases \eqref{case a}. 

Observing that the scales of  variables $x$ and $y$ are $1$ and $\e^{\f12}$ in \eqref{WTB case a}, we introduce the following scaling for the sake of simplicity
\beq\label{scaling}
\wt{v}(t,x,y)=v(t,x,\e^{\f12}y),\quad \wt{w}(t,x,y)=w(t,x,\e^{\f12}y),
\quad \wt{\z}(t,x,y)=\z(t,x,\e^{\f12}y).
\eeq 
Then $(\wt{v},\wt{w},\wt{\z})$ satisfies \eqref{WTB case a} with $\e^{\f12}\p_y$ being replaced by $\p_y$ and $\wt{v}_y=\wt{w}_x$. Using curl-free condition  $\wt{v}_y=\wt{w}_x$, omitting  $\wt{}$  of $(\wt{v},\wt{w},\wt{\z})$, \eqref{WTB case a} reads
\beq\label{WTB 2}\left\{\begin{aligned}
&(1-b\e \p_x^2)v_t+(1+a\e \p_x^2)\z_x+\e vv_x+\e w v_y+\f\e2 \z\z_x=0,\quad t>0,\, (x,y)\in\R^2,\\
&(1-e\e \p_x^2)w_t+(1+f\e \p_x^2)\z_y+\e vw_x+\e ww_y+\f\e2\z\z_y=0,\\
&(1-d\e \p_x^2)\z_t+(1+c\e \p_x^2)v_x+(1+g\e\p_x^2)w_y+\e v\z_x+\e w \z_y+ \f\e2 \z v_x + \f{\e}{2}\z w_y=0.
\end{aligned}\right.\eeq
\begin{remark} \begin{enumerate}
\item To preserve the curl-free condition $v_y=w_x$, one needs to assume that $b=e\geq0,\, a=f$.
\item We remark that the nonlinear terms of system \eqref{WTB 2} are symmetric. However, according to \eqref{constraint}, \eqref{cases} and the restriction $b=e\geq0,\, a=f$, none of the possible cases for \eqref{WTB 2} is symmetric (for both linear and nonlinear terms).
\end{enumerate}
\end{remark}
We set the initial data for \eqref{WTB 2} as follows
\beq\label{initial data}
v|_{t=0}=v_0,\quad w|_{t=0}=w_0,\quad \z|_{t=0}=\z_0,
\eeq
and $\p_yv_0=\p_xw_0$.

 The non-zero eigenvalue of the linearization of  new system \eqref{WTB 2} are
\beq\label{eigen value a}
\lambda_{\pm}(\xi)=\pm i\Lambda(\xi)\quad\text{with}\quad\Lambda(\xi)=\Bigl(\xi_1^2\cdot\f{(1-a\e \xi_1^2)(1-c\e \xi_1^2)}{(1+b\e \xi_1^2)(1+d\e \xi_1^2)}+\xi_2^2\cdot\f{(1-f\e \xi_1^2)(1-g\e \xi_1^2)}{(1+e\e \xi_1^2)(1+d\e \xi_1^2)}\Bigr)^{\f12}.
\eeq
Due to \eqref{eigen value a} and \eqref{case a}, we rewrite $\Lambda(\xi)=\cJ_\e^{-1}(\xi_1)\cA(\xi)$ with
\beq\label{operator 1}
\cJ_\e(\xi_1)=\f{(1+b\e \xi_1^2)^{\f12}(1+d\e \xi_1^2)^{\f12}}{1-g\e \xi_1^2},\quad\cA(\xi)=\bigl(\xi_1^2\cdot\cK_\e(\xi_1)+\xi_2^2\bigr)^{\f12},\quad
\cK_\e(\xi_1)=\f{1-c\e \xi_1^2}{1-g\e \xi_1^2}.
\eeq
We remark that operators $\cK_\e(D_x)$ and $\cA(D)$ are crucial in the procedure of deriving the well-posedness theory of \eqref{WTB 2}.

From \eqref{constraint} and \eqref{case a}, we deduce that $c=g-\f13\leq-\f13$ and $b+d\geq\f23$. Then
\begin{itemize}
\item[1).] {\it if $g=0$}, there holds
\beq\label{operator 2}
\cA(\xi)\sim\cK_\e^{\f12}(\xi_1)|\xi_1|+|\xi_2|,\quad\cK_\e(\xi_1)=1+\f13\e \xi_1^2;
\eeq
\item[2).] {\it if $g<0$}, there holds
\beq\label{operator 3}
\cA(\xi)\sim\cK_\e^{\f12}(\xi_1)|\xi_1|+|\xi_2|\sim |\xi|,\quad\cK_\e(\xi_1)\sim 1.
\eeq
\end{itemize}

 According to the properties of $\cA(\xi)$ in \eqref{operator 2} and \eqref{operator 3}, we only consider the following two typical cases for simplicity:
\beq\label{special case}\begin{aligned}
	\textbf{case\, 1:}\quad &a=f=g=0, b=d=e=\f13,c=-\f13,\\
	\textbf{case\, 2:}\quad & a=f=g= -\frac{1}{6},b=d=e=\frac{1}{2},c=-\frac{1}{2}.
\end{aligned}\eeq
The corresponding eigenvalues are 
\begin{equation}\label{eigen value}
	\left\{  \begin{aligned}
		&\textbf{case 1:}\quad \lambda_{1,\pm} = \pm i \Lam_1(\xi), \quad
		\text{with}\quad \Lam_1(\xi) = \left( \xi_1^2\cdot\frac{1}{1+\f\e3 \xi_1^2}+{\xi_2^2} \cdot \frac{1}{(1+\f\e3 \xi_1^2)^2} \right)^{\f12},\\
		&\textbf{case 2:}\quad \lambda_{2,\pm} = \pm i \Lam_2(\xi), \quad
		\text{with}\quad \Lam_2(\xi) = \left(  \xi_1^2 \cdot \frac{1+\f\e6 \xi_1^2}{1+\f\e2 \xi_1^2} + \xi_2^2\cdot\left(  \frac{1+\f\e6 \xi_1^2}{1+\f\e2 \xi_1^2} \right)^2\right)^\f12  .
	\end{aligned} \right.
\end{equation}
We remark that
\beno
\Lambda_1(\xi)\sim(1+\f\e3\xi_1^2)^{-1}\bigl((1+\f\e3\xi_1^2)^{\f12}|\xi_1|+|\xi_2|\bigr),\quad\Lambda_2(\xi)\sim |\xi|.
\eeno

The main result of this paper is stated as follows:
\begin{theorem}\label{main theorem} Let $s>3$, $\e\in(0,1)$ and $J_\e=J_\e(D_x)=1-b\e\p_x^2.$
\begin{enumerate}
\item {(\bf Long time existence for case 1)} For $a=f=g=0, b=d=e=\f13,c=-\f13$, if $(v_0,w_0,\z_0)$ satisfies
\beq\label{initial 1}\begin{aligned}
&\bigl(J_\e v_0,\,J_\e^{\f12}\p_x v_0,\,\p_y v_0\bigr)\in H^s(\R^2),\quad
\bigl(J_\e^{\f12} w_0,\,\p_x w_0,\,J_\e^{-\f12}\p_y w_0\bigr)\in H^s(\R^2),\\
&\bigl(J_\e^{\f12}\z_0,\,\p_x\z_0,\,J_\e^{-\f12}\p_y\z_0\bigr)\in H^s(\R^2)
\quad\text{and}\quad \p_yv_0=\p_xw_0,
\end{aligned}\eeq
there exists a constant $\e_0>0$ such that for any $\e\in(0,\e_0)$, \eqref{WTB 2}-\eqref{initial data} has a unique solution $(v,w,\z)$ over a time interval $[0,T_0/\e]$ for some $T_0>0$. Moreover, there holds
\beq\label{energy estimate for case 1}\begin{aligned}
E_s(t)\eqdefa&\bigl\|\bigl(J_\e v,\,J_\e^{\f12}\p_x v,\,\p_y v\bigr)\bigr\|_{H^s}+
\bigl\|\bigl(J_\e^{\f12} w,\,\p_x w,\,J_\e^{-\f12}\p_y w\bigr)\bigr\|_{H^s}\\
&
+\bigl\|\bigl(J_\e^{\f12} w,\,\p_x w,\,J_\e^{-\f12}\p_y w\bigr)\bigr\|_{H^s}\lesssim E_s(0),\quad\forall t\in[0,T_0/\e].
\end{aligned}\eeq
\item {(\bf Long time existence for case 2)} For $a=f=g= -\frac{1}{6},b=d=e=\frac{1}{2},c=-\frac{1}{2}$, if $(v_0,w_0,\z_0)$ satisfies 
\beq\label{initial 2}
(J_\e^{\f12}v_0,\,J_\e^{\f12}w_0,\,J_\e^{\f12}\z_0)\in H^{s+1}(\R^2)\quad\text{and}
\quad\p_yv_0=\p_xw_0,
\eeq
there exists a constant $\e_0>0$ such that for any $\e\in(0,\e_0)$, \eqref{WTB 2}-\eqref{initial data} has a unique solution $(v,w,\z)$ over a time interval $[0,T_0/\e]$ for some $T_0>0$. Moreover, there holds
\beq\label{energy estimate for case 2}
E_s(t)\eqdefa\|J_\e^{\f12}v\|_{H^{s+1}}+\|J_\e^{\f12}w\|_{H^{s+1}}+\|J_\e^{\f12}\z\|_{H^{s+1}}\lesssim E_s(0),\quad\forall t\in[0,T_0/\e].
\eeq
\end{enumerate}	
\end{theorem}
\begin{remark}
\begin{enumerate}
\item In other scenarios outlined in \eqref{case a}, the structures of operators $\cA_\e(D)$ and $\cK_\e(D_x)$ (as shown in \eqref{operator 1}) closely resemble those in \eqref{operator 2} or \eqref{operator 3}. Consequently,  similar  assertion of long time existence also holds for \eqref{WTB 2}-\eqref{initial data}. The main idea of the proof is similar to Case 1 and Case 2.
\item It is worth noticing that case 1 is precluded given the explicit parameter expressions $(a, b, c, d, e, f, g)$ detailed in \cite{LS}. Nonetheless, it remains plausible under the constraint specified in \eqref{constraint}. Within the contexts of cases in \eqref{case a}, case 1 represents a straightforward and common scenario of system \eqref{WTB 2}, facilitating a clearer outline for proving long time existence.
\item Coming back to \eqref{WTB case a}, the theorem also holds with $H^s(\R^2)$-norms being replaced by $H^s_\e(\R^2)$-norms.
\end{enumerate}
\end{remark}

\subsection{Main idea of the proof}
In this subsection, we outline the main idea and the strategy of the proof of our main result. 

$\bullet$ {\bf The difficulty:} The principal difficulty in the proof stems from the loss of $y$-derivatives induced by dispersion only  in the $x$-direction. More precisely, in view of  \eqref{eigen value a} and \eqref{operator 1}, it becomes clear that the eigenvalues $\lambda_{\pm}(\xi)$ and the operator symbol $\mathcal{A}_\e(\xi)$ exhibit anisotropy concerning $\xi_1$ and $\xi_2$. Consequently, during the process of the proof, the appearance of commutators of the form $[P(D_x), f]\partial_y g$ arises, leading to the loss of $y$-derivatives. Here, $P(D_x)$ represents a Fourier multiplier involving $\partial_x$. To overcome the principal difficulty, our approach is to symmetrize the system by introducing appropriate good unknowns.

$\bullet$ {\bf Idea of the proof.} 

{\it 1). Symmetrization of the linear part.}

 Since $b=e\geq 0$, $a=f=g\leq 0$, we introduce the first good unknowns as follows
\beq\label{new unknowns case a}
p \eqdefa v_x+ \cK_\e^{-1} w_y,\quad \theta \eqdefa\cK_\e^{-1}\cJ_\e\Lam(D)\z=\cK_\e^{-1}\cA(D)\z,
\eeq
where $\cA(\xi)=\cJ_\e(\xi_1)\Lambda(\xi)$ and $\cK_\e=\cK_\e(\xi_1)$ are defined in \eqref{eigen value a} and \eqref{operator 1} respectively.
Due to \eqref{eigen value a}, \eqref{operator 1} and the condition $v_y=w_x$, $(v,w,\z)$ can be also rewritten in terms of $(p,\theta)$ as
\beq\label{v w zeta in terms of p theta case a}
\vv V\eqdefa(v,w)^T = - \cK_\e\cA^{-2}(D) \na p, \quad \z=\cK_\e\cA^{-1}(D)\theta.
\eeq

With $(p,\theta)$ and the curl-free condition $v_y=w_x$, we deduce from system \eqref{WTB 2} that
\beq\label{WTB 3}\left\{\begin{aligned}
&(1-b\e\p_x^2)p_t-(1+c\e\p_x^2)\cK_\e^{-1}\cA(D)\theta=-\e\cK_\e^{-1}(\vv V\cdot\na\cK_\e p)
+\f\e2\cK^{-1}_\e(\z\cK_\e\cA(D)\theta)+O(\e),\\
&(1-d\e\p_x^2)\theta_t+(1+c\e\p_x^2)\cK_\e^{-1}\cA(D)p=-\e\cK_\e^{-1}(\vv V\cdot\na\cK_\e\theta)-\f\e2\cK^{-1}_\e\cA(D)\bigl(\z\cK_\e p\bigr)\\
&\qquad\qquad\qquad
+\f\e2\cK_\e^{-1}\cA(D)\bigl[\z\cA^{-2}(D)\cK_\e(1-\cK_\e)\p_x^2p\bigr]+O(\e),
\end{aligned}\right.
\eeq
where $O(\e)$ represents the lower order quadratic terms of order $O(\e)$. 
\begin{remark}
Due to the properties of $\cK_\e(\xi_1)$ and $\cA(\xi)$ in \eqref{operator 2} and \eqref{operator 3}, the "undesirable" term $\f\e2\cK_\e^{-1}\cA(D)\bigl[\z\cA^{-2}(D)\cK_\e(1-\cK_\e)\p_x^2p\bigr]$ is of order $O(\e^{\f12})$ if treated as a lower order term. By neglecting this "undesirable" term, the principal component of \eqref{WTB 3} turns out to be symmetric. Subsequently, employing the standard hyperbolic energy method leads to the local existence theory on a timescale of order $O(\e^{-\f12})$ with $\bigl((1-b\e\p_x^2)^{\f12}\cK_\e p,(1-d\e\p_x^2)^{\f12}\cK_\e\theta\bigr)\in H^s(\R^2)$ for $s>1$. To improve the time scale up  to $O(1/\e)$, term $\f\e2\cK_\e^{-1}\cA(D)\bigl[\z\cA^{-2}(D)\cK_\e(1-\cK_\e)\p_x^2p\bigr]$ should be treated as a part of quasilinear terms.
\end{remark}

{\it 2). Symmetrization of \eqref{WTB 3} via good unknowns.}

In order to symmetrize both the linear and nonlinear terms of \eqref{WTB 3}, we employ a nonlinear transformation by introducing suitable unknowns. Due to the characteristics of $\mathcal{K}_\e(\xi_1)$ and $\mathcal{A}(\xi)$ as described in \eqref{operator 2} and \eqref{operator 3}, we have two options for $\mathcal{K}_\e(\xi_1)$ and $\mathcal{A}(\xi)$. To elucidate the core concept of the proof effectively, we focus solely on two specific scenarios outlined in \eqref{special case}. Within these scenarios, it holds true that $\mathcal{J}_\e(\xi_1) = \mathcal{K}_\e(\xi_1)$.

{\bf Case 1. $a=f=g=0, b=d=e=\f13,c=-\f13$.} In this case, we denote by
\beq\label{operator for case 1}
J_\e(\xi_1)\eqdefa\cK_\e(\xi_1)=1+\f\e3\xi_1^2, \quad A(\xi)\eqdefa\cA(\xi)=\bigl(J_\e(\xi_1)\xi_1^2+\xi_2^2\bigr)^{\f12}.
\eeq

With notations $J_\e(\xi_1)$ and $A(\xi)$, introducing the good unknowns
\beq\label{def of wtptheta case 1}\begin{aligned}
	\widetilde{p} \eqdefa & \,p +\f\e2 J_\e^{-1} (\zeta  J_\e p)+ \e  J_\e^{-1}(\vv V\cdot\nabla A^{-1}(D)  J_\e \theta) - \f{\e^2}{6} \frac{1+\f\e2 \z}{2+\f\e2 \z}\z  A^{-2}(D) \p_x^4 p, \\
	\widetilde{\theta} \eqdefa&\, \theta  + \f\e2 J_\e^{-1}A^{-1}(D)(\z \cdot A(D)J_\e \theta)-\e J_\e^{-1} A^{-1}(D) (\vv V\cdot \nabla J_\e p),
\end{aligned}\eeq
we obtain a symmetric system in $(\wt p,\wt\theta)$ as follows:
\begin{equation}\label{equation for tilde p and theta a}
	\left\{ \begin{aligned}
		&J_\e\wt p_t - \left(1+\f{\e}{2}\z -\f{\e^2}{6} \frac{1+\f\e2\z}{2+\f\e2\z} \z A^{-2}(D) \p_x^4  \right)A(D)\wt\theta= - \e\vv V\cdot\na\wt p+O(\e), \\
		&J_\e\wt\theta_t  + \left(1+\f{\e}{2}\z -\f{\e^2}{6} \frac{1+\f\e2\z}{2+\f\e2\z} \z A^{-2}(D) \p_x^4  \right) A(D)\wt p =-\e\vv V\cdot\na\wt\theta+O(\e).
	\end{aligned} \right.
\end{equation}
Here $(\wt{p},\wt{\theta})$ are called the \textit{good unknowns} in the sense of Alinhac (see \cite{ABZ}, \cite{AM}, \cite{Alin}).

By virtue of \eqref{new unknowns case a} and \eqref{def of wtptheta case 1}, one could check that for $s>3$
\beno\begin{aligned}
& \|J_\e^{\f12}p\|_{H^s}+\|J_\e^{\f12}\theta\|_{H^s}\sim\|(J_\e^{\f12}v_x,v_y)\|_{H^s}+\|(w_x,J^{-\f12}_\e w_y)\|_{H^s}
+\|(\z_x,J^{-\f12}_\e \z_y)\|_{H^s},\\
&\text{and}\quad \|J_\e^{\f12}(\wt p-p)\|_{H^s}+\|J_\e^{\f12}(\wt\theta-\theta)\|_{\dH^s}=O(\e),
\end{aligned}\eeno
provided that $\sqrt\e(\|J^{\f12}_\e\z\|_{H^s}+\|\vv V\|_{H^s})\leq1$ over the lifespan of  solutions to \eqref{WTB 2}.

{\bf Case 2. $a=f=g=-\f16, b=d=e=\f12,c=-\f12$.} In this case, we denote by
\beq\label{operator for case 2}
J_\e(\xi_1)\eqdefa1+\f\e2\xi_1^2,\quad K_\e(\xi_1)\eqdefa\cK_\e(\xi_1)=\f{1+\f\e2\xi_1^2}{1+\f\e6\xi_1^2}, \quad B(\xi)\eqdefa\cA(\xi)=\bigl(K_\e(\xi_1)\xi_1^2+\xi_2^2\bigr)^{\f12}.
\eeq

Using the  notations $J_\e(\xi_1)$, $K_\e(\xi_1)$ and $B(\xi)$, and introducing  the good unknowns
\beq\label{def of wtptheta case 2}\begin{aligned}
	\wt{p} \eqdefa \,&p + \e J_\e^{-1} \bigl(\vv V\cdot \nabla B^{-1}(D)K_\e \theta + \f12 \z \cdot K_\e p  \bigr) - \frac{\e^2}{6} \gamma_\e(\z,D_x) Y_\e^{-2} B^{-2}(D)  \p_x^4 p,\\
	\wt\theta \eqdefa \,& \theta - \e J_\e^{-1} B^{-1}(D) \bigl( \vv V\cdot\na K_\e p - \f12 \z \cdot B(D)K_\e\theta  \bigr),
\end{aligned}\eeq
we obtain a symmetric system for $(\wt p,\wt\theta)$ as follows:
\beq\label{equations for wt p theta 2 a}
\left\{\begin{aligned}
&J_\e\wt p_t-\Bigl(Y_\e B(D)+\f\e2\z B(D)
-\frac{\e^2}{6}\gamma_\e(\z,D_x) Y_\e^{-1} B^{-1}(D)\p_x^4\Bigr)\wt\theta=-\e\vv V\cdot\nabla\wt p+O(\e),\\
&J_\e\wt\theta_t+\Bigl(Y_\e B(D)+\f\e2\z B(D)
-\frac{\e^2}{6}\gamma_\e(\z,D_x) Y_\e^{-1} B^{-1}(D)\p_x^4\Bigr)\wt p=-\e\vv V\cdot\nabla\wt\theta+O(\e).
\end{aligned}\right.\eeq
Here $Y_\e=Y_\e(D_x)=1-\f\e6\p_x^2$ and $\gamma_\e(\z, D_x)$ is a linear operator defined by
\beno
\gamma_\e(\z,D_x) \eqdefa( 1+ \f\e2\z Y_\e^{-1}) ( 2+ \f\e2\z Y_{\e}^{-1})^{-1} (\z\cdot).
\eeno
In particular, $\gamma_\e(\z, D_x)$ is self-adjoint on $L^2(\R^2)$.
One could refer to subsection 4.2 for the detailed definition and properties of $\gamma_\e(\z,D_x)$. 

By virtue of \eqref{new unknowns case a} and \eqref{def of wtptheta case 1}, one could check that for $s>3$
\beno\begin{aligned}
& \|J_\e^{\f12}p\|_{H^s}+\|J_\e^{\f12}\theta\|_{H^s}\sim\|J_\e^{\f12}\na\vv V\|_{H^s}+\|J_\e^{\f12}\na\z\|_{H^s},\\
&\text{and}\quad \|J_\e^{\f12}(\wt p-p)\|_{H^s}+\|J_\e^{\f12}(\wt\theta-\theta)\|_{\dH^s}=O(\e),
\end{aligned}\eeno
provided that $\sqrt\e(\|J^{\f12}_\e\z\|_{H^{s+1}}+\|\vv V\|_{H^{s+1}})\leq\f{1}{C_s}$ over the lifespan of  solutions to \eqref{WTB 2}.

Hence, in order to establish a uniform energy estimate on \eqref{WTB 2} for both Case 1 and Case 2, we divide the estimate into two parts: i). the lower order energy estimate in term of $(v,w,\z)$; ii). the highest order energy estimate in term of $(\wt p,\wt\theta)$. For a comprehensive understanding of the proof, detailed explanations are provided in Section 3 and Section 4.

\begin{remark} 
\begin{itemize}
\item[(1).]It is worth to emphasize that $J_\e^{\f12}\wt p\in H^s(\R^2)$ and $J_\e^{\f12}\wt\theta\in \dH^s(\R^2)$, whereas $J_\e^{\f12}\wt\theta\notin H^s(\R^2)$. Indeed, in the construction of $\wt\theta$ as stated in \eqref{def of wtptheta case 1} and \eqref{def of wtptheta case 2}, there exist correction terms in the form of $A^{-1}(D)(fg)$ and $B^{-1}(D)(fg)$. Notably, product estimates like $\|A^{-1}(D)(fg)\|_{L^2}$ and $\|B^{-1}(D)(fg)\|_{L^2}$ look closely to $\||D|^{-1}(fg)\|_{L^2}=\|fg\|_{\dH^{-1}}$. This similarity arises because $\dH^{-1}(\R^2)$ is a critical space in which the product estimate encounters limitations.
\item[(2).] For other cases in \eqref{case a}, the creation of good unknowns $(\wt p,\wt\theta)$ mimics the approach taken in Case 1 and Case 2, and the symmetric system for $(\wt p,\wt\theta)$ can  similarly be deduced from \eqref{WTB 3}.
\end{itemize}
\end{remark}

{\it 3). Anisotropic operators.} The structure of the linearization of system \eqref{WTB 2} is anisotropic due to the dispersion occurring solely in $x$-direction. This characteristic is obvious when examining the eigenvalues provided
 in \eqref{eigen value a}. In looking for the symmetric formulation for system \eqref{WTB 2}, numerous terms emerge that entail commutators and product estimates involving anisotropic operators such as $J_\e(D_x)$, $K_\e(D_x)$, $A(D)$ and $B(D)$. 

 As anisotropic operators, the commutator estimates involving $[J_\e(D_x),f]g$, $[J_\e^{-1}(D_x),f]g$ and $[K_\e(D_x),f]g$ lead to a reduction of one order of regularity for $g$ in the $x$-direction. Similarly, the estimate for $J_\e^{-1}(D_x)(fg)$ bears semblance to the Leibniz rule applied to $J_\e^{-1}(D_x)$.  On the other hand, commutator estimates incorporating $A(D)$, $B(D)$, $A^{-1}(D)$ and $B^{-1}(D)$ result in a reduction of one order of regularity for functions in specific spaces. For a comprehensive understanding of these properties, readers can consult the detailed discussion provided in Section 2.

Due to the properties of such anisotropic operators, it is imperative to approach nonlinear terms and energy estimates with meticulous care. 

{\it 4). The operator $\gamma_\e(\z,D_x)$ for Case 2.} Searching good unknowns for Case 2  poses a significantly greater challenge compared to Case 1 and is far from straightforward. The construction of nonlinear and nonlocal anisotropic operator $\gamma_\e(\z,D_x)$ is crucial. This operator is designed by using  von Neumann theorem within specialized spaces and must also be self-adjoint in the space  $L^2(\R^2)$. With $\gamma_\e(\z,D_x)$ available, we can construct the good unknowns $(\wt p,\wt\theta)$ that satisfy a symmetric system. Throughout this process, numerous terms arise involving operations with operators such as $\gamma_\e(\z,D_x)$,  $J_\e(D_x)$, $J^{-1}_\e(D_x)$, $B(D)$ and $B^{-1}(D)$ and others. Therefore, a thorough examination of the properties of $\gamma_\e(\z,D_x)$ is crucial in the proof.  For further insights, readers are directed to Subsection 4.2.

\setcounter{equation}{0}
\section{Preliminary}
\subsection{Definitions and notations}

The notation $f\sim g$ means that there exists a constant $C$ such that $\f{1}{C}f\leq g\leq Cf$.  Notations $f\lesssim g$ and $g\gtrsim f$ mean that there exists a constant $C$ such that $f\leq Cg$. 

If  $A, B$ are two operators, $[A,B]=AB-BA$ denotes their commutator.

For any $s\in\R$, $H^s(\R^2)$ and $\dH^s(\R^2)$ denote the classical  $L^2$ based nonhomogeneous and homogenous Sobolev spaces with the norm $\|\cdot\|_{H^s}$ and $\|\cdot\|_{\dH^s}$ respectively.
The notation $\|\cdot\|_{L^p}$ stands for  the $L^p(\R^2)$ norm for $1\leq p \leq \infty$.

The Fourier transform of a tempered distribution $u\in\mathcal{S}'(\R^2)$ is denoted by $\widehat{u}$, which is defined by
\beno
\widehat{u}(\xi)\eqdefa\mathcal{F}(u)(\xi)=\int_{\R^2}e^{-ix\cdot\xi}u(x)dx.
\eeno
We use $\mathcal{F}^{-1}(f)$ to denote the inverse Fourier  transform of $f(\xi)$.

If $f$ and $u$ are two functions defined on $\R^2$, the  Fourier multiplier  $f(D)u$  is defined in term of Fourier transform, i.e.,
\beno
\widehat{f(D)u}(\xi)=f(\xi)\widehat{u}(\xi),
\eeno
where $D=(D_x,D_y)$ with $D_x=\f{1}{i}\p_x$ and $D_y=\f{1}{i}\p_y$.

We shall use notations
\beno
\langle\xi\rangle=\bigl(1+|\xi|^2\bigr)^{\f12},\quad\langle D\rangle=\bigl(1+|D|^2\bigr)^{\f12}.
\eeno

The $L^2(\R^2)$ inner product is denoted by $(f|g)_{L^2}\eqdefa\int_{\R^2}f(x)\overline{g}(x)dx$. And for any $s\in\R$, we shall also use the following notations
\beno
(f|g)_{H^s}\eqdefa(\langle D\rangle^sf\,|\,\langle D\rangle^s g)_{L^2},\quad\forall f,g\in H^s(\R^2),
\eeno
\beno
(f|g)_{\dH^s}\eqdefa(|D|^sf\,|\,|D|^s g)_{L^2},\quad\forall f,g\in\dH^s(\R^2).
\eeno

\subsection{Technical lemmas}
In this subsection, we  present several technical lemmas involving  product and commutator estimates. Firstly, we recall the classical tame product and commutator estimates in Sobolev spaces: 
\begin{enumerate}
\item[(1).]  if $s\geq0$, one has for any $f,g\in H^s\cap L^\infty(\R^2)$,
\beq\label{tame}
	\norm{f\cdot g}{H^s} \lesssim \norm{f}{H^s} \norm{g}{L^\infty}+\norm{f}{L^\infty}\norm{g}{H^s} ;
\eeq
\item[(2).] if $t_0>1$, $s\geq t_0+1$, one has for any $f\in H^s(\R^2),\, g\in H^{s-1}(\R^2)$,
\beq\label{commutator D}
\|[|D|^s,\, f]g\|_{L^2}\lesssim\|\na f\|_{H^{t_0}}\|g\|_{\dH^{s-1}}+\|f\|_{\dH^s}\|g\|_{H^{t_0}}\lesssim\|f\|_{H^s}\|g\|_{H^{s-1}}.
\eeq
\end{enumerate}
One can also verify \eqref{tame} and \eqref{commutator D} via  Fourier analysis. We omit details here.

Our first technical lemma concerns  the product and commutator estimates involving the following operators: 
\beno
J_\e=J_\e(D_x)=1-\sigma_1\e\p_x^2,\quad K_\e=K_\e(D_x)=(1-\sigma_2\e\p_x^2)(1-\sigma_3\e\p_x^2)^{-1},
\eeno
with $\sigma_1,\,\sigma_2,\sigma_3>0, \sigma_2\neq\sigma_3.$ We shall always assume that $\e\in(0,1)$ throughout the whole article.

\begin{lemma}\label{product estimates}
Assume that $s>3$. Then for any smooth enough functions $f,g$, there hold
\begin{enumerate} 
\item[1).] (Product estimates on $J_\e^{-1}$)  
	\beq\label{product estimate 1}
	\begin{aligned}
	&\|J_\e^{-\frac{1}{2}}(J_\e^{\f12}f\cdot J_\e^{\f12} g)\|_{H^r}\lesssim\|f\|_{H^s}\cdot\| g\|_{H^r},\quad r\in[0,s],\\
	&\|J_\e^{-\frac{3}{2}}(J_\e^{\f32}f\cdot J_\e^{-\f12} g)\|_{H^s}
	+\|J_\e^{-1}(J_\e^{\f12}f\cdot J_\e g)\|_{H^s}\lesssim\|f\|_{H^s}\cdot\| g\|_{H^s};
	\end{aligned}\eeq
\item[2).] (Commutator estimate on $J_\e^{-1}$) 
\beq\label{commutator 1a}
\Bigl\|J_\e^{-\frac{1}{2}}\Bigl([J_\e^{-1},f] J_\e^{\f32} g\Bigr)\Bigr\|_{H^s}\lesssim\sqrt\e\|f_x\|_{H^s}\cdot\| g\|_{H^s};
\eeq
\item[3).] (Commutator estimate on $|D|^s$) 
\beq\label{commutator D1}
\|J_\e^{-\f12}([|D|^s,\, f]J_\e^{\f12}g)\|_{L^2}\lesssim\|J^{\f12}_\e f\|_{H^s}\|g\|_{H^{s-1}};
\eeq
\item[4).] (Commutator estimate on $K_\e$) 
\beq\label{commutator K}
\|[K_\e, f]\p_x g\|_{H^s}\lesssim\|f_x\|_{H^s}\|g\|_{H^s}.
\eeq
\end{enumerate}
\end{lemma}

\begin{proof}
1). Firstly, it is easy to check that
\beno\begin{aligned}
\mathcal{F}\Bigl[J_\e^{-\f12}(J_\e^{\f12}f\cdot J_\e^{\f12} g)\Bigr](\xi)&=\f{1}{4\pi^2}\int_{\R^2}\underbrace{(1+\sigma_1\e\xi_1^2)^{-\f12}(1+\sigma_1\e\eta_1^2)^{\f12}\bigl(1+\sigma_1\e|\xi_1-\eta_1|^2\bigr)^{\f12}}_{q_1(\xi_1,\eta_1)}\hat{f}(\xi-\eta)\hat{g}(\eta)d\eta.
\end{aligned}\eeno

For $q_1(\xi_1,\eta_1)$, there holds
\beno
q_1(\xi_1,\eta_1)\lesssim
\left\{\begin{aligned}
&\bigl(1+\sigma_1\e|\xi_1-\eta_1|^2\bigr)^{\f12},\quad\text{for}\quad |\xi_1-\eta_1|<\f14|\eta_1|,\\
&(1+\sigma_1\e\eta_1^2)^{\f12},\quad\text{for}\quad |\eta_1|<\f14|\xi_1-\eta_1|,\\
&(1+\sigma_1\e\eta_1^2)\sim(1+\sigma_1\e|\xi_1-\eta_1|^2),\quad\text{for}\quad \f14|\eta_1|\leq|\xi_1-\eta_1|\leq4|\eta_1|,
\end{aligned}\right.
\eeno
which implies that
\beno
q_1(\xi_1,\eta_1)\lesssim 1+\sigma_1\e\min\{|\eta_1|^2,|\xi_1-\eta_1|^2\},
\eeno
\beno
\text{and}\quad\langle\xi\rangle^s q_1(\xi_1,\eta_1)\lesssim
\langle\eta\rangle^s\bigl(1+\sigma_1\e|\xi_1-\eta_1|^2\bigr)
+\langle\xi-\eta\rangle^s(1+\sigma_1\e\eta_1^2).
\eeno

Then we obtain
\beno\begin{aligned}
\|J_\e^{-\frac{1}{2}} ( J_\e^{\frac{1}{2}} f\cdot  J_\e^{\frac{1}{2}}g)\|_{H^s}\lesssim&\Bigl\|\int_{\R^2}\bigl(1+\sigma_1\e|\xi_1-\eta_1|^2\bigr)|\hat{f}(\xi-\eta)|\cdot \langle\eta\rangle^s|\hat{g}(\eta)|d\eta\Bigr\|_{L^2}\\
&+\Bigl\|\int_{\R^2}\langle\xi-\eta\rangle^s|\hat{f}(\xi-\eta)|\cdot (1+\sigma_1\e\eta_1^2)|\hat{g}(\eta)|d\eta\Bigr\|_{L^2},
\end{aligned}\eeno
which yields to
\beno\begin{aligned}
\|J_\e^{-\frac{1}{2}} ( J_\e^{\frac{1}{2}} f\cdot  J_\e^{\frac{1}{2}}g)\|_{H^s}&\lesssim\|\widehat{J_\e f}(\xi)\|_{L^1}\cdot\|\langle\eta\rangle^s\hat{g}(\eta)\|_{L^2}+\|\langle\eta\rangle^s\hat{f}(\eta)\|_{L^2}\cdot\|\widehat{J_\e g}(\xi)\|_{L^1}\\
&\lesssim \|\widehat{J_\e f}(\xi)\|_{L^1}\cdot\|g\|_{H^s}+\|f\|_{H^s}\cdot\|\widehat{J_\e g}(\xi)\|_{L^1}.
\end{aligned}\eeno

Since for $t_0>1$,
\beno
\|\widehat{J_\e f}(\xi)\|_{L^1}\leq\|(1+|\xi|^2)^{-\f{t_0}{2}}\|_{L^2}\cdot
\|(1+|\xi|^2)^{\f{t_0}{2}}\widehat{J_\e f}(\xi)\|_{L^2}\lesssim\|J_\e f\|_{H^{t_0}},
\eeno
we obtain 
\beq\label{Product 1}
\|J_\e^{-\frac{1}{2}} ( J_\e^{\frac{1}{2}} f\cdot  J_\e^{\frac{1}{2}}g)\|_{H^s}\lesssim\|J_\e f\|_{H^{t_0}}\cdot\|g\|_{H^s}+\|f\|_{H^s}\cdot\|J_\e g\|_{H^{t_0}}.
\eeq
By taking $t_0=s-2>1$, we obtain the first inequality of  \eqref{product estimate 1} for $r=s$.

Moreover, since 
\beno
q_1(\xi_1,\eta_1)\lesssim1+\sigma_1\e|\xi_1-\eta_1|^2,
\eeno
it is easy to check that
\beq\label{Product 2}
\|J_\e^{-\frac{1}{2}} ( J_\e^{\frac{1}{2}} f\cdot  J_\e^{\frac{1}{2}}g)\|_{L^2}\lesssim\|J_\e f\|_{H^{t_0}}\cdot\|g\|_{L^2}\lesssim\|f\|_{H^s}\cdot\|g\|_{L^2}.
\eeq
This is the first inequality of  \eqref{product estimate 1} for $r=0$. 

Due to \eqref{Product 1} and \eqref{Product 2}, we obtain the first inequality of  \eqref{product estimate 1} for any $r\in[0,s]$ by standard interpolation.


Similar derivation gives rise to the second inequality of \eqref{product estimate 1}.

\smallskip

2). To prove \eqref{commutator 1a}, we first have
\beno\begin{aligned}
&\quad\mathcal{F}\Bigl[J_\e^{-\frac{1}{2}}\Bigl([J_\e^{-1},f] J_\e^{\f32} g\Bigr)\Bigr](\xi)\\
&=\f{1}{4\pi^2}\int_{\R^2}\underbrace{(1+\sigma_1\e\xi_1^2)^{-\f12}\bigl[(1+\sigma_1\e\xi_1^2)^{-1}(1+\sigma_1\e\eta_1^2)-1\bigr](1+\sigma_1\e\eta_1^2)^{\f12}}_{q_2(\xi_1,\eta_1)}\hat{f}(\xi-\eta)\hat{g}(\eta)d\eta.
\end{aligned}\eeno

For $q_2(\xi_1,\eta_1)$, it is easy to check that
\beno
q_2(\xi_1,\eta_1)\lesssim
\sqrt\e|\xi_1-\eta_1|\bigl(1+\sigma_1\e\min\{|\eta_1|^2,|\xi_1-\eta_1|^2\}\bigr),
\eeno
which gives rise to 
\beno
\Bigl\|J_\e^{-\frac{1}{2}}\Bigl([J_\e^{-1},f] J_\e^{\f32} g\Bigr)\Bigr\|_{H^s}
\lesssim\sqrt\e\|J_\e f_x\|_{H^{t_0}}\|g\|_{H^s}+\sqrt\e\|f_x\|_{H^s}\|J_\e g\|_{H^{t_0}}.
\eeno
Then we obtain \eqref{commutator 1a} by taking $t_0=s-2>1$.

\smallskip

3). Following similar arguments as 1) and 2), we arrive at
 \beno
\bigl\|J_\e^{-\f12}\bigl([|D|^s,f] J_\e^{\f12} g\bigr)\bigr\|_{L^2}
\lesssim\|J^{\f12}_\e\na f\|_{H^{t_0}}\|g\|_{\dH^{s-1}}+\|J^{\f12}_\e f\|_{\dH^s}\|g\|_{H^{t_0}}\lesssim\|J^{\f12}_\e f\|_{H^s}\|g\|_{H^{s-1}},
\eeno
where we took $t_0=s-2>1$ in the last inequality. This is exactly \eqref{commutator D1}.

\smallskip

4). For $K_\e(D_x)$, it is easy to get 
\beno
|K_\e(\xi_1)-K_\e(\eta_1)|=\f{|\sigma_2-\sigma_3|\cdot\e|\xi_1+\eta_1|\cdot|\xi_1-\eta_1|}{(1+\sigma_3\e\xi_1^2)(1+\sigma_3\e\eta_1^2)}
\lesssim|\xi_1-\eta_1|\cdot\sqrt\e(1+\sigma_3\e\eta_1^2)^{-\f12}.
\eeno
Then following similar arguments as in previous steps, we obtain \eqref{commutator K}. Thus, the proof of the lemma is complete.
\end{proof}

The second technical lemma is about the commutator and product estimates involving the operators $A(D)$ and $B(D)$, where
\beno
A(\xi)\eqdefa\bigl(J_\e(\xi_1)\xi_1^2+\xi_2^2\bigr)^{\f12},\quad B(\xi)\eqdefa\bigl(K_\e(\xi_1)\xi_1^2+\xi_2^2\bigr)^{\f12}.
\eeno

\begin{lemma}\label{commutator lem}
Assume that $s>3$. Then for any smooth enough functions $f,g$, there hold
\begin{itemize}
\item[1).] (Estimates on $A(D)$) 
\begin{align}
&\|J_\e^{-\f12}\bigl([A(D),f]g\bigr)\|_{H^s}
\lesssim\bigl(\|f\|_{H^s}+\|J^{-\f12}_\e A(D) f\|_{H^s}\bigr)\|g\|_{H^s},
\label{commutator A2}\\
&\|J_\e^{-\f12}A(D)(fg)\|_{H^s}
\lesssim\bigl(\|f\|_{H^s}+\|J^{-\f12}_\e A(D) f\|_{H^s}\bigr)
\bigl(\|g\|_{H^s}+\|J^{-\f12}_\e A(D) g\|_{H^s}\bigr),\label{product A2}\\
&
\|J_\e^{-\f12}\bigl([A(D),f]g\bigr)\|_{L^2}\lesssim\|f\|_{H^s}\|g\|_{L^2};\label{commutator A3}
\end{align}
\item[2).]  (Estimates on $A^{-1}(D)$)
\begin{align}
&\|J_\e^{-\f12}\bigl([A^{-1}(D),f]A(D)g\bigr)\|_{\dH^s}\lesssim\|f\|_{H^s}\cdot\| g\|_{H^{s-1}},\label{commutator A1}\\
&\|J_\e^{-\f12}A^{-1}(D)\bigl(fA(D)J_\e^{\f12}g\bigr)\|_{\dH^s}\lesssim\|f\|_{H^s}\cdot\|g\|_{H^s};\label{product A1}
\end{align}
\item[3).](Estimates on $B(D)$)
\begin{align}
&\|J_\e^{-\f12}\bigl([B(D),f]g\bigr)\|_{H^s}
\lesssim\|J^{-\f12}_\e\na f\|_{H^s}\|J^{-\f12}_\e g\|_{H^s},\label{commutator B1}\\
&\|[B(D),f]g\|_{H^{r}}\lesssim\|f\|_{H^s}\| g\|_{H^{r}},\quad r\in [0,s-1],\label{commutator B3a}\\
&\|J^{\f12}_\e([B(D),f]g)\|_{H^r}\lesssim\|J_\e^{\f12}f\|_{H^s}\|J^{\f12}_\e g\|_{H^r},\quad r\in[0,s-1];\label{commutator B3}
\end{align} 
\item[4).](Estimates on $B^{-1}(D)$)
\begin{align}
&\|J_\e^{-\f12}([B^{-1}(D),f]B(D)g)\|_{\dH^s}
\lesssim  \| J_\e^{-\f12} \na f \|_{H^{s-1}} \| J_\e^{-\f12} g \|_{H^{s-1}} ,\label{commutator B2}\\
&\|J_\e^{-\f12}B^{-1}(D)\bigl(fB(D)J_\e^{\f12}g\bigr)\|_{\dH^s}\lesssim\|f\|_{H^s}\cdot\|g\|_{H^s}.\label{product B2}
\end{align}
\end{itemize}
\end{lemma}
\begin{proof} 
1). For \eqref{commutator A2}, we have
\beno
\mathcal{F}\Bigl(J_\e^{-\f12}\bigl([A(D),f]g\Bigr)(\xi)
=\f{1}{4\pi^2}\int_{\R^2}\underbrace{(1+\sigma_1\e\xi_1^2)^{-\f12}\bigl(A(\xi)-A(\eta)\bigr)}_{q_3(\xi,\eta)}\hat{f}(\xi-\eta)\hat{g}(\eta)d\eta.
\eeno
Since $A(\xi) = \Bigl((1+\sigma_1\e\xi_1^2)\xi_1^2+\xi_2^2\Bigr)^{\f12}\sim(1+ \sqrt\e |\xi_1|)|\xi_1|+|\xi_2|$, we have
\beno\begin{aligned}
|A(\xi)-A(\eta)|&=\f{|A(\xi)^2-A(\eta)^2|}{A(\xi)+A(\eta)}
\lesssim(1+\sqrt\e|\xi_1+\eta_1|)\cdot|\xi_1-\eta_1|+|\xi_2-\eta_2|\\
&\lesssim(1+\sqrt\e\max\{|\xi_1-\eta_1|,|\eta_1|\})\cdot|\xi_1-\eta_1|+|\xi_2-\eta_2|
\end{aligned}\eeno
which implies
\beq\label{L0}
|q_3(\xi,\eta)|\lesssim
\left\{\begin{aligned}
&|\xi_1-\eta_1|+(1+\sigma_1\e\eta_1^2)^{-\f12}|\xi_2-\eta_2|,\quad\text{for}\quad|\xi_1-\eta_1|\leq\f14|\eta_1|,\\
&|\xi_1-\eta_1|+(1+\sigma_1\e|\xi_1-\eta_1|^2)^{-\f12}|\xi_2-\eta_2|,\quad\text{for}\quad|\xi_1-\eta_1|\geq4|\eta_1|,\\
&(1+\sigma_1\e|\eta_1|^2)^{\f12}\cdot|\xi_1-\eta_1|+|\xi_2-\eta_2|,\quad\text{for}\quad \f14|\eta_1|\leq|\xi_1-\eta_1|\leq4|\eta_1|.
\end{aligned}\right.\eeq
Then we have
\beno
\langle\xi\rangle^s|q_3(\xi,\eta)|\lesssim
A(\xi-\eta)\cdot\langle\eta\rangle^s+(1+\sigma_1\e|\xi_1-\eta_1|^2)^{-\f{1}{2}}A(\xi-\eta)\langle\xi-\eta\rangle^s
\cdot(1+\sigma_1\e\eta_1^2),
\eeno
which gives rise to
\beno\begin{aligned}
\|J_\e^{-\f12}\bigl([A(D),f]g\bigr)\|_{H^s}&\lesssim\|A(D)f\|_{H^{t_0}}\|g\|_{H^s}+\|J^{-\f12}_\e A(D)f\|_{H^s}\|J_\e g\|_{H^{t_0}}\\
&\lesssim\bigl(\|f\|_{H^s}+\|J^{-\f12}_\e A(D)f\|_{H^s}\bigr)\|g\|_{H^s},
\end{aligned}\eeno
where we took $t_0=s-2>1$ in the last inequality. This is exactly \eqref{commutator A2}.

Consequently, since 
\beno
J_\e^{-\f12}A(D)(fg)=J_\e^{-\f12}\bigl([A(D),f]g\bigr)+J_\e^{-\f12}\bigl(fA(D)g\bigr),
\eeno
using \eqref{product estimate 1} and \eqref{commutator A2}, we obtain
\beno\begin{aligned}
\|J_\e^{-\f12}A(D)(fg)\|_{H^s}&\lesssim\bigl(\|f\|_{H^s}+\|J_\e^{-\f12}A(D)f\|_{H^s}\bigr)\|g\|_{H^s}+\|J_\e^{-\f12}f\|_{H^s}\|J_\e^{-\f12}A(D)g\|_{H^s}\\
&\lesssim\bigl(\|f\|_{H^s}+\|J_\e^{-\f12}A(D)f\|_{H^s}\bigr)\bigl(\|g\|_{H^s}
+\|J_\e^{-\f12}A(D)g\|_{H^s}\bigr).
\end{aligned}\eeno
Then we obtain the estimate \eqref{product A2}. 

Moreover, due to \eqref{L0}, we have $|q_3(\xi,\eta)|\lesssim A(\xi-\eta)$ and
\beno
\|J^{-\f12}_\e\bigl([A(D),f]g\bigr)\|_{L^2}\lesssim\|A(D)f\|_{H^{t_0}}\|g\|_{L^2}
\lesssim\|f\|_{H^s}\|g\|_{L^2},
\quad\forall\,f\in H^{s}(\R^2),\, g\in L^2(\R^2),
\eeno
where we took $t_0=s-2>1$ in the last inequality.
This is \eqref{commutator A3}.

\medskip
2). For \eqref{commutator A1}, we firstly have
\beno
\mathcal{F}\Bigl(J_\e^{-\f12}\bigl([A^{-1}(D),f]A(D)g\Bigr)(\xi)
=\f{1}{4\pi^2}\int_{\R^2}\underbrace{(1+\sigma_1\e\xi_1^2)^{-\f12}A^{-1}(\xi)\bigl(A(\eta)-A(\xi)\bigr)}_{q_4(\xi,\eta)}\hat{f}(\xi-\eta)\hat{g}(\eta)d\eta.
\eeno
Since
\beno
|q_4(\xi,\eta)|=A^{-1}(\xi)|q_3(\xi,\eta)|\lesssim |\xi|^{-1}|q_3(\xi,\eta)|
\eeno
using \eqref{L0}, we have for $s>3$
\beno\begin{aligned}
|\xi|^s|q_4(\xi,\eta)|&\lesssim(|\eta|^{s-1}+|\xi-\eta|^{s-1})|q_3(\xi,\eta)|\\
&\lesssim
(1+\sigma_1\e|\xi_1-\eta_1|^2)^{\f12}|\xi-\eta|\cdot|\eta|^{s-1}+|\xi-\eta|^s\cdot(1+\sigma_1\e\eta_1^2)^{\f12},
\end{aligned}\eeno
which yields to
\beno
\|J_\e^{-\f12}\bigl([A^{-1}(D),f]A(D)g\bigr)\|_{\dH^s}
\lesssim\|J_\e^{\f12}\na f\|_{H^{t_0}}\cdot\|g\|_{\dH^{s-1}}
+\|f\|_{\dH^s}\cdot\|J_\e^{\f12} g\|_{H^{t_0}}\lesssim\|f\|_{H^s}\cdot\|g\|_{H^{s-1}},
\eeno
where we took $t_0=s-2>1$  in the last inequality. Thus, we obtain \eqref{commutator A1}.

Consequently, since
\beno
J_\e^{-\f12}A^{-1}(D)\bigl(fA(D)J_\e^{\f12}g\bigr)=J_\e^{-\f12}\bigl(fJ_\e^{\f12}g\bigr)+J_\e^{-\f12}\bigl([A^{-1}(D),f]A(D)J_\e^{\f12}g\bigr)
\eeno
using \eqref{product estimate 1} and \eqref{commutator A1}, we get that
\beno
\|J_\e^{-\f12}A^{-1}(D)\bigl(fA(D)J_\e^{\f12}g\bigr)\|_{\dH^s}
\lesssim\|J_\e^{-\f12}f\|_{H^s}\|g\|_{H^s}+\|f\|_{H^s}\cdot\|J_\e^{\f12} g\|_{H^{s-1}}
\lesssim\|f\|_{H^s}\cdot\|g\|_{H^s}.
\eeno
This is exactly \eqref{product A1}.

\medskip

3). For $B(D)$, it is easy to check that
\beno
|B(\xi)-B(\eta)|=\f{|B(\xi)^2-B(\eta)^2|}{B(\xi)+B(\eta)}\lesssim|\xi-\eta|.
\eeno
which along with the fact $s>3$ and \eqref{product estimate 1} give rises to 
\beno
\|J^{-\f12}_\e\bigl([B(D),f]g\bigr)\|_{H^s}
\lesssim\|J^{-\f12}_\e\na f\|_{H^s}\|J^{-\f12}_\e g\|_{H^s}
\eeno
This is the desired estimate \eqref{commutator B1}.


\medskip

4). Since $B(D) \sim |\xi|$ and $[B^{-1}(D),f] B(D) g = -B^{-1}(D) ( [B(D),f] g)$, we have
\begin{align*}
	\| J_\e^{-\f12} \bigl( [B^{-1}(D),f] B(D) g \bigr)  \|_{\dot{H}^s} \sim \|  J_\e^{-\f12} \bigl( [B(D),f] g  \bigr)  \|_{\dot{H}^{s-1}},  
\end{align*}
which along with \eqref{commutator B1} yields to \eqref{commutator B2}.

The proofs of \eqref{product B2}, \eqref{commutator B3a} and \eqref{commutator B3}
 are similar to the previous proofs. We omit details here.
\end{proof}

The last technical lemma is about the estimate involving the composition function.

\begin{lemma}\label{composition lemma}
Assume that $s>1$, $F(\cdot)\in C^\infty(\R)$, $F(0)=0$, $F'\in L^\infty(\R)$. Then for any $u\in H^s(\R^2)$,  we have $F\bigl(u(\cdot)\bigr)\in H^s(\R^2)$ and
\beq\label{composition}
\|F(u(\cdot))\|_{H^s}\leq C(\|F'\|_{L^\infty},\|u\|_{L^\infty})\|u\|_{H^s},
\eeq
where $C(\lambda_1,\lambda_2)>0$ is a nondecreasing constant dependent on $\lambda_1$ and $\lambda_2$. 
\end{lemma}
Lemma \ref{composition lemma} is a consequence of Theorem 2.61 in \cite{BCD}, since $H^s(\R^2)\subset\dB^{\f12}_{2,2}(\R^2)\cap\dB^s_{2,2}(\R^2)\cap L^\infty(\R^2)$ and 
$
\|F(u)\|_{L^2}=\|F(u)-F(0)\|_{L^2}\lesssim\|F'\|_{L^\infty}\|u\|_{L^2}.
$

\setcounter{equation}{0}
\section{Proof of Theorem \ref{main theorem} for case $a=f=g=0, b=d=e=\f13,c=-\f13$}
\subsection{The reduction of (\ref{WTB 2}) for case 1}
For case $a=f=g=0,b=d=e=\f13,c=-\f13$, system \eqref{WTB 2} reads 
\beq\label{WTB case 3}
\left\{\begin{aligned}
	&(1-\f\e3  \p_x^2)v_t+\z_x+\e vv_x+\e wv_y+\f\e2\z \z_x=0,\quad t>0,\, (x,y)\in\R^2,\\
	&(1-\f\e3 \e \p_x^2)w_t+\z_y+\e v w_x+\e w w_y+ \f\e2\z \z_y=0,\\
	&(1-\f{\e}{3} \p_x^2)\z_t+(1-\f\e3 \p_x^2)v_x+w_y+\e  v \z_x+\e w\z_y+ \f\e2 \z v_x + \f\e2\z w_y=0,
\end{aligned}\right.\eeq

Recall that the eigenvalues of the linearized system of \eqref{WTB case 3} is $\lambda_{1,\pm}(\xi)=\pm i\Lambda_1(\xi)$ with 
\beq\label{eigen value 1}
\Lam_1(\xi) = \left( \xi_1^2\cdot\frac{1}{1+\f\e3 \xi_1^2}+{\xi_2^2} \cdot \frac{1}{(1+\f\e3 \xi_1^2)^2} \right)^{\f12}=(1+\f\e3\xi_1^2)^{-1}\Bigl((1+\f\e3\xi_1^2)\xi_1^2+\xi_2^2\Bigr)^{\f12}.
\eeq
In this section, we shall use the notations $J_\e=J_\e(D_x)=1-\f{\e}{3}\p_x^2$, $A(D)=J_\e\Lambda_1(D)$ and 
\beq\label{def of A}
A(\xi)=\Bigl((1+\f\e3\xi_1^2)\xi_1^2+\xi_2^2\Bigr)^{\f12}
\sim(1+\sqrt\e|\xi_1|)|\xi_1|+|\xi_2|.
\eeq

To symmetrize the linear part of system \eqref{WTB case 3}, we define
\beq\label{new unknows}
p\eqdefa v_x +J_\e^{-1}w_y,\quad\theta\eqdefa \Lambda_1(D)\z.
\eeq
Due to the condition $v_y=w_x$,  we have
\beq\label{v w in term of p}
v=-A^{-2}(D)J_\e\p_x p,\quad w=-A^{-2}(D)J_\e\p_yp,\quad\z=A^{-1}(D)J_\e\theta.
\eeq
Then we get the following equivalent relation between $(v,w,\z)$ and $(p,\theta)$.

\begin{lemma}\label{lem for energy functional 1}
Assume that $(v,w,\z)$ are smooth enough functions satisfying $v_y=w_x$ and $(p,\theta)$ are defined in \eqref{new unknows}. Then there holds for any $s\geq 0$ 
\beq\label{equivalent functional 1}
\|J_\e^{\f12}p\|_{\dH^s}+\|J_\e^{\f12}\theta\|_{\dH^s}
\sim\|(J_\e^{\f12}v_x,v_y)\|_{\dH^s}+\|(w_x,J^{-\f12}_\e w_y)\|_{\dH^s}
+\|(\z_x,J^{-\f12}_\e \z_y)\|_{\dH^s}.
\eeq
\end{lemma}
\begin{proof}
Thanks to \eqref{new unknows} and \eqref{eigen value 1}, we have
\beno\begin{aligned}
&\|J_\e^{\f12}p\|_{\dH^s}\lesssim\|J_\e^{\f12}v_x\|_{\dH^s}+\|J_\e^{-\f12}w_y\|_{\dH^s}\quad\text{and}\quad\|J_\e^{\f12}\theta\|_{\dH^s}\lesssim\|\z_x\|_{\dH^s}+\|J_\e^{-\f12}\z_y\|_{\dH^s}
\end{aligned}\eeno

On the other side, due to \eqref{eigen value 1}, one has
\beq\label{eigen value 1 a}
A^2(\xi)\sim (1+\f\e3\xi_1^2)|\xi_1|^2+|\xi_2|^2\quad\text{and}\quad
(1+\f\e3\xi_1^2)^{\f12}|\xi_1|\cdot|\xi_2|\lesssim A^2(\xi),
\eeq
which along with \eqref{v w in term of p} implies
\beno\begin{aligned}
&\|J_\e^{\f12}v_x\|_{\dH^s}+\|v_y\|_{\dH^s}+\|w_x\|_{\dH^s}+\|J_\e^{-\f12}w_y\|_{\dH^s}\lesssim\|J_\e^{\f12}p\|_{\dH^s}\\
\text{and}\quad&\|\z_x\|_{\dH^s}+\|J_\e^{-\f12}\z_y\|_{\dH^s}\lesssim\|J_\e^{\f12}\theta\|_{\dH^s}.
\end{aligned}\eeno
Then we obtain \eqref{equivalent functional 1}. The lemma is proved.
\end{proof}

Now, we derive the evolution equations for $p$ and $\theta$.  Firstly, for simplicity, denoting by
\beno
\vv V=(v,w)^T,
\eeno
 we rewrite \eqref{v w in term of p} to
\beno
\vv V=-A^{-2}(D)J_\e\na p.
\eeno

\textbf{ Evolution equation of $p$.} Thanks to the first two equations of \eqref{WTB case 3}, we get
\beno\begin{aligned}
	J_\e p_t-A(D)\theta&=-\e \p_x( vv_x+wv_y+\f12  \z \z_x)-\e J_\e^{-1}\p_y( vw_x+ww_y+\f12 \z \z_y)\\
	&=-\e\bigl[\underbrace{vv_{xx}+wv_{xy}+\f12\z\z_{xx}+J_\e^{-1}(vw_{xy}+ww_{yy}+\f12\z\z_{yy})}_{I_1}\bigr]-\e N_{p,1}
\end{aligned}\eeno
where 
\beq\label{def of N p 1}
N_{p,1}\eqdefa v_xv_x+w_xv_y+\f12\z_x\z_x+J_\e^{-1}(v_yw_x+w_yw_y+\f12\z_y\z_y).
\eeq

For $I_1$, using the condition $v_y=w_x$ and the relation $(1-\f{\e}{3}\p_x^2)\p_x^2+\p_y^2=-A^2(D)$, we have
\beno\begin{aligned}
I_1&=vv_{xx}+ww_{xx}+\f12\z\z_{xx}+J_\e^{-1}(vv_{yy}+ww_{yy}+\f12\z\z_{yy})\\
&=-J_\e^{-1}\Bigl(vA^2(D)v+wA^2(D)w
+\f12\z A^2(D)\z\Bigr)\\
&\qquad-[J_\e^{-1},v]J_\e v_{xx}-[J_\e^{-1},w]J_\e w_{xx}
-[J_\e^{-1},\z]J_\e\z_{xx},
\end{aligned}\eeno
which along with \eqref{v w in term of p} implies
\beno
I_1=J_\e^{-1}\bigl(v\p_xJ_\e p+w\p_yJ_\e p-\f12\z A(D)J_\e\theta\bigr)+N_{p,2}
\eeno
where 
\beq\label{def of N p 2}
N_{p,2}\eqdefa-[J_\e^{-1},v]J_\e v_{xx}-[J_\e^{-1},w]J_\e w_{xx}
-[J_\e^{-1},\z]J_\e\z_{xx}.
\eeq

Then using notation $\vv V=(v,w)^T$, we obtain the evolution equation of $p$ as follows
\beq\label{equation for p}
J_\e p_t-A(D)\theta=-\e J_\e^{-1}\bigl(\vv V\cdot\na J_\e p\bigr)+\f\e2J_\e^{-1}\bigl(\z A(D)J_\e\theta\bigr)+\e N_p,
\eeq
where $N_p\eqdefa -(N_{p,1}+N_{p,2})$ with $N_{p,1}$ and $N_{p,2}$ being defined in \eqref{def of N p 1} and  \eqref{def of N p 2}.

\bigskip

\noindent\textbf{ Evolution equation of $\theta$.} By virtue of \eqref{new unknows}, we deduce from the third equation in  \eqref{WTB case 3} that 
\beno
	J_\e\theta_t+A(D)p=-\e \Lambda_1(D)\Bigl(\vv V\cdot\na\z+ \f12  \z\div\vv V  \Bigr).
\eeno

Due to \eqref{v w in term of p} and the relation $(1-\f{\e}{3}\p_x^2)\p_x^2+\p_y^2=-A^2(D)$, we have 
\beno
\div\vv V=-\bigl(\p_x^2+\p_y^2\bigr)A^{-2}(D)J_\e p=J_\e p-\f{\e}{3} A^{-2}(D)J_\e\p_x^4 p,
\eeno
which yields to
\beno
\Lambda_1(D)\bigl(\z\div\vv V \bigr)=\Lambda_1(D)\bigl(\z J_\e p\bigr)
-\f\e3 A(D)\bigl(\z A^{-2}(D)\p_x^4 p\bigr)+\f\e3\Lambda_1(D)\bigl([J_\e,\z]A^{-2}(D)\p_x^4 p\bigr).
\eeno

Then using \eqref{v w in term of p} again, we obtain the evolution equation for $\theta$ as follows:
\beq\label{equation for theta}\begin{aligned}
J_\e\theta_t+A(D)p=&-\e \Lambda_1(D)\bigl(\vv V\cdot\na A^{-1}(D)J_\e\theta\bigr)-\f\e2\Lambda_1(D)\bigl(\z J_\e p\bigr)\\
&\quad+\f{\e^2}{6} A(D)\bigl(\z A^{-2}(D)\p_x^4 p\bigr)+\e N_{\theta},
\end{aligned}\eeq
where
\beq\label{def of N theta}
N_{\theta}\eqdefa-\f\e6\Lambda_1(D)\bigl([J_\e,\z]A^{-2}(D)\p_x^4 p\bigr).
\eeq

Combining \eqref{equation for p} and \eqref{equation for theta}, we deduce the evolution system of $(p,\theta)$:
\begin{equation}\label{system for p theta}
	\left\{ \begin{aligned}
		J_\e p_t - A(D) \theta =\, & - \varepsilon J_\e^{-1} (\vv V\cdot \nabla J_\e p )  + \f\e2 J_\e^{-1} (\z A(D)J_\e\theta) +\e N_p, \\
		J_\e \theta_t+ A(D)p =\, & - \e \Lam_1(D) (\vv V\cdot \nabla A^{-1}(D)  J_\e\theta) - \f\e2 \Lam_1(D) (\z J_\e p) \\
		& +\f{\e^2}{6} A(D) (\z A^{-2}(D) \p_x^4 p)+\e N_\theta.
	\end{aligned} \right.
\end{equation}

\begin{remark}
	The following Lemma \ref{lem for nonlinear term 1} shows that the nonlinear terms $\e N_p$ and $\e N_\theta$ are lower order terms of order $O(\e)$. The quadratic term $\f{\e^2}{6} A(D) (\z A^{-2}(D) \p_x^4 p)$ is of order $ O(\e^{\f12})$ if we treat it as the lower order term. Then the local existence theory  on the time scale $1/\e^{\f12}$ for system \eqref{WTB case 3} follows from the classical hyperbolic energy method. To enlarge the time scale up to $O(1/\e)$, the term $\f{\e^2}{6} A(D) (\z A^{-2}(D) \p_x^4 p)$ should be treated as  part of quasilinear terms.
\end{remark}

\begin{lemma}\label{lem for nonlinear term 1}
Let $s>3$, $(\vv V,\z)=(v,w,\z)$ be a smooth enough solution to \eqref{WTB case 3} and $(p,\theta)$ be defined in \eqref{new unknows}. There holds
\beq\label{estimate for N p theta}
\|J^{-\f12}_\e N_p\|_{H^s}+\|J^{-\f12}_\e N_\theta\|_{H^s}\lesssim
\|J^{\f12}_\e p\|_{H^s}^2+\|J^{\f12}_\e\theta\|_{H^s}^2.
\eeq
\end{lemma}
\begin{proof} We derive the bounds of $N_p$ and $N_\theta$ one by one.

{\bf Step 1. Bound of $\|J^{-\f12}_\e N_p\|_{H^s}$.} Firstly, thanks to \eqref{def of N p 1} and $J_\e=1-\f\e3\p_x^2$, we have 
\beq\label{L1}
\|J^{-\f12}_\e N_{p,1}\|_{H^s}\lesssim\|J^{-\f12}_\e\bigl(v_xv_x+w_xv_y+\f12\z_x\z_x\bigr)\|_{H^s}+\|J_\e^{-\f12}(v_yw_x+w_yw_y+\f12\z_y\z_y)\|_{H^s}
\eeq

Using \eqref{product estimate 1}, we have
\beno
\|J_\e^{-\f12}(w_yw_y)\|_{H^s}\lesssim\Bigl\|J_\e^{-\f12}\Bigl(J_\e^{\f12}\bigl(J_\e^{-\f12}w_y\bigr)\cdot J_\e^{\f12}\bigl(J_\e^{-\f12}w_y\bigr)\Bigr)\Bigr\|_{H^s}\lesssim\|J_\e^{-\f12}w_y\|_{H^s}^2.
\eeno
Similar estimates hold for the other terms in the right hand side (r.h.s) of \eqref{L1}. Then we obtain
\beq\label{L2}\begin{aligned}
\|J^{-\f12}_\e N_{p,1}\|_{H^s}&\lesssim\|J_\e^{-\f12}\na v\|_{H^s}^2+\|J_\e^{-\f12}\na w\|_{H^s}^2+\|J_\e^{-\f12}\na\z\|_{H^s}^2\\
&\lesssim\|\na v\|_{H^s}^2++\|(w_x,J_\e^{-\f12}w_y)\|_{H^s}^2+\|(\z_x,J_\e^{-\f12}\z_y)\|_{H^s}^2.
\end{aligned}\eeq

For $N_{p,2}$ in \eqref{def of N p 2}, we have
\beq\label{L2a}
\|J^{-\f12}_\e N_{p,2}\|_{H^s}\lesssim\|J^{-\f12}_\e([J_\e^{-1},v]J_\e v_{xx})\|_{H^s}+\|J^{-\f12}_\e([J_\e^{-1},w]J_\e w_{xx})\|_{H^s}
+\|J^{-\f12}_\e([J_\e^{-1},\z]J_\e\z_{xx})\|_{H^s}.
\eeq
Due to \eqref{commutator 1a}, we get
\beno
\|J^{-\f12}_\e([J_\e^{-1},v]J_\e v_{xx})\|_{H^s}\lesssim\sqrt\e\|v_x\|_{H^s}\|J^{-\f12}_\e v_{xx}\|_{H^s}\lesssim\|v_x\|_{H^s}^2.
\eeno
Similar estimates hold for the last two terms in \eqref{L2a}. Then we obtain
\beq\label{L3}
\|J^{-\f12}_\e N_{p,2}\|_{H^s}\lesssim\|v_x\|_{H^s}^2+\|w_x\|_{H^s}^2+\|\z_x\|_{H^s}^2.
\eeq

Thanks to \eqref{L2} and \eqref{L3}, we obtain 
\beq\label{L3a}
\|J^{-\f12}_\e N_p\|_{H^s}\lesssim\|\na v\|_{H^s}^2++\|(w_x,J_\e^{-\f12}w_y)\|_{H^s}^2+\|(\z_x,J_\e^{-\f12}\z_y)\|_{H^s}^2.
\eeq 

\medskip

{\bf Step 2. Bound of $\|J^{-\f12}_\e N_\theta\|_{H^s}$.} Recalling the expression of $N_\theta$ in \eqref{def of N theta}, we have
\beno
N_{\theta}=-\f\e6\Lambda_1(D)\bigl([J_\e,\z]A^{-2}(D)\p_x^4 p\bigr)
=\f{\e^2}{18}\Lambda_1(D)\Bigl(\p_x^2\z\cdot A^{-2}(D)\p_x^4 p+2\p_x\z\cdot \p_xA^{-2}(D)\p_x^4 p\Bigr).
\eeno
Using the fact that (see \eqref{eigen value 1} )
\beno
\Lambda_1(\xi)\sim(1+\f\e3\xi_1^2)^{-\f12}|\xi_1|+(1+\f\e3\xi_1^2)^{-1}|\xi_2|,
\eeno
we get
\beq\label{L4}\begin{aligned}
\|J^{-\f12}_\e N_\theta\|_{H^s}&\lesssim\e^2\|J_\e^{-1}\p_x\bigl(\p_x^2\z\cdot A^{-2}(D)\p_x^4 p\bigr)\|_{H^s}+\e^2\|J_\e^{-1}\p_x\bigl(\p_x\z\cdot \p_xA^{-2}(D)\p_x^4 p\bigr)\|_{H^s}\\
&\qquad+\e^2\|J_\e^{-\f32}\p_y\bigl(\p_x^2\z\cdot A^{-2}(D)\p_x^4 p\bigr)\|_{H^s}+\e^2\|J_\e^{-\f32}\p_y\bigl(\p_x\z\cdot\p_x A^{-2}(D)\p_x^4 p\bigr)\|_{H^s}\\
&\eqdefa I_{21}+I_{22}+I_{23}+I_{24}.
\end{aligned}\eeq

For term $I_{22}$, using \eqref{product estimate 1} and \eqref{def of A}, we have
\beno\begin{aligned}
I_{22}
&\lesssim\e^{\f32}\|J_\e^{-\f12}\bigl[J_\e^{\f12}(J_\e^{-\f12}\p_x\z)\cdot J_\e^{\f12}(J_\e^{-\f12}A^{-2}(D)\p_x^5 p)\bigr]\|_{H^s}\\
&\lesssim\e^{\f32}\|J_\e^{-\f12}\p_x\z\|_{H^s}\cdot\|J_\e^{-\f12}A^{-2}(D)\p_x^5 p\|_{H^s}\lesssim\|\z_x\|_{H^s}\|p\|_{H^s}.
\end{aligned}\eeno
The same estimate holds for term $I_{21}$.

For term $I_{23}$, we have
\beno
I_{23}\lesssim\e^2\|J_\e^{-\f32}\bigl(\p_x^2\p_y\z\cdot A^{-2}(D)\p_x^4 p)\bigr)\|_{H^s}+\e^2\|J_\e^{-\f12}\bigl(\p_x^2\z\cdot \p_yA^{-2}(D)\p_x^4 p)\bigr)\|_{H^s}\eqdefa I_{231}+I_{232}.
\eeno
Using \eqref{product estimate 1} and \eqref{def of A}, we get
\beno\begin{aligned}
I_{231}&=\e^2\|J_\e^{-\f32}\bigl[J^{\f32}_\e (J^{-\f32}_\e\p_x^2\p_y\z)\cdot J^{-\f12}_\e (J^{\f12}_\e A^{-2}(D)\p_x^4 p)\bigr]\|_{H^s}\\
&\lesssim\e^2\|J^{-\f32}_\e\p_x^2\z_y\|_{H^s}\|J^{\f12}_\e A^{-2}(D)\p_x^4 p\|_{H^s}\lesssim\|J^{-\f12}_\e\z_y\|_{H^s}\|J^{\f12}_\e p\|_{H^s}.
\end{aligned}\eeno
Similarly, we obtain $I_{232}\lesssim\|\z_x\|_{H^s}\|J^{\f12}_\e p\|_{H^s}$. Then we obtain
\beno
I_{23}\lesssim\|(\z_x,J^{-\f12}_\e\z_y)\|_{H^s}\|J^{\f12}_\e p\|_{H^s}.
\eeno

For term $I_{24}$, we have
\beno
I_{24}\lesssim\e^2\|J_\e^{-1}\bigl(\p_x\p_y\z\cdot A^{-2}(D)\p_x^5 p\bigr)\|_{H^s}+\e^2\|J_\e^{-1}\bigl(\p_x\z\cdot \p_yA^{-2}(D)\p_x^5 p)\bigr)\|_{H^s},
\eeno
which along with \eqref{product estimate 1} and \eqref{def of A} shows that
\beno
|I_{24}|\lesssim \|(\z_x,J^{-\f12}_\e\z_y)\|_{H^s}\|J^{\f12}_\e p\|_{H^s}.
\eeno

Combining the estimates for $I_{21}$, $I_{22}$, $I_{23}$ and $I_{24}$, we deduced from \eqref{L4} that 
\beq\label{L4a}
\|J^{-\f12}_\e N_\theta\|_{H^s}\lesssim\|(\z_x,J^{-\f12}_\e\z_y)\|_{H^s}\|J^{\f12}_\e p\|_{H^s}\lesssim\|(\z_x,J^{-\f12}_\e\z_y)\|_{H^s}^2+\|J^{\f12}_\e p\|_{H^s}^2.
\eeq

Thanks to \eqref{equivalent functional 1}, we derive \eqref{estimate for N p theta} from \eqref{L3a} and \eqref{L4a}.
The lemma is proved.
\end{proof}

\subsection{Symmetrization of system \eqref{system for p theta}.}

In order to improve the existence time scale, we shall symmetrize system \eqref{system for p theta} (for both {\it linear} and {\it nonlinear} parts) by introducing good unknowns via suitable nonlinear transformations.

\subsubsection{Good unknowns.}
Firstly, we introduce the good unknowns $(\wt{p},\wt\theta)$ as follows: 
\begin{align}
	\widetilde{p} \eqdefa & \,p +\f\e2 J_\e^{-1} (\zeta  J_\e p)+ \e  J_\e^{-1}(\vv V\cdot\nabla A^{-1}(D)  J_\e \theta) - \f{\e^2}{6} \frac{1+\f\e2 \z}{2+\f\e2 \z}\z  A^{-2}(D) \p_x^4 p, \label{wtp in case 1}\\
	\widetilde{\theta} \eqdefa&\, \theta  + \f\e2 J_\e^{-1}A^{-1}(D)(\z \cdot A(D)J_\e \theta)-\e J_\e^{-1} A^{-1}(D) (\vv V\cdot \nabla J_\e p).\label{wttheta in case 1}
\end{align}

\begin{lemma}\label{functional 2}
Let $s>3$ and $(\vv V,\z)=(v,w,\z)$ be a smooth enough solution of \eqref{WTB case 3} over some time interval $[0,T^*]$ satisfying
\beq\label{ansatz 2}
\sup_{t\in[0,T^*]}\sqrt\e(\|J^{\f12}_\e\z\|_{H^s}+\|\vv V\|_{H^s})\leq1.
\eeq 
 Then there hold for any $t\in[0,T^*]$
\beq\label{equivalent functional 2}\begin{aligned}
&\|J^{\f12}_\e(\wt{p}-p)\|_{H^s}\lesssim\e\|J^{\f12}_\e\z\|_{H^s}\|J^{\f12}_\e p\|_{H^s}+\e\|\vv V\|_{H^s}\|J^{\f12}_\e \theta\|_{H^s},\\
&\|J^{\f12}_\e(\wt\theta-\theta)\|_{\dH^s}
\lesssim\e\|\z\|_{H^s}\|J_\e^{\f12}\theta\|_{H^s}+\e\|\vv V\|_{H^s}\|J_\e^{\f12}p\|_{H^s}.
\end{aligned}\eeq
Moreover, there holds
\beq\label{estimate for wt theta}
\|J^{\f12}_\e\wt p\|_{H^s}+\|J^{\f12}_\e\na\wt\theta\|_{H^{s-1}}\lesssim\|J_\e^{\f12} p\|_{H^s}+\|J_\e^{\f12} \theta\|_{H^s}.
\eeq
\end{lemma}
\begin{proof} We divide the proof into three steps.

{\bf Step 1. Bound of $\|J^{\f12}_\e(\wt{p}-p)\|_{H^s}$.}
Thanks to \eqref{wtp in case 1}, we obtain by using \eqref{product estimate 1} and \eqref{def of A} that
\beno\begin{aligned}
\|J^{-\f12}_\e(\z J_\e p)\|_{H^s}&\lesssim\|J^{-\f12}_\e\z\|_{H^s}\|J^{\f12}_\e p\|_{H^s}\lesssim\|\z\|_{H^s}\|J^{\f12}_\e p\|_{H^s},\\
\|J_\e^{-\f12}(\vv V\cdot\nabla A^{-1}(D)  J_\e \theta)\|_{H^s}
&\lesssim\|J^{-\f12}_\e\vv V\|_{H^s}\|J^{-\f12}_\e \nabla A^{-1}(D)J_\e\theta\|_{H^s}\lesssim\|\vv V\|_{H^s}\|J^{\f12}_\e \theta\|_{H^s}.
\end{aligned}\eeno

For the last term in \eqref{wtp in case 1}, using \eqref{def of A} and the fact that $\|J^{\f12}_\e f\|_{H^s}\sim\|f\|_{H^s}+\sqrt\e\|\p_x f\|_{H^s}$, we get
\beno\begin{aligned}
\e\Bigl\|J^{\f12}_\e\Bigl(\frac{1+\f\e2 \z}{2+\f\e2 \z}\z  A^{-2}(D) \p_x^4 p\Bigr)\Bigr\|_{H^s}
&\lesssim\Bigl\|J^{\f12}_\e\Bigl(\frac{(1+\f\e2 \z)\z}{2+\f\e2 \z}\Bigr)\Bigr\|_{H^s}\cdot\e\|J^{\f12}_\e A^{-2}(D) \p_x^4 p\|_{H^s}\\
&
\lesssim\Bigl\|J^{\f12}_\e\Bigl(\frac{(1+\f\e2 \z)\z}{2+\f\e2 \z}\Bigr)\Bigr\|_{H^s}\cdot\|J^{\f12}_\e p\|_{H^s}.
\end{aligned}\eeno

Using Lemma \ref{composition lemma} with $F(z)=\f{(1+\f{z}{2})z}{2+\f{z}{2}}$, we obtain 
\beq\label{L8}
\Bigl\|\frac{(1+\f\e2 \z)\z}{2+\f\e2 \z}\Bigr\|_{H^s}\leq C(\e\|\z\|_{L^\infty})\|\z\|_{H^s}\lesssim\|\z\|_{H^s},
\eeq
where we used the fact that $\|\z\|_{L^\infty}\lesssim\|\z\|_{H^s}$ and the ansatz \eqref{ansatz 2} in the last inequality.
Since $\frac{1+\f\e2 \z}{2+\f\e2 \z}\z=\f\z2+\frac{\f\e4|\z|^2}{2+\f\e2 \z}$ and
\beno
\p_x(\frac{(1+\f\e2 \z)\z}{2+\f\e2 \z})=\f12\p_x\z+\frac{\f\e2\z}{2+\f\e2\z}\cdot\p_x\z-\frac{\f{\e^2}{8}|\z|^2}{(2+\f\e2\z)^2}\cdot\p_x\z,
\eeno
 we get by using the product estimates \eqref{tame}, Lemma \ref{composition lemma} and \eqref{ansatz 2} that 
\beq\begin{aligned}\label{L7a}
\Bigl\|\p_x(\frac{(1+\f\e2 \z)\z}{2+\f\e2 \z})\Bigr\|_{H^s}
&\lesssim\|\z_x\|_{H^s}+\Bigl\|\frac{\f\e2\z}{2+\f\e2\z}\Bigr\|_{H^s}\cdot\|\z_x\|_{H^s}
+\Bigl\|\frac{(\e\z)^2}{(2+\f\e2\z)^2}\Bigr\|_{H^s}\cdot\|\z_x\|_{H^s}\\
&\lesssim C(\e\|\z\|_{H^s})(1+\e\|\z\|_{H^s})\|\z_x\|_{H^s}\lesssim\|\z_x\|_{H^s},
\end{aligned}\eeq
which along with \eqref{L8} implies 
\beq\label{L7}
\Bigl\|J^{\f12}_\e\Bigl(\frac{(1+\f\e2 \z)\z}{2+\f\e2 \z}\Bigr)\Bigr\|_{H^s}
\lesssim\|J^{\f12}_\e\z\|_{H^s}.
\eeq
Then we obtain
\beno
\e\Bigl\|J^{\f12}_\e\Bigl(\frac{1+\f\e2 \z}{2+\f\e2 \z}\z  A^{-2}(D) \p_x^4 p\Bigr)\Bigr\|_{H^s}\lesssim\|J^{\f12}_\e\z\|_{H^s}\|J^{\f12}_\e p\|_{H^s}.
\eeno

Combining the above estimates, we finally obtain 
\beq\label{L5}
\|J^{\f12}_\e(\wt{p}-p)\|_{H^s}\lesssim\e\bigl(\|J^{\f12}_\e\z\|_{H^s}\|J^{\f12}_\e p\|_{H^s}+\|\vv V\|_{H^s}\|J^{\f12}_\e \theta\|_{H^s}\bigr).
\eeq
This is the first inequality of \eqref{equivalent functional 2}.

\medskip

{\bf Step 2. Bound of $\|J^{\f12}_\e(\wt\theta-\theta)\|_{\dH^s}$.} Due to \eqref{wttheta in case 1},  we have
\beno
\|J^{\f12}_\e(\wt\theta-\theta)\|_{\dH^s}
\lesssim\e\|J_\e^{-\f12}A^{-1}(D)(\z \cdot A(D)J_\e \theta)\|_{\dH^s}+\e\|J_\e^{-\f12} A^{-1}(D) (\vv V\cdot A(D)A^{-1}(D)\nabla J_\e p)\|_{\dH^s},
\eeno
which along with \eqref{product A1} and \eqref{def of A} implies 
\beq\label{L6}\begin{aligned}
\|J^{\f12}_\e(\wt\theta-\theta)\|_{\dH^s}
&\lesssim\e\|\z\|_{H^s}\|J_\e^{\f12}\theta\|_{H^s}+\e\|\vv V\|_{H^s}\cdot \|A^{-1}(D)\nabla J^{\f12}_\e p)\|_{H^s}\\
&\lesssim\e\|\z\|_{H^s}\|J_\e^{\f12}\theta\|_{H^s}+\e\|\vv V\|_{H^s}\|J_\e^{\f12}p\|_{H^s}
\end{aligned}\eeq
This is the second inequality of \eqref{equivalent functional 2}. 

\medskip

{\bf Step 3. Proof of \eqref{estimate for wt theta}.} Using \eqref{wttheta in case 1} and the fact that $A(\xi)\geq |\xi|$, we get
\beno\begin{aligned}
\|J^{\f12}_\e\na\wt\theta\|_{L^2}&\lesssim\|J^{\f12}_\e\na\theta\|_{L^2}
+\e\|J_\e^{-\f12}\na A^{-1}(D)(\z \cdot A(D)J_\e \theta)\|_{L^2}+\e\|J_\e^{-\f12}\na A^{-1}(D) (\vv V\cdot \nabla J_\e p)\|_{L^2}\\
&\lesssim\|J^{\f12}_\e\na\theta\|_{L^2}
+\e\|\z \cdot A(D)J_\e \theta\|_{L^2}+\e\|\vv V\cdot \nabla J_\e p\|_{L^2}\\
&\lesssim\|J^{\f12}_\e\na\theta\|_{L^2}
+\e\|\z\|_{L^\infty} \cdot \|A(D)J_\e \theta\|_{L^2}+\e\|\vv V\|_{L^\infty}\cdot\| \nabla J_\e p\|_{L^2}
\end{aligned}\eeno
from which, the ansztz \eqref{ansatz 2} and the fact that $s>3$, we deduce that
\beq\label{L6a}
\|J^{\f12}_\e\na\wt\theta\|_{L^2}
\lesssim\|J^{\f12}_\e\na\theta\|_{L^2}
+\e\|\z\|_{H^s} \cdot \|J_\e^{\f12} \theta\|_{H^s}+\e\|\vv V\|_{H^s}\cdot\|J_\e^{\f12} p\|_{H^s}\lesssim\|J_\e^{\f12} \theta\|_{H^s}+\|J_\e^{\f12} p\|_{H^s}.
\eeq
 
 Combining \eqref{L6a} and \eqref{equivalent functional 2}, we obtain \eqref{estimate for wt theta}. The proof is completed.
\end{proof}

\subsubsection{Symmetrization of \eqref{system for p theta}.} Now, we derive the evolution equations for $\wt p$ and $\wt\theta$. With $(\widetilde{p},\widetilde{\theta})$, we deduce from \eqref{system for p theta} that
\beq\label{E 1}
		J_\e p_t = A(D) \wt\theta + \e N_p ,
\eeq
and
\beq\label{E 1a}\begin{aligned}	
J_\e \theta_t&=-A(D)\wt p+\f{\e^2}{6}A(D)\Bigl(\frac{\z}{2+\f\e2 \z} A^{-2}(D) \p_x^4 p\Bigr)+\e N_\theta.
	\end{aligned}
\eeq

{\bf Evolution equation for $\wt p$.} By virtue of \eqref{wtp in case 1}, one has
\beq\label{E 2}
\begin{aligned}
J_\e \wt{p}_t=&\underbrace{(1+\f\e2 \zeta)J_\e p_t+\e\vv V\cdot\nabla A^{-1}(D)  J_\e \theta_t - \f{\e^2}{6} J_\e\Bigl(\frac{1+\f\e2 \z}{2+\f\e2 \z} \z  A^{-2}(D) \p_x^4 p_t\Bigr)}_{I_2}+\e N_{\wt p,1},
\end{aligned}\eeq
where 
\beq\label{def of N wtp 1}
N_{\wt p,1}\eqdefa \f12 \zeta_t  J_\e p+\vv V_t\cdot\nabla A^{-1}(D)  J_\e \theta - \f{\e}{6} J_\e\Bigl(\p_t\bigl(\frac{1+\f\e2 \z}{2+\f\e2 \z} \z \bigr)A^{-2}(D) \p_x^4 p\Bigr).
\eeq

For $I_2$, using \eqref{E 1} and \eqref{E 1a}, we have
\beno
I_2=(1+\f\e2\z)A(D)\wt\theta-\e\vv V\cdot\nabla\wt p
- \f{\e^2}{6} J_\e\Bigl(\frac{1+\f\e2 \z}{2+\f\e2 \z} \z  A^{-1}(D)J_\e^{-1} \p_x^4\wt\theta\Bigr)+\e N_{\wt p,2},
\eeno
where 
\beq\label{def of N wtp 2}\begin{aligned}
N_{\wt p,2}&\eqdefa\f{\e^2}{6}\vv V\cdot\na\Bigl(\frac{\z}{2+\f\e2 \z} A^{-2}(D) \p_x^4 p\Bigr)+(1+\f\e2\z) N_p+\e\vv V\cdot\na A^{-1}(D)N_\theta\\
&\qquad- \f{\e^2}{6} J_\e\Bigl(\frac{1+\f\e2 \z}{2+\f\e2 \z} \z  A^{-2}(D)J_\e^{-1} \p_x^4N_p\Bigr).
\end{aligned}\eeq

Then we deduce from \eqref{E 2} that
\beq\label{eq for wt p}\begin{aligned}
J_\e \wt{p}_t-\Bigl(1+\f\e2\z-\f{\e^2}{6}\frac{1+\f\e2 \z}{2+\f\e2 \z} \z  A^{-2}(D) \p_x^4\Bigr)A(D)\wt\theta= - \e\vv V\cdot \na\wt p+\e N_{\wt{p}},
\end{aligned}\eeq
where $N_{\wt{p}}=N_{\wt p,1}+N_{\wt p,2}+N_{\wt p,3}$ and
\beq\label{def of N wtp 3}
N_{\wt p,3}=- \f{\e}{6} \left[ J_\e,\frac{1+\f\e2 \z}{2+\f\e2 \z} \z  \right] A^{-1}(D)J_\e^{-1}\p_x^4  \wt\theta=\f{\e^2}{18} \left[\p_x^2,\frac{1+\f\e2 \z}{2+\f\e2 \z} \z  \right] A^{-1}(D)J_\e^{-1}\p_x^4  \wt\theta.
\eeq

\medskip

{\bf Evolution equation for $\wt\theta$.}
Thanks to \eqref{wttheta in case 1}, one has 
\beno\begin{aligned}
J_\e \wt\theta_t=&J_\e\theta_t+ \f\e2 A^{-1}(D)(\z \cdot A(D)J_\e \theta_t)-\e  A^{-1}(D) (\vv V\cdot \nabla J_\e p_t)\\
&+\e\Bigl\{\f12 A^{-1}(D)(\z_t \cdot A(D)J_\e \theta)- A^{-1}(D) (\vv V_t\cdot \nabla J_\e p)\Bigr\}
\end{aligned}\eeno
which gives rise to 
\beq\label{E 3}
J_\e \wt\theta_t=\underbrace{(1+\f\e2\z)J_\e\theta_t-\e\vv V\cdot\na A^{-1}(D)J_\e p_t}_{I_3}+\e N_{\wt\theta,1},
\eeq
where
\beq\label{def of N wttheta 1}\begin{aligned}
N_{\wt\theta,1}=&\f12 A^{-1}(D)(\z_t \cdot A(D)J_\e \theta)- A^{-1}(D) (\vv V_t\cdot \nabla J_\e p)\\
&+\f12[A^{-1}(D),\z]A(D)J_\e \theta_t
-[A^{-1}(D),\vv V]\cdot\na J_\e p_t.
\end{aligned}\eeq

For $I_3$, using \eqref{E 1} and \eqref{E 1a}, we have
\beno
I_3
=-(1+\f\e2\z)A(D)\Bigl(\wt{p}-\f{\e^2}{6}\frac{\z}{2+\f\e2 \z} A^{-2}(D) \p_x^4\wt p  \Bigr)-\e\vv V\cdot\na\wt\theta+\e N_{\wt\theta,2},
\eeno
where 
\beq\label{def of N wttheta 2}
N_{\wt\theta,2}\eqdefa{\color{red}-}\f{\e}{6}(1+\f\e2\z)A(D)\Bigl(\frac{\z}{2+\f\e2 \z} A^{-2}(D) \p_x^4 (\wt p-p) \Bigr)+(1+\f\e2\z)N_\theta+\e\vv V\cdot\na A^{-1}(D)N_p.
\eeq

Then we deduce from \eqref{E 3} that
\beq\label{eq for wt theta}
J_\e \wt\theta_t+\Bigl(1+\f\e2\z-\f{\e^2}{6}\frac{1+\f\e2 \z}{2+\f\e2 \z} \z  A^{-2}(D) \p_x^4\Bigr)A(D)\wt p=-\e\vv V\cdot\na\wt\theta+\e N_{\wt\theta},
\eeq
where $N_{\wt\theta}=N_{\wt\theta,1}+N_{\wt\theta,2}+N_{\wt\theta,3}$ and
\beq\label{def of N wttheta 3}
N_{\wt\theta,3}=\f\e6(1+\f\e2\z)\cdot\Bigl[A(D),\frac{\z}{2+\f\e2 \z}\Bigr]A^{-2}(D) \p_x^4\wt p .
\eeq

\medskip

Combining \eqref{eq for wt p} and \eqref{eq for wt theta}, we obtain the symmetric evolution system for $(\wt{p}, \wt{\theta})$  as follows
\begin{equation}\label{equation for tilde p and theta}
	\left\{ \begin{aligned}
		&J_\e\wt p_t - \left(1+\f{\e}{2}\z -\f{\e^2}{6} \frac{1+\f\e2\z}{2+\f\e2\z} \z A^{-2}(D) \p_x^4  \right)A(D)\wt\theta= - \e\vv V\cdot\na\wt p+\e N_{\wt p}, \\
		&J_\e\wt\theta_t  + \left(1+\f{\e}{2}\z -\f{\e^2}{6} \frac{1+\f\e2\z}{2+\f\e2\z} \z A^{-2}(D) \p_x^4  \right) A(D)\wt p =-\e\vv V\cdot\na\wt\theta+\e N_{\wt\theta}.
	\end{aligned} \right.
\end{equation}

\begin{remark}
We remark that \eqref{equation for tilde p and theta} is a symmetric quasilinear system for $(\wt p,\wt\theta)$ with $N_{\wt p}$ and $N_{\wt\theta}$ being the nonlinear and lower order terms.
\end{remark}

\begin{lemma}\label{lem for nonlinear term 2}
Under the assumptions of Lemma \ref{functional 2} and 
\beq\label{ansatz 1}
\sup_{t\in[0,T^*]}\e(\|J^{\f12}_\e p\|_{H^s}+\|J^{\f12}_\e\theta\|_{H^s})\leq 1,
\eeq
there holds for any $t\in(0,T^*]$:
\beq\label{estimate for wt N p theta}
\|J^{-\f12}_\e N_{\wt{p}}\|_{\dH^s}+\|J^{-\f12}_\e N_{\wt\theta}\|_{\dH^s}\lesssim
\|\z\|_{H^s}^2+\|\vv V\|_{H^s}^2+\|J_\e^{\f12} p\|_{H^s}^2+\|J_\e^{\f12} \theta\|_{H^s}^2.
\eeq 
\end{lemma}
\begin{proof}
We divide the proof into several steps.

{\bf Step 1. Estimate of $N_{\wt{p}}$.} We derive the bounds of $N_{\wt{p},1}$, $N_{\wt{p},2}$ and $N_{\wt{p},3}$ one by one.

\smallskip

{\it Step 1.1. Estimate of $N_{\wt{p},1}$.} For the first two terms of $N_{\wt{p},1}$ in \eqref{def of N wtp 1}, using \eqref{Product 1} and the fact that $A(\xi)\geq|\xi|$, we  get
\beno
\|J^{-\f12}_\e(\zeta_t  J_\e p)\|_{\dH^s}\lesssim\|J^{-\f12}_\e\zeta_t\|_{H^s}\|J^{\f12}_\e p\|_{H^s},
\eeno
\beno
\|J^{-\f12}_\e(\vv V_t\cdot\nabla A^{-1}(D)  J_\e \theta)\|_{\dH^s}\lesssim\|J^{-\f12}_\e\vv V_t\|_{H^s}\|J^{-\f12}_\e \nabla A^{-1}(D)J_\e\theta\|_{H^s}
\lesssim
\|J^{-\f12}_\e\vv V_t\|_{H^s}\|J^{\f12}_\e\theta\|_{H^s}.
\eeno

For the last term of $N_{\wt{p},1}$, we have
\beno\begin{aligned}
\e\|J^{\f12}_\e\Bigl(\p_t\bigl(\frac{1+\f\e2 \z}{2+\f\e2 \z} \z \bigr)A^{-2}(D) \p_x^4 p\Bigr)\|_{\dH^s}
&\lesssim\|J^{\f12}_\e\p_t\bigl(\frac{1+\f\e2 \z}{2+\f\e2 \z} \z \bigr)\|_{H^s}\cdot\e\|J^{\f12}_\e A^{-2}(D) \p_x^4 p\|_{H^s}\\
&\lesssim\|J^{\f12}_\e\p_t\bigl(\frac{1+\f\e2 \z}{2+\f\e2 \z} \z \bigr)\|_{H^s}
\|J^{\f12}_\e p\|_{H^s}.
\end{aligned}\eeno
Similarly as the derivation of \eqref{L7}, we get by using the ansatz $\sqrt\e\|J^{\f12}_\e\z\|_{H^s}\leq 1$ that 
\beno
\|J^{\f12}_\e\p_t\bigl(\frac{1+\f\e2 \z}{2+\f\e2 \z} \z \bigr)\|_{H^s}
\lesssim\|J^{\f12}_\e\z_t\|_{H^s},
\eeno
which yields to
\beno
\e\|J^{\f12}_\e\Bigl(\p_t\bigl(\frac{1+\f\e2 \z}{2+\f\e2 \z} \z \bigr)A^{-2}(D) \p_x^4 p\Bigr)\|_{\dH^s}\lesssim\|J^{\f12}_\e\z_t\|_{H^s}\|J^{\f12}_\e p\|_{H^s}.
\eeno

Then we obtain
\beq\label{estimate for N wtp 1}
\|J^{-\f12}_\e N_{\wt{p},1}\|_{\dH^s}\lesssim
\|J^{\f12}_\e\z_t\|_{H^s}\|J^{\f12}_\e p\|_{H^s}+\|J^{-\f12}_\e\vv V_t\|_{H^s}\|J^{\f12}_\e\theta\|_{H^s}.
\eeq

\smallskip

{\it Step 1.2. Estimate of $N_{\wt{p},2}$.} We estimate the terms of $N_{\wt{p},2}$
in \eqref{def of N wtp 2} one by one.

{\bf Estimate of the first term.}  For the first term, we get by using \eqref{product estimate 1} that
\beq\label{L9}\begin{aligned}
\e^2\Bigl\|J^{-\f12}_\e&\Bigl(\vv V\cdot\na\bigl(\frac{\z}{2+\f\e2 \z} A^{-2}(D) \p_x^4 p\bigr)\Bigr)\Bigr\|_{\dH^s}
\lesssim\e^2\|J^{-\f12}_\e\vv V\|_{H^s}\cdot\bigl\|J^{-\f12}_\e\bigl(\frac{\z}{2+\f\e2 \z} A^{-2}(D) \p_x^4 p\bigr)\bigr\|_{H^{s+1}}\\
&\lesssim\sqrt\e\|\vv V\|_{H^s}\cdot\bigl\|J^{-\f12}_\e\bigl(\frac{\z}{2+\f\e2 \z}\bigr)\bigr\|_{H^{s+1}}\cdot\e^{\f32}\|J^{-\f12}_\e A^{-2}(D) \p_x^4 p\|_{H^{s+1}}.
\end{aligned}\eeq

Since $\|f\|_{H^{s+1}}\sim\|f\|_{L^2}+\|f\|_{\dH^{s+1}}$, $A^2(\xi)=J_\e(\xi_1)\xi_1^2+\xi_2^2\geq|\xi|^2$ and $J_\e(\xi_1)=1+\f\e3\xi_1^2$, we have
\beq\label{L10c}
\e^{\f32}\|J^{-\f12}_\e A^{-2}(D) \p_x^4 p\|_{H^{s+1}}\lesssim\|J^{\f12}_\e p\|_{H^s}.
\eeq

For term $\bigl\|J^{-\f12}_\e\bigl(\frac{\z}{2+\f\e2 \z}\bigr)\bigr\|_{H^{s+1}}$, we first get
\beq\label{L10}
\bigl\|J^{-\f12}_\e\bigl(\frac{\z}{2+\f\e2 \z}\bigr)\bigr\|_{H^{s+1}}
\lesssim\bigl\|\frac{\z}{2+\f\e2 \z}\bigr\|_{H^s}
+\bigl\|J^{-\f12}_\e\p_x\bigl(\frac{\z}{2+\f\e2 \z}\bigr)\bigr\|_{H^s}
+\bigl\|J^{-\f12}_\e\p_y\bigl(\frac{\z}{2+\f\e2 \z}\bigr)\bigr\|_{H^s}.
\eeq
Similar derivations as \eqref{L8} and \eqref{L7a} give rise to
\beq\label{L10a}
\Bigl\|\frac{\z}{2+\f\e2 \z}\Bigr\|_{H^s}\lesssim\|\z\|_{H^s}\quad\text{and}
\quad\Bigl\|\p_x\Bigl(\frac{\z}{2+\f\e2 \z}\Bigr)\Bigr\|_{H^s}\lesssim\|\z_x\|_{H^s}.
\eeq
Noticing that
\beno
\p_y\bigl(\frac{\z}{2+\f\e2 \z}\bigr)
=\bigl(\f12-\frac{\f\e4\z}{2+\f\e2 \z}\bigr)\p_y\z-\frac{\f\e2\z}{(2+\f\e2 \z)^2}\cdot\p_y\z,
\eeno
we obtain by using \eqref{product estimate 1} and \eqref{ansatz 2} that
\beno\begin{aligned}
\Bigl\|J^{-\f12}_\e\p_y\bigl(\frac{\z}{2+\f\e2 \z}\bigr) \Bigr\|_{H^s}
&\lesssim\Bigl(1+\Bigl\|\frac{\e\z}{2+\f\e2 \z}\Bigr\|_{H^s}+\Bigl\|\frac{\e\z}{(2+\f\e2 \z)^2}\Bigr\|_{H^s}\Bigr)\|J^{-\f12}_\e\p_y\z\|_{H^s}\\
&\lesssim C(\e\|\z\|_{H^s})(1+\e\|\z\|_{H^s})\|J^{-\f12}_\e\p_y\z\|_{H^s}
\lesssim\|J^{-\f12}_\e\p_y\z\|_{H^s},
\end{aligned}\eeno
which together with \eqref{L10a} yields 
\beq\label{L10b}
\Bigl\|J^{-\f12}_\e A(D)\Bigl(\frac{\z}{2+\f\e2 \z}\Bigr)\Bigr\|_{H^s}\sim\Bigl\|\p_x\Bigl(\frac{\z}{2+\f\e2 \z}\Bigr)\Bigr\|_{H^s}+\Bigl\|J^{-\f12}_\e\p_y\bigl(\frac{\z}{2+\f\e2 \z}\bigr) \Bigr\|_{H^s}
\lesssim\|(\z_x,J^{-\f12}_\e\z_y)\|_{H^s}.
\eeq

Due to \eqref{L10}, \eqref{L10a} and \eqref{L10b}, we get
\beno
\bigl\|J^{-\f12}_\e\bigl(\frac{\z}{2+\f\e2 \z}\bigr)\bigr\|_{H^{s+1}}
\lesssim\|\z\|_{H^s}+\|(\z_x,J^{-\f12}_\e\z_y)\|_{H^s},
\eeno
from which, \eqref{L10c} and  \eqref{L9}, we deduce that
\beq\label{L9a}\begin{aligned}
\e^2\Bigl\|J^{-\f12}_\e\Bigl(\vv V\cdot\na\bigl(\frac{\z}{2+\f\e2 \z} A^{-2}(D) \p_x^4 p\bigr)\Bigr)\Bigr\|_{\dH^s}&\lesssim\sqrt\e\|\vv V\|_{H^s}\cdot\bigl(\|\z\|_{H^s}+\|(\z_x,J^{-\f12}_\e\z_y)\|_{H^s}\bigr)\|J^{\f12}_\e p\|_{H^s}\\
&
\lesssim\bigl(\|\z\|_{H^s}+\|(\z_x,J^{-\f12}_\e\z_y)\|_{H^s}\bigr)\|J^{\f12}_\e p\|_{H^s}.
\end{aligned}\eeq

{\bf Estimates of the remaining terms.} For  the  second and third terms, using \eqref{product estimate 1} and \eqref{ansatz 2}, we have
\beq\label{L9b}
\|J^{-\f12}_\e[(1+\f\e2\z) N_p]\|_{\dH^s}
\lesssim\|J^{-\f12}_\e N_p\|_{H^s}+\e
\|J^{-\f12}_\e\z\|_{H^s}\|J^{-\f12}_\e N_p\|_{H^s}
\lesssim\|J^{-\f12}_\e N_p\|_{H^s},
\eeq
\beq\label{L9c}
\e\|J^{-\f12}_\e[\vv V\cdot\na A^{-1}(D)N_\theta]\|_{\dH^s}
\lesssim\e
\|J^{-\f12}_\e\vv V\|_{H^s}\|J^{-\f12}_\e\na A^{-1}(D)N_\theta\|_{H^s}\lesssim\|J^{-\f12}_\e N_\theta\|_{H^s}.
\eeq

For  the  last term of $N_{\wt{p},2}$, by using \eqref{product estimate 1} and \eqref{ansatz 2}, we get
\beno\begin{aligned}
&\e^2\Bigl\|J^{\f12}_\e\Bigl(\frac{1+\f\e2 \z}{2+\f\e2 \z} \z  A^{-2}(D)J_\e^{-1} \p_x^4N_p\Bigr)\Bigr\|_{\dH^s}
\lesssim\e\Bigl\|J^{\f12}_\e\Bigl(\frac{(1+\f\e2 \z)\z}{2+\f\e2 \z} \Bigr)\Bigr\|_{H^s}\cdot\e\|A^{-2}(D)\p_x^4J^{-\f12}_\e N_p\|_{H^s},
\end{aligned}\eeno
which along with \eqref{L7} and \eqref{def of A} implies 
\beq\label{L9d}
\e^2\Bigl\|J^{\f12}_\e\Bigl(\frac{1+\f\e2 \z}{2+\f\e2 \z} \z  A^{-2}(D)J_\e^{-1} \p_x^4N_p\Bigr)\Bigr\|_{\dH^s}\lesssim\e\|J^{\f12}_\e\z\|_{H^s}\|J^{-\f12}_\e N_p\|_{H^s}\lesssim\|J^{-\f12}_\e N_p\|_{H^s}.
\eeq

{\bf Estimate of $N_{\wt{p},2}$.} Combining \eqref{L9a}, \eqref{L9b}, \eqref{L9c} and \eqref{L9d}, we obtain
\beq\label{estimate for N wtp 2}
\|J^{-\f12}_\e N_{\wt{p},2}\|_{\dH^s}\lesssim\bigl(\|\z\|_{H^s}+\|(\z_x,J^{-\f12}_\e\z_y)\|_{H^s}\bigr)\|J^{\f12}_\e p\|_{H^s}+\|J^{-\f12}_\e N_p\|_{H^s}
+\|J^{-\f12}_\e N_\theta\|_{H^s}.
\eeq

\smallskip

{\it Step 1.3. Estimate of $N_{\wt{p},3}$.} Thanks to \eqref{def of N wtp 3}, we have
\beno
N_{\wt p,3}=\f{\e^2}{9}\p_x\Bigl(\frac{1+\f\e2 \z}{2+\f\e2 \z} \z\Bigr)\cdot
\p_xJ_\e^{-1}A^{-1}(D)\p_x^4 \wt\theta+\f{\e^2}{18}\p_x^2\Bigl(\frac{1+\f\e2 \z}{2+\f\e2 \z} \z\Bigr)\cdot J_\e^{-1}A^{-1}(D)\p_x^4 \wt\theta.
\eeno
By using \eqref{product estimate 1} and \eqref{def of A}, we get 
\beno\begin{aligned}
\|J^{-\f12}_\e N_{\wt{p},3}\|_{\dH^s}\lesssim&
\Bigl\|J^{-\f12}_\e\p_x\Bigl(\frac{1+\f\e2 \z}{2+\f\e2 \z} \z\Bigr)\Bigr\|_{H^s}\cdot\e^2
\|J^{-\f32}_\e\p_xA^{-1}(D)\p_x^4 \wt\theta\|_{H^s}\\
&+\e^{\f12}\Bigl\|J^{-\f12}_\e\p_x^2\Bigl(\frac{1+\f\e2 \z}{2+\f\e2 \z} \z\Bigr)\Bigr\|_{H^s}\cdot\e^{\f32}
\|J^{-\f32}_\e A^{-1}(D)\p_x^4 \wt\theta\|_{H^s}\\
\lesssim&\Bigl\|\p_x\Bigl(\frac{1+\f\e2 \z}{2+\f\e2 \z} \z\Bigr)\Bigr\|_{H^s}
\cdot{\|\na\wt\theta\|_{H^{s-1}}}
\end{aligned}\eeno
which along with \eqref{estimate for wt theta}, \eqref{L7a} and the ansatz \eqref{ansatz 2} implies 
\beq\label{estimate for N wtp 3}
\|J^{-\f12}_\e N_{\wt{p},3}\|_{\dH^s}\lesssim\|\z_x\|_{H^s}\cdot\bigl(\|J_\e^{\f12} p\|_{H^s}
+\|J_\e^{\f12} \theta\|_{H^s}\bigr).
\eeq

\smallskip

{\it Step 1.4. Estimate of $N_{\wt{p}}$.} Thanks to \eqref{estimate for N wtp 1}, \eqref{estimate for N wtp 2} and \eqref{estimate for N wtp 3}, we obtain
\beq\label{estimate for N wt p}\begin{aligned}
\|J^{-\f12}_\e N_{\wt{p}}\|_{\dH^s}&\lesssim
\bigl(\|J^{\f12}_\e\z_t\|_{H^s}+\|J^{-\f12}_\e\vv V_t\|_{H^s}+\|\z\|_{H^s}+\|(\z_x,J^{-\f12}_\e\z_y)\|_{H^s}\bigr)\cdot\bigl(\|J_\e^{\f12} p\|_{H^s}+\|J_\e^{\f12} \theta\|_{H^s}\bigr)\\
&\qquad+\|J^{-\f12}_\e N_p\|_{H^s}+\|J^{-\f12}_\e N_\theta\|_{H^s}.
\end{aligned}\eeq

\medskip

{\bf Step 2. Estimate of $N_{\wt\theta}$.} We derive the bounds of $N_{\wt\theta,1}$, $N_{\wt\theta,2}$ and $N_{\wt\theta,3}$ one by one.

\smallskip

{\it Step 2.1. Estimate of $N_{\wt\theta,1}$.} For the first two terms of  $N_{\wt\theta,1}$ in \eqref{def of N wttheta 1}, by using \eqref{product A1}, we have
\beno\begin{aligned}
&\|J^{-\f12}_\e A^{-1}(D)(\z_t \cdot A(D)J_\e \theta)\|_{\dH^s}
+\|J^{-\f12}_\e  A^{-1}(D) (\vv V_t\cdot \nabla J_\e p)\|_{\dH^s}\\
&\lesssim\|\z_t\|_{H^s}\|J^{\f12}_\e\theta\|_{H^s}+\|\vv V_t\|_{H^s}\|J^{\f12}_\e\na A^{-1}(D) p\|_{H^s}\\
&\lesssim\|\z_t\|_{H^s}\|J^{\f12}_\e\theta\|_{H^s}+\|\vv V_t\|_{H^s}\|J^{\f12}_\e p\|_{H^s}.
\end{aligned}\eeno

While for the last two terms of $N_{\wt\theta,1}$, we get by using \eqref{commutator A1} that
\beno\begin{aligned}
&\|J^{-\f12}_\e\bigl([A^{-1}(D),\z]A(D)J_\e \theta_t\bigr)\|_{\dH^s}
+\|J^{-\f12}_\e\bigl([A^{-1}(D),\vv V]\cdot\na J_\e p_t\bigr)\|_{\dH^s}\\
&\lesssim\|\z\|_{H^s}\|J_\e\theta_t\|_{H^{s-1}}+\|\vv V\|_{H^s}\|J_\e p_t\|_{H^{s-1}}.
\end{aligned}\eeno

Then we obtain 
\beq\label{estimate for N wttheta 1}
\|J^{-\f12}_\e N_{\wt\theta,1}\|_{\dH^s}\lesssim
\|\z_t\|_{H^s}\|J^{\f12}_\e\theta\|_{H^s}+\|\vv V_t\|_{H^s}\|J^{\f12}_\e p\|_{H^s}
+\|\z\|_{H^s}\|J_\e\theta_t\|_{H^{s-1}}+\|\vv V\|_{H^s}\|J_\e p_t\|_{H^{s-1}}.
\eeq

\smallskip

{\it Step 2.2. Estimate of $N_{\wt\theta,2}$.} For the first term of $N_{\wt\theta,2}$ in \eqref{def of N wttheta 2}, by using \eqref{product estimate 1} and \eqref{product A2}, we have
\beno\begin{aligned}
&\quad\e\Bigl\|J^{-\f12}_\e\Bigl((1+\f\e2\z)A(D)\bigl(\frac{\z}{2+\f\e2 \z} A^{-2}(D) \p_x^4 (\wt p-p) \bigr)\Bigr)\Bigr\|_{\dH^s}\\
&\lesssim(1+\e\|\z\|_{H^s})
\cdot\bigl(\|\frac{\z}{2+\f\e2 \z}\|_{H^s}+\bigl\|J^{-\f12}_\e A(D)\bigl(\frac{\z}{2+\f\e2 \z}\bigr)\bigr\|_{H^s}\bigr)\\
&\qquad\times\e\bigl(\|A^{-2}(D) \p_x^4 (\wt p-p)\|_{H^s}+\|J^{-\f12}_\e A^{-1}(D) \p_x^4 (\wt p-p)\|_{H^s}\bigr)\\
&\lesssim\bigl(\|\z\|_{H^s}+\|(\z_x,J^{-\f12}_\e\z_y)\|_{H^s}\bigr)
\cdot\e^{-\f12}\|J^{\f12}_\e (\wt p-p)\|_{H^s},
\end{aligned}\eeno
where we used \eqref{ansatz 2}, \eqref{L10a} and \eqref{L10b} in the last inequality. Thanks to \eqref{equivalent functional 2} and \eqref{ansatz 2}, we get
\beq\label{L11}\begin{aligned}
&\quad\e\Bigl\|J^{-\f12}_\e\Bigl((1+\f\e2\z)A(D)\bigl(\frac{\z}{2+\f\e2 \z} A^{-2}(D) \p_x^4 (\wt p-p) \bigr)\Bigr)\Bigr\|_{\dH^s}\\
&\lesssim\bigl(\|\z\|_{H^s}+\|(\z_x,J^{-\f12}_\e\z_y)\|_{H^s}\bigr)\cdot\sqrt\e
\bigl(\|J^{\f12}_\e\z\|_{H^s}\|J^{\f12}_\e p\|_{H^s}+\|\vv V\|_{H^s}\|J^{\f12}_\e \theta\|_{H^s}\bigr)\\
&\lesssim\bigl(\|\z\|_{H^s}+\|(\z_x,J^{-\f12}_\e\z_y)\|_{H^s}\bigr)\cdot
\bigl(\|J^{\f12}_\e p\|_{H^s}+\|J^{\f12}_\e \theta\|_{H^s}\bigr).
\end{aligned}\eeq

 For the last two terms of $N_{\wt\theta,2}$, by using \eqref{product estimate 1} and \eqref{ansatz 2}, we get 
\beno\begin{aligned}
&\|J^{-\f12}_\e\bigl((1+\f\e2\z)N_\theta\bigr)\|_{\dH^s}+\e\|J^{-\f12}_\e\bigl(\vv V\cdot\na A^{-1}(D)N_p\bigr)\|_{\dH^s}\\
&\lesssim(1+\e\|\z\|_{H^s})\|J^{-\f12}_\e N_\theta\|_{H^s}
+\e\|\vv V\|_{H^s}\|J^{-\f12}_\e\na A^{-1}(D)N_p\|_{H^s}\\
&\lesssim\|J^{-\f12}_\e N_\theta\|_{H^s}+\|J^{-\f12}_\e N_p\|_{H^s},
\end{aligned}\eeno
which along with \eqref{L11} implies
\beq\label{estimate for N wttheta 2}
\|J^{-\f12}_\e N_{\wt\theta,2}\|_{\dH^s}\lesssim\bigl(\|\z\|_{H^s}+\|(\z_x,J^{-\f12}_\e\z_y)\|_{H^s}\bigr)\cdot\bigl(\|J^{\f12}_\e p\|_{H^s}+\|J^{\f12}_\e \theta\|_{H^s}\bigr)+\|J^{-\f12}_\e N_p\|_{H^s}+\|J^{-\f12}_\e N_\theta\|_{H^s}.
\eeq

\smallskip

{\it Step 2.3. Estimate of $N_{\wt\theta,3}$.} By using \eqref{product estimate 1} and \eqref{commutator A2}, we deduce from \eqref{def of N wttheta 3} that
\beno\begin{aligned}
\|J^{-\f12}_\e N_{\wt\theta,3}\|_{\dH^s}
&\lesssim(1+\e\|J^{-\f12}\z\|_{H^s})\cdot\e\Bigl\|J^{-\f12}_\e\Bigl(\Bigl[A(D),\frac{\z}{2+\f\e2 \z}\Bigr]A^{-2}(D) \p_x^4\wt p \Bigr)\Bigr\|_{H^s}\\
&\lesssim(1+\e\|\z\|_{H^s})\cdot\Bigl(\Bigl\|\frac{\z}{2+\f\e2 \z}\Bigr\|_{H^s}+\Bigl\|J^{-\f12}_\e A(D)\bigl(\frac{\z}{2+\f\e2 \z}\bigr)\Bigr\|_{H^s}\Bigr)\cdot\e\|A^{-2}(D) \p_x^4\wt p\|_{H^s},
\end{aligned}\eeno
which together with \eqref{ansatz 2}, \eqref{L10a} and \eqref{L10b} implies
\beno
\|J^{-\f12}_\e N_{\wt\theta,3}\|_{\dH^s}\lesssim\bigl(\|\z\|_{H^s}+\|(\z_x,J^{-\f12}_\e\z_y\|_{H^s}\bigr)\|\wt p\|_{H^s}.
\eeno

Thanks to \eqref{estimate for wt theta} and \eqref{ansatz 2}, we obtain 
\beq\label{estimate for N wttheta 3}
\|J^{-\f12}_\e N_{\wt\theta,3}\|_{\dH^s}\lesssim\bigl(\|\z\|_{H^s}+\|(\z_x,J^{-\f12}_\e\z_y\|_{H^s}\bigr)\cdot\bigl(\|J^{\f12}_\e p\|_{H^s}+\|J^{\f12}_\e \theta\|_{H^s}\bigr).
\eeq

\smallskip

{\it Step 2.4. Estimate of $N_{\wt\theta}$.} Combining estimates \eqref{estimate for N wttheta 1}, \eqref{estimate for N wttheta 2} and \eqref{estimate for N wttheta 3}, we get 
\beq\label{estimate for N wt theta}\begin{aligned}
\|J^{-\f12}_\e N_{\wt\theta}\|_{\dH^s}&\lesssim\bigl(\|J^{\f12}_\e\z_t\|_{H^s}+\|\vv V_t\|_{H^s}+\|\z\|_{H^s}+\|(\z_x,J^{-\f12}_\e\z_y)\|_{H^s}\bigr)
\cdot\bigl(\|J_\e^{\f12} p\|_{H^s}+\|J_\e^{\f12} \theta\|_{H^s}\bigr)\\
&\qquad
+\|\z\|_{H^s}\|J_\e\theta_t\|_{H^{s-1}}+\|\vv V\|_{H^s}\|J_\e p_t\|_{H^{s-1}}
+\|J^{-\f12}_\e N_p\|_{H^s}+\|J^{-\f12}_\e N_\theta\|_{H^s}.
\end{aligned}\eeq

\medskip

{\bf Step 3. Estimates of $(\z_t,\vv V_t)$ and $(p_t,\theta_t)$.} We derive the bounds of $\|J^{\f12}_\e\z_t\|_{H^s}$, $\|\vv V_t\|_{H^s}$, $\|J_\e p_t\|_{H^{s-1}}$
and $\|J_\e\theta_t\|_{H^{s-1}}$ one by one.

{\it Step 3.1. Estimates of $\|J^{\f12}_\e\z_t\|_{H^s}$ and $\|\vv V_t\|_{H^s}$}. Going back to system \eqref{WTB case 3}, we have
\beno\begin{aligned}
&J^{\f12}_\e\z_t=-J^{\f12}_\e v_x-J^{-\f12}_\e w_y-\e J^{-\f12}_\e(\vv V\cdot\na \z+ \f12 \z \div\vv V),\\
&J^{\f12}_\e\vv V_t=-J^{-\f12}_\e\na\z-\e J^{-\f12}_\e(\vv V\cdot\na\vv V+\f12\z\na\z),
\end{aligned}\eeno
where we used the notations $J_\e=1-\f\e3\p_x^2$ and $\vv V=(v,w)^T$. Thanks to \eqref{product estimate 1}, we obtain
\beno\begin{aligned}
&\|J^{\f12}_\e\z_t\|_{H^s}\lesssim\|J^{\f12}_\e v_x\|_{H^s}+\|J^{-\f12}_\e w_y\|_{H^s}+\e\|J^{-\f12}_\e\vv V\|_{H^s}\|J^{-\f12}_\e\na\z\|_{H^s}+\e\|J^{-\f12}_\e\z\|_{H^s}\|J^{-\f12}_\e\div\vv V\|_{H^s},\\
&\|J^{\f12}_\e\vv V_t\|_{H^s}\lesssim\|J^{-\f12}_\e\na\z\|_{H^s}
+\e\|J^{-\f12}_\e\vv V\|_{H^s}\|J^{-\f12}_\e\na\vv V\|_{H^s}+\e\|J^{-\f12}_\e\z\|_{H^s}\|J^{-\f12}_\e\na\z\|_{H^s},
\end{aligned}\eeno
which along with \eqref{ansatz 2}, \eqref{equivalent functional 1} and the fact that $\|J^{-\f12}_\e\na f\|_{H^s}\lesssim\|(\p_xf,J^{-\f12}_\e\p_yf)\|_{H^s}$ shows 
\beq\label{L16}\begin{aligned}
\|J^{\f12}_\e\z_t\|_{H^s}+\|J^{\f12}_\e\vv V_t\|_{H^s}
&\lesssim\|J^{\f12}_\e v_x\|_{H^s}+\|J^{-\f12}_\e\na\vv V\|_{H^s}+\|J^{-\f12}_\e\na\z\|_{H^s}\\
&\lesssim\|J^{\f12}_\e p\|_{H^s}+\|J^{\f12}_\e\theta\|_{H^s}.
\end{aligned}\eeq

\smallskip

{\it Step 3.2. Estimates of $\|J_\e p_t\|_{H^{s-1}}$ and $\|J_\e\theta_t\|_{H^{s-1}}$.} Firstly, we deduce from \eqref{E 1} that
\beno
\|J_\e p_t\|_{H^{s-1}}\leq\|A(D)\wt\theta\|_{H^{s-1}}+\e\|N_p\|_{H^{s-1}}
\lesssim\|J^{\f12}_\e\na\wt\theta\|_{H^{s-1}}+\e\|J^{-\f12}_\e N_p\|_{H^s},
\eeno
which along with \eqref{estimate for wt theta} implies
\beq\label{L17}
\|J_\e p_t\|_{H^{s-1}}\lesssim\|J_\e^{\f12} p\|_{H^s}+\|J_\e^{\f12} \theta\|_{H^s}+\e\|J^{-\f12}_\e N_p\|_{H^s}.
\eeq

Due to \eqref{E 1a}, we have
\beno\begin{aligned}
\|J_\e\theta_t\|_{H^{s-1}}&\leq\|A(D)\wt p\|_{H^{s-1}}+\f{\e^2}{6}\Bigl\|A(D)\Bigl(\frac{\z}{2+\f\e2 \z} A^{-2}(D) \p_x^4 p\Bigr)\Bigr\|_{H^{s-1}}+\e\|N_\theta\|_{H^{s-1}}\\
&\lesssim\|J^{\f12}_\e\wt p\|_{H^s}+\e^2\Bigl\|J^{\f12}_\e\Bigl(\frac{\z}{2+\f\e2 \z} A^{-2}(D) \p_x^4 p\Bigr)\Bigr\|_{H^s}+\e\|J^{-\f12}_\e N_\theta\|_{H^s}.
\end{aligned}\eeno
By virtue of \eqref{def of A}, \eqref{L10a} and \eqref{ansatz 2}, we obtain
\beno\begin{aligned}
\e^2\Bigl\|J^{\f12}_\e\Bigl(\frac{\z}{2+\f\e2 \z} A^{-2}(D) \p_x^4 p\Bigr)\Bigr\|_{H^s}&\lesssim\e\Bigl\|J^{\f12}_\e\Bigl(\frac{\z}{2+\f\e2 \z}\Bigr)\Bigr\|_{H^s}\cdot\e\| J^{\f12}_\e A^{-2}(D) \p_x^4 p\|_{H^s}\\
&\lesssim\e\|J^{\f12}_\e\z\|_{H^s}\cdot\|J^{\f12}_\e p\|_{H^s}\lesssim\|J^{\f12}_\e p\|_{H^s}.
\end{aligned}\eeno
Then  we get by using \eqref{estimate for wt theta} that
\beq\label{L18}
\|J_\e\theta_t\|_{H^{s-1}}\lesssim\|J^{\f12}_\e p\|_{H^s}+\|J^{\f12}_\e\theta\|_{H^s}+\e\|J^{-\f12}_\e N_\theta\|_{H^s}.
\eeq

Thanks to  Lemma \ref{lem for nonlinear term 1},
we deduce from \eqref{L17} and \eqref{L18} that
\beno
\|J_\e p_t\|_{H^{s-1}}+\|J_\e\theta_t\|_{H^{s-1}}
\lesssim\|J_\e^{\f12} p\|_{H^s}+\|J_\e^{\f12} \theta\|_{H^s}+
\e\|J_\e^{\f12} p\|_{H^s}^2+\e\|J_\e^{\f12} \theta\|_{H^s}^2,
\eeno
which along with \eqref{ansatz 1} yields to 
\beq\label{L19}
\|J_\e p_t\|_{H^{s-1}}+\|J_\e\theta_t\|_{H^{s-1}}
\lesssim\|J_\e^{\f12} p\|_{H^s}+\|J_\e^{\f12} \theta\|_{H^s}.
\eeq

\medskip

{\bf Step 4. Estimates of $N_{\wt p}$ and $N_{\wt\theta}$.} Plugging the estimates \eqref{L16}, \eqref{L19} into \eqref{estimate for N wt p} and \eqref{estimate for N wt theta}, and using \eqref{estimate for N p theta} and \eqref{equivalent functional 1}, we finally obtain
\beno
\|J^{-\f12}_\e N_{\wt p}\|_{\dH^s}+\|J^{-\f12}_\e N_{\wt\theta}\|_{\dH^s}
\lesssim\bigl(\|\z\|_{H^s}+\|\vv V\|_{H^s}+\|J_\e^{\f12} p\|_{H^s}+\|J_\e^{\f12} \theta\|_{H^s}\bigr)\bigl(\|J_\e^{\f12} p\|_{H^s}+\|J_\e^{\f12} \theta\|_{H^s}\bigr),
\eeno
which yields to \eqref{estimate for wt N p theta}. The proof is complete.
\end{proof}

\subsection{The proof of Theorem \ref{main theorem} for case 1}		
	
The proof of Theorem \ref{main theorem} relies on the continuity argument and the {\it a priori} energy estimate.	

\subsubsection{ Ansatz for continuity arguments.} Let  $(\vv V,\z)$ be a smooth enough solution of \eqref{WTB case 3} over time interval $[0,T_0/\e]$. Here $T_0>0$ will be determined later on. 

Our first ansatz for continuity arguments is
\beq\label{ansatz a}
\sup_{t\in[0,T_0/\e]}\sqrt\e(\|J^{\f12}_\e\z\|_{H^s}+\|\vv V\|_{H^s})+\sup_{t\in[0,T_0/\e]}\e(\|J^{\f12}_\e p\|_{H^s}+\|J^{\f12}_\e\theta\|_{H^s})\leq1,
\eeq
where $(p,\theta)$ are defined in \eqref{new unknows}. 

Let us define the energy functional associated to \eqref{WTB case 3} as follows:
\beq\label{energy functional for zeta V}
E_{s}(t)\eqdefa E_{s,l}(t)+E_{s,h}(t),
\eeq
and
\beq\label{lower order functional}\begin{aligned}
E_{s,l}(t)&\eqdefa\|J_\e v\|_{H^s}^2+\|J^{\f12}_\e w\|_{H^s}^2+\|J^{\f12}_\e\z\|_{H^s}^2\\
&\sim\|v\|_{H^s}^2+\e^2\|\p_x^2v\|_{H^s}^2+\|(w,\z)\|_{H^s}^2+\e\|(\p_xw,\p_x\z)\|_{H^s}^2.
\end{aligned}\eeq
\beq\label{highest order functional}\begin{aligned}
E_{s,h}(t)&\eqdefa\|(J^{\f12}_\e v_x, v_y)\|_{\dH^s}^2+\|(w_x, J^{-\f12}_\e w_y)\|_{\dH^s}^2+\|(\z_x, J^{-\f12}_\e\z_y)\|_{\dH^s}^2,\\
&\sim\|J^{\f12}p\|_{\dH^s}^2+\|J^{\f12}\theta\|_{\dH^s}^2,
\end{aligned}\eeq
where we used \eqref{equivalent functional 1} in the last formula.
We remark that $E_{s,l}(t)$ and $E_{s,h}(t)$ are lower and highest order energy functionals of \eqref{WTB case 3} respectively.

Our second ansatz is about the energy which reads
\beq\label{ansatz b}
\sup_{t\in[0,T_0/\e]}E_s(t)\leq 2C_0 E_s(0).
\eeq
Here the constant $C_0>1$ will also be determined later on. 

The standard continuity argument shows that: since for sufficiently small $\e>0$,
\beno
E_s(0)\leq C_0 E_s(0)\quad\text{and}\quad
\sqrt\e(\|J^{\f12}_\e\z_0\|_{H^s}+\|(v_0,w_0)\|_{H^s})+\e(\|J^{\f12}_\e p_0\|_{H^s}+\|J^{\f12}_\e\theta_0\|_{H^s})\leq\f12,
\eeno
the following {\it a priori} energy estimates \eqref{total energy estimate} and the classical mollification  method show that the solution of \eqref{WTB case 3} exists in a short time interval $[0,T^*)$ and the ansatz \eqref{ansatz a}, \eqref{ansatz b} also hold over $[0,T^*)$. Here $T^*>0$ is the lifespan of the solution to \eqref{WTB case 3} with \eqref{ansatz a} and \eqref{ansatz b} being correct. Without loss of generality, we assume that $T^*>T_0/\e$.

To close the continuity argument, we need to verify that: there exists sufficiently small $\e_0>0$, such that for any $\e\in(0,\e_0)$, we improve the ansatz \eqref{ansatz a} and \eqref{ansatz b} as follows:
\beq\label{improve ansatz a}
\sup_{t\in[0,T_0/\e]}\sqrt\e(\|J^{\f12}_\e\z\|_{H^s}+\|(v,w)\|_{H^s})+\sup_{t\in[0,T_0/\e]}\e(\|J^{\f12}_\e p\|_{H^s}+\|J^{\f12}_\e\theta\|_{H^s})\leq\f12,
\eeq
\beq\label{improve ansatz b}
\sup_{t\in[0,T_0/\e]}E_s(t)\leq C_0 E_s(0).
\eeq

Indeed, we take
\beno
C_0=2C_1,\quad T_0=\f{1}{4\sqrt{C_1}C_2E_s^{\f12}(0)},
\eeno
where $C_1>1$ and $C_2>1$ are constants stated in the following Proposition \ref{prop for a priori estimate}.
Then  we deduce from \eqref{total energy estimate} that
\beno
\max_{t\in[0,T_0/\e]} E_s(t)\leq 2C_1 E_s(0)=C_0 E_s(0).
\eeno 
This is \eqref{improve ansatz b}. 

Moreover, due to \eqref{equivalent functional 1}, \eqref{lower order functional} and \eqref{highest order functional}, there exists $C_3>0$ such that
\beno
\sqrt\e(\|J^{\f12}_\e\z\|_{H^s}+\|(v,w)\|_{H^s})+\e(\|J^{\f12}_\e p\|_{H^s}+\|J^{\f12}_\e\theta\|_{H^s})\leq C_3\sqrt{\e E_s(t)}\leq C_3\sqrt{\e C_0E_s(0)}.
\eeno
Taking $\e_0=\f{1}{4C_3^2C_0E_s(0)}$, we obtain \eqref{improve ansatz a} for any $\e\in(0,\e_0)$. 

Thus, the continuity argument is closed and Theorem \ref{main theorem} for case 1 is proved.

\medskip

It remains to verify Proposition \ref{prop for a priori estimate} in the following subsection.

\subsubsection{The {\it a priori} energy estimates.} We shall derive the  {\it a  priori} energy estimates of solutions to \eqref{WTB case 3}. The main result is stated in the following proposition.

\begin{proposition}\label{prop for a priori estimate}
Let $s>3$ and $(\vv V,\z)=(v,w,\z)$ be a smooth enough solution of \eqref{WTB case 3} satisying ansatz \eqref{ansatz a} and \eqref{ansatz b} for some $T_0>0$. There holds
\beq\label{total energy estimate}
E_s(t)\leq C_1E_s(0)+C_2\e t\max_{\tau\in[0,t]} E_s(\tau)^{\f32},\quad\forall\, t\in[0,T_0/\e],
\eeq
where $C_1,C_2>1$ are universal constants.
\end{proposition}
\begin{proof}
We divide the proof of \eqref{total energy estimate} into two parts:  the {\it lower order} energy estimate and the {\it  highest order} energy estimate. The former is derived from the original system \eqref{WTB case 3} in terms of $(v,w,\z)$, and the latter  is derived from the system \eqref{equation for tilde p and theta} in terms of $(\wt p,\wt\theta)$.

\smallskip

{\bf Step 1. The lower order energy estimates.} Thanks to the definition of $E_{s,l}(t)$ and system \eqref{WTB case 3}, one calculates that 
\beno\begin{aligned}
\f12\f{d}{dt}E_{s,l}(t)&=\bigl(J_\e v_t\,|\,J_\e v\bigr)_{H^s}
+\bigl(J_\e w_t\,|\,w\bigr)_{H^s}+\bigl(J_\e\z_t\,|\,\z\bigr)_{H^s}\\
&=-\e\bigl(vv_x+ wv_y+\f12\z \z_x\,|\,J_\e v\bigr)_{H^s}
-\e\bigl(v w_x+w w_y+ \f12\z \z_y\,|\,w\bigr)_{H^s}\\
&\qquad
-\e\bigl(v \z_x+w\z_y+ \f12 \z v_x + \f12\z w_y\,|\,\z\bigr)_{H^s}\\
&\eqdefa-\e II_1-\e II_2-\e II_3.
\end{aligned}\eeno

For $II_1$, using the curl-free condition $v_y=w_x$, we get
\beno\begin{aligned}
|II_1|&\lesssim\|J_\e v\|_{H^s}\cdot\|vv_x+ ww_x+\f12\z \z_x\|_{H^s}\\
&
\lesssim\|J_\e v\|_{H^s}\bigl(\|v\|_{H^s}\|v_x\|_{H^s}+\|w\|_{H^s}\|w_x\|_{H^s}+\|\z\|_{H^s}\|\z_x\|_{H^s}\bigr),
\end{aligned}\eeno
which along with \eqref{energy functional for zeta V}, \eqref{lower order functional} and \eqref{highest order functional} implies 
\beno
|II_1|\lesssim  E_s^{\f32}(t).
\eeno

For $II_2$, using \eqref{product estimate 1}, we have
\beno\begin{aligned}
|II_2|&\lesssim\|J^{\f12}_\e w\|_{H^s}\cdot\|J^{-\f12}_\e(v w_x+w w_y+ \f12\z \z_y)\|_{H^s}\\
&
\lesssim\|J^{\f12}_\e w\|_{H^s}\bigl(\|v\|_{H^s}\|J^{-\f12}_\e w_x\|_{H^s}+\|w\|_{H^s}\|J^{-\f12}_\e w_y\|_{H^s}+\|\z\|_{H^s}\|J^{-\f12}_\e\z_y\|_{H^s}\bigr),
\end{aligned}\eeno
which yields to 
\beno
|II_2|\lesssim  E_s^{\f32}(t).
\eeno
The same estimate holds for $II_3$. 

Then there exists constant $c_1>0$ such that
\beq\label{lower order energy estimate}
\f12\f{d}{dt}E_{s,l}(t)\leq c_1\e  E_s^{\f32}(t),\quad\forall\, t\in(0,T_0/\e].
\eeq

\smallskip

{\bf Step 2. The highest order energy estimates.} In this step, we derive the highest order energy estimates for \eqref{WTB case 3} via $(\wt p,\wt\theta)$ that satisfies \eqref{equation for tilde p and theta}. 

{\it Step 2.1. Energy functional for \eqref{equation for tilde p and theta}.} Let us define the highest order energy functional associated with \eqref{equation for tilde p and theta} as 
\beno
\wt{E}_{s,h}(t)\eqdefa\|J^{\f12}_\e\wt p\|_{\dH^s}^2+\|J^{\f12}_\e\wt\theta\|_{\dH^s}^2.
\eeno

Due to \eqref{equivalent functional 2}, there exists $c_2>0$ such that
\beno
\wt{E}_{s,h}(t)\geq\|J^{\f12}_\e p\|_{\dH^s}^2+\|J^{\f12}_\e \theta\|_{\dH^s}^2- c_2\e^2\bigl(\|J^{\f12}_\e\z\|_{H^s}^2+\|(v,w)\|_{H^s}^2\bigr)\bigl(\|J^{\f12}_\e p\|_{H^s}^2+\|J^{\f12}_\e\theta\|_{H^s}^2\bigr),
\eeno
which along with the ansatz \eqref{ansatz a} yields to
\beno
\wt{E}_{s,h}(t)\geq\|J^{\f12}_\e p\|_{\dH^s}^2+\|J^{\f12}_\e \theta\|_{\dH^s}^2-
c_2\e\bigl(\|J^{\f12}_\e p\|_{H^s}^2+\|J^{\f12}_\e\theta\|_{H^s}^2\bigr).
\eeno 
Since $\|f\|_{H^s}^2\sim\|f\|_{\dH^s}^2+\|f\|_{L^2}^2$, there exists $c_3>0$ such that
\beno
\wt{E}_{s,h}(t)\geq(1-c_3\e)(\|J^{\f12}_\e p\|_{\dH^s}^2+\|J^{\f12}_\e \theta\|_{\dH^s}^2)-
c_3\e\bigl(\|J^{\f12}_\e p\|_{L^2}^2+\|J^{\f12}_\e\theta\|_{L^2}^2\bigr).
\eeno 
Taking $\e_1=\f{1}{2c_3}$, we get for any $\e\in(0,\e_1]$,
\beno
\wt{E}_{s,h}(t)\geq\f12\bigl(\|J^{\f12}_\e p\|_{\dH^s}^2+\|J^{\f12}_\e \theta\|_{\dH^s}^2\bigr)-\f12\bigl(\|J^{\f12}_\e p\|_{L^2}^2+\|J^{\f12}_\e\theta\|_{L^2}^2\bigr)
\eeno

Thanks to \eqref{equivalent functional 1}, we have 
\beno
\|J^{\f12}_\e p\|_{\dH^s}^2+\|J^{\f12}_\e \theta\|_{\dH^s}^2\sim E_{s,h}(t)
\quad\text{and}\quad\|J^{\f12}_\e p\|_{L^2}^2+\|J^{\f12}_\e\theta\|_{L^2}^2\lesssim E_{s,l}(t).
\eeno
Then there exist $c_4,c_5>0$ such that 
\beno
\wt{E}_{s,h}(t)\geq c_4E_{s,h}(t)-c_5 E_{s,l}(t),\quad\forall\, t\in[0,T_0/\e].
\eeno

On the other hand, it is easy to check that there exists $c_6>0$ such that
\beno
\wt{E}_{s,h}(t)\leq c_6E_{s,h}(t)+c_6 E_{s,l}(t)=c_6 E_s(t),\quad\forall\, t\in[0,T_0/\e]
\eeno

Thus, we obtain
\beq\label{equivalence for highest functional}
c_4E_{s,h}(t)-c_5 E_{s,l}(t)\leq \wt{E}_{s,h}(t)\leq c_6E_{s}(t),\quad\forall\, t\in[0,T_0/\e].
\eeq

\smallskip

{\it Step 2.2. Energy estimate for $(\wt p,\wt\theta)$.} Thanks to the definition of $\wt{E}_{s,h}(t)$ and the system \eqref{equation for tilde p and theta}, one calculates that
\beq\label{F 0}
\f12\f{d}{dt}\wt{E}_{s,h}(t)=\bigl(J_\e\wt p_t\,|\,\wt p\bigr)_{\dH^s}+
\bigl(J_\e\wt\theta_t\,|\,\wt\theta\bigr)_{\dH^s}=\mathcal{S}+\mathcal{T} +\mathcal{N},
\eeq
where 
\beno\begin{aligned}
\mathcal{S} &\eqdefa\Bigl( \bigl(1+\f{\e}{2}\z -\f{\e^2}{6} \frac{1+\f\e2\z}{2+\f\e2\z} \z A^{-2}(D) \p_x^4  \bigr)A(D)\wt{\theta} \,\Big|\,  \wt{p} \Bigr)_{\dH^s}-\Bigl( \bigl(1+\f{\e}{2}\z -\f{\e^2}{6} \frac{1+\f\e2\z}{2+\f\e2\z} \z A^{-2}(D) \p_x^4  \bigr) A(D)\wt{p}\,\Big|\,  \wt{\theta} \Bigr)_{\dH^s}, \\
\mathcal{T} & \eqdefa -\e \bigl(\vv V\cdot \na\wt{p} \,|\, \wt{p} \bigr)_{\dH^s}-\e \bigl(\vv V\cdot \na \wt{\theta} \,|\, \wt{\theta} \bigr)_{\dH^s}, \quad
\mathcal{N}  \eqdefa \e\bigl(N_{\wt{p}} \,|\, \wt{p} \bigr)_{\dH^s} + \e\bigl( N_{\wt\theta} \,|\, \wt{\theta} \bigr)_{\dH^s}. 
\end{aligned}\eeno

{\it Step 2.2.1. Estimate of $\cS$.} For $\cS$, a direct caculation gives rise to 
\beno\begin{aligned}
\cS&=\f\e2\bigl\{\bigl(|D|^s(\z A(D)\wt\theta)\,|\,|D|^s\wt p\bigr)_{L^2}
-\bigl(|D|^s(\z A(D)\wt p)\,|\,|D|^s\wt\theta\bigr)_{L^2}\bigr\}\\
&\qquad+\f{\e^2}{6}\bigl\{-\bigl(|D|^s(\wt\z A^{-1}(D) \p_x^4 \wt{\theta})\,\big|\, |D|^s\wt p\bigr)_{L^2}+\bigl(|D|^s(\wt\z A^{-1}(D) \p_x^4 \wt{p})\,\big|\, |D|^s\wt\theta\bigr)_{L^2}\bigr\}\\
&\eqdefa\f\e2\cS_1+\f{\e^2}{6}\cS_2,
\end{aligned}\eeno
where we denoted $\wt\z\eqdefa\frac{1+\f\e2\z}{2+\f\e2\z} \z$ for simplicity.

{\bf i).} For $\cS_1$, we have
\beno\begin{aligned}
\cS_1&=\bigl([|D|^s,\z] A(D)\wt\theta\,|\,|D|^s\wt p\bigr)_{L^2}
+\bigl([A(D),\z]|D|^s\wt p\,|\,|D|^s\wt\theta\bigr)_{L^2}
-\bigl([|D|^s,\z]A(D)\wt p\,|\,|D|^s\wt\theta\bigr)_{L^2}.
\end{aligned}\eeno

Thanks to \eqref{commutator A3}, we have 
\beno\begin{aligned}
|\bigl([A(D),\z]|D|^s\wt p\,|\,|D|^s\wt\theta\bigr)_{L^2}|
&\lesssim\|J^{-\f12}_\e\bigl([A(D),\z]|D|^s\wt p\bigr)\|_{L^2}\|J^{\f12}_\e|D|^s\wt\theta\|_{L^2}
\lesssim\|\z\|_{H^s}\||D|^s\wt p\|_{L^2}\|J^{\f12}_\e\wt\theta\|_{\dH^s}\\
&\lesssim\|\z\|_{H^s}\|\wt p\|_{\dH^s}\|J^{\f12}_\e\wt\theta\|_{\dH^s}
\end{aligned}\eeno

Due to \eqref{commutator D}, we have
\beno\begin{aligned}
|\bigl([|D|^s,\z] A(D)\wt\theta\,|\,|D|^s\wt p\bigr)_{L^2}|
&\lesssim\|[|D|^s,\z] A(D)\wt\theta\|_{L^2}\||D|^s\wt p\|_{L^2}
\lesssim\|\z\|_{H^s}\|A(D)\wt\theta\|_{H^{s-1}}\|\wt p\|_{\dH^s}\\
&\lesssim\|\z\|_{H^s}\|\wt p\|_{\dH^s}\|J^{\f12}_\e\na\wt\theta\|_{H^{s-1}}.
\end{aligned}\eeno
Similarly, we have
\beno
|\bigl([|D|^s,\z] A(D)\wt p\,|\,|D|^s\wt\theta\bigr)_{L^2}|
\lesssim\|\z\|_{H^s}\|J^{\f12}_\e\wt p\|_{H^s}\|\wt\theta\|_{\dH^s}.
\eeno

Then we obtain
\beq\label{F 1}
|\cS_1|\lesssim\|\z\|_{H^s}\|J^{\f12}_\e\wt p\|_{H^s}\|J^{\f12}_\e\na\wt\theta\|_{H^{s-1}}.
\eeq

{\bf ii).} For $\cS_2$, we firstly get that
\beno
&\bigl(|D|^s(\wt\z A^{-1}(D) \p_x^4 \wt{\theta})\,\big|\, |D|^s\wt p\bigr)_{L^2}
=\bigl(|D|^s(\wt\z A^{-1}(D) \p_x^2 \wt{\theta})\,\big|\, |D|^s\p_x^2\wt p\bigr)_{L^2}+\cQ_1(\wt\theta,\wt p),
\eeno
where
\beno
\cQ_1(\wt\theta,\wt p)\eqdefa2\bigl(|D|^s(\p_x\wt\z A^{-1}(D) \p_x^2 \wt{\theta})\,\big|\, |D|^s\p_x\wt p\bigr)_{L^2}+\bigl(|D|^s(\p_x^2\wt\z A^{-1}(D) \p_x^4 \wt{\theta})\,\big|\, |D|^s\wt p\bigr)_{L^2}.
\eeno

Due to \eqref{tame} and \eqref{product estimate 1}, it is easy to check that
\beno\begin{aligned}
\e|\cQ_1(\wt\theta,\wt p)|&\lesssim\sqrt\e\|\p_x\wt\z A^{-1}(D) \p_x^2 \wt{\theta}\|_{\dH^s}\cdot\sqrt\e\|\p_x\wt p\|_{\dH^s}+\e\|J^{-\f12}_\e(\p_x^2\wt\z A^{-1}(D) \p_x^2 \wt{\theta})\|_{\dH^s}\cdot\|J^{\f12}_\e\wt p\|_{\dH^s}\\
&\lesssim\|\p_x\wt\z\|_{H^s}\cdot\sqrt\e\|A^{-1}(D) \p_x^2 \wt{\theta})\|_{H^s}\cdot\|J^{\f12}_\e\wt p\|_{\dH^s}\\
&\qquad+\sqrt\e\|J^{-\f12}_\e\p_x^2\wt\z\|_{H^s}\cdot\sqrt\e\|J^{-\f12}_\e A^{-1}(D) \p_x^2 \wt{\theta})\|_{H^s}\cdot\|J^{\f12}_\e\wt p\|_{\dH^s},\\
&\lesssim\|\p_x\wt\z\|_{H^s}\cdot\sqrt\e\|\p_x\wt\theta\|_{H^s}\cdot\|J^{\f12}_\e\wt p\|_{\dH^s},
\end{aligned}\eeno
where we used the fact that $A(\xi)\sim|\xi|+\sqrt\e\xi_1^2$ in the last inequality. Due to the ansatz \eqref{ansatz a} and the notation $\wt\z\eqdefa\frac{1+\f\e2\z}{2+\f\e2\z} \z$, we obtain by using \eqref{L7a} that
\beq\label{F 2}
\e|\cQ_1(\wt\theta,\wt p)|\lesssim\|\z_x\|_{H^s}\|J^{\f12}_\e\wt p\|_{\dH^s}\|J^{\f12}_\e\na\wt\theta\|_{H^{s-1}}.
\eeq

Similarly, we get
\beno
\bigl(|D|^s(\wt\z A^{-1}(D) \p_x^4 \wt p)\,\big|\, |D|^s\wt\theta\bigr)_{L^2}
=\bigl(|D|^s(\wt\z A^{-1}(D) \p_x^2 \wt p)\,\big|\, |D|^s\p_x^2\wt\theta\bigr)_{L^2}+\cQ_1(\wt p,\wt\theta)
\eeno
and 
\beq\label{F 3}
\e|\cQ_1(\wt p,\wt\theta)|\lesssim\|\z_x\|_{H^s}\|J^{\f12}_\e\wt\theta\|_{\dH^s}\|J^{\f12}_\e\na\wt p\|_{H^{s-1}}.
\eeq

Therefore, denoting by $\Theta=A^{-1}(D) \p_x^2 \wt{\theta}$ and $P=A^{-1}(D) \p_x^2 \wt{p}$ for simplicity, we get
\beno\begin{aligned}
\cS_2&=-\bigl(|D|^s(\wt\z \Theta)\,\big|\, |D|^sA(D) P\bigr)_{L^2}+\bigl(|D|^s(\wt\z P)\,\big|\, |D|^sA(D)\Theta\bigr)_{L^2}-\cQ_1(\wt\theta,\wt p)+\cQ_1(\wt p,\wt\theta)\\
&=-\bigl(A(D)([|D|^s,\wt\z]\Theta)\,\big|\, |D|^s P\bigr)_{L^2}+\bigl(A(D)([|D|^s,\wt\z] P)\,\big|\, |D|^s\Theta\bigr)_{L^2}\\
&\qquad-\bigl([A(D),\wt\z\bigr] |D|^s \Theta\,\big|\, |D|^s P\bigr)_{L^2}-\cQ_1(\wt\theta,\wt p)+\cQ_1(\wt p,\wt\theta).
\end{aligned}\eeno

Thanks to the fact that $A(\xi)\sim(1+\sqrt\e|\xi_1|)|\xi_1|+|\xi_2|$, we get
\beq\label{F 4}\begin{aligned}
|\bigl(A(D)([|D|^s,\wt\z]\Theta)\,\big|\, |D|^s P\bigr)_{L^2}|
&\lesssim\|J^{-\f12}_\e A(D)([|D|^s,\wt\z]\Theta)\|_{L^2}\|J^{\f12}_\e P\|_{\dH^s}\\
&\lesssim\bigl(\|\p_x([|D|^s,\wt\z]\Theta)\|_{L^2}+\|J^{-\f12}_\e\p_y([|D|^s,\wt\z]\Theta)\|_{L^2}\bigr)\|J^{\f12}_\e P\|_{\dH^s}.
\end{aligned}\eeq
By virtue of Leibniz formula and \eqref{commutator D}, we have
\beq\label{F 5}
\|\p_x([|D|^s,\wt\z]\Theta)\|_{L^2}
\lesssim\|\p_x\wt\z\|_{H^s}\|\Theta\|_{H^{s-1}}+\|\wt\z\|_{H^s}\|\p_x\Theta\|_{H^{s-1}}.
\eeq
At the meantime, we get
\beno\begin{aligned}
\|J^{-\f12}_\e\p_y([|D|^s,\wt\z]\Theta)\|_{L^2}
&\lesssim\|J^{-\f12}_\e|D|^s(\p_y\wt\z\Theta)\|_{L^2}+\|J^{-\f12}_\e(\p_y\wt\z|D|^s\Theta)\|_{L^2}+\|J^{-\f12}_\e([|D|^s,\wt\z]\p_y\Theta)\|_{L^2}\\
&\lesssim\|J^{-\f12}_\e(\p_y\wt\z\Theta)\|_{\dH^s}+\|\p_y\wt\z|D|^s\Theta\|_{L^2}
+\|J^{-\f12}_\e([|D|^s,\wt\z]\p_y\Theta)\|_{L^2}\\
&\lesssim\|J^{-\f12}_\e\p_y\wt\z\|_{H^s}\|J^{-\f12}_\e\Theta\|_{H^s}
+\|\p_y\wt\z\|_{L^\infty}\||D|^s\Theta\|_{L^2}
+\|J^{\f12}_\e\wt\z\|_{H^s}\|J^{-\f12}_\e\p_y\Theta\|_{H^{s-1}},
\end{aligned}\eeno
where we used \eqref{product estimate 1} and \eqref{commutator D1} in the last inequality. Then we obtain
\beq\label{F 6}
\|J^{-\f12}_\e\p_y([|D|^s,\wt\z]\Theta)\|_{L^2}
\lesssim\bigl(\|J^{-\f12}_\e\p_y\wt\z\|_{H^s}+\|J^{\f12}_\e\wt\z\|_{H^s}\bigr)
\bigl(\|\Theta\|_{H^s}+\|J^{-\f12}_\e\p_y\Theta\|_{H^{s-1}}\bigr).
\eeq

Thanks to \eqref{F 4}, \eqref{F 5} and \eqref{F 6}, we get
\beq\label{F 7}\begin{aligned}
&|\bigl(A(D)([|D|^s,\wt\z]\Theta)\,\big|\, |D|^s P\bigr)_{L^2}|\\
&
\lesssim\bigl(\|(\p_x\wt\z,J^{-\f12}_\e\p_y\wt\z)\|_{H^s}+\|J^{\f12}_\e\wt\z\|_{H^s}\bigr)
\bigl(\|\Theta\|_{H^s}+\|J^{-\f12}_\e\p_y\Theta\|_{H^{s-1}}\bigr)\|J^{\f12}_\e P\|_{\dH^s}.
\end{aligned}\eeq

Similarly, we have
\beq\label{F 8}\begin{aligned}
&|\bigl(A(D)([|D|^s,\wt\z] P)\,\big|\, |D|^s\Theta\bigr)_{L^2}|\\
&
\lesssim\bigl(\|(\p_x\wt\z,J^{-\f12}_\e\p_y\wt\z)\|_{H^s}+\|J^{\f12}_\e\wt\z\|_{H^s}\bigr)
\bigl(\|P\|_{H^s}+\|J^{-\f12}_\e\p_y P\|_{H^{s-1}}\bigr)\|J^{\f12}_\e\Theta\|_{\dH^s}.
\end{aligned}\eeq

For term $\bigl([A(D),\wt\z\bigr] |D|^s \Theta\,\big|\, |D|^s P\bigr)_{L^2}$, using \eqref{commutator A3}, we get
{\beno\begin{aligned}
|\bigl([A(D),\wt\z\bigr] |D|^s \Theta\,\big|\, |D|^s P\bigr)_{L^2}|
&\lesssim\|J^{-\f12}_\e\bigl([A(D),\wt\z\bigr] |D|^s \Theta\bigr)\|_{L^2}\|J^{\f12}_\e|D|^s P\|_{L^2}\\
&\lesssim\|\wt\z\|_{H^s}\||D|^s \Theta\|_{L^2}\|J^{\f12}_\e|D|^s P\|_{L^2},
\end{aligned}
\eeno}
which yields to
\beq\label{F 9}
|\bigl([A(D),\wt\z\bigr] |D|^s \Theta\,\big|\, |D|^s P\bigr)_{L^2}|
\lesssim\|\wt\z\|_{H^s}\|\Theta\|_{\dH^s}\|J^{\f12}_\e P\|_{\dH^s}.
\eeq

Combining \eqref{F 7}, \eqref{F 8} and \eqref{F 9}, we arrive at
{\beq\label{F 10}\begin{aligned}
|\cS_2|&\lesssim\bigl(\|(\p_x\wt\z,J^{-\f12}_\e\p_y\wt\z)\|_{H^s}+\|J^{\f12}_\e\wt\z\|_{H^s}\bigr)\cdot
\bigl(\|J^{\f12}_\e P\|_{H^s}+\|J^{-\f12}_\e\p_y P\|_{H^{s-1}}\bigr)\cdot\\
&\qquad\cdot\bigl(\|J^{\f12}_\e\Theta\|_{H^s}+\|J^{-\f12}_\e\p_y\Theta\|_{H^{s-1}}\bigr)
+|\cQ_1(\wt\theta,\wt p)|+|\cQ_1(\wt p,\wt\theta)|.
\end{aligned}\eeq}

Due to the ansatz \eqref{ansatz a} and the notation $\wt\z\eqdefa\frac{1+\f\e2\z}{2+\f\e2\z} \z$, similar arguments as \eqref{L10b} yields to
\beq\label{F 11}
\|(\p_x\wt\z,J^{-\f12}_\e\p_y\wt\z)\|_{H^s}+\|J^{\f12}_\e\wt\z\|_{H^s}
\lesssim\|(\z_x,J^{-\f12}_\e\z_y)\|_{H^s}+\|J^{\f12}_\e\z\|_{H^s}.
\eeq

Since $\Theta=A^{-1}(D) \p_x^2 \wt{\theta}$, we have
\beq\label{F 12}
\sqrt\e\bigl(\|J^{\f12}_\e\Theta\|_{H^s}+\|J^{-\f12}_\e\p_y\Theta\|_{H^{s-1}}\bigr)
\lesssim\sqrt\e\|\p_x\wt\theta\|_{H^s}
\lesssim\|J^{\f12}_\e\na\wt\theta\|_{H^{s-1}}.
\eeq
Similarly, we have
\beq\label{F 13}
\sqrt\e\bigl(\|J^{\f12}_\e P\|_{H^s}+\|J^{-\f12}_\e\p_yP\|_{H^{s-1}}\bigr)
\lesssim\sqrt\e\|\p_x\wt p\|_{H^s}
\lesssim\|J^{\f12}_\e\wt p\|_{H^s}.
\eeq

By virtue of \eqref{F 2}, \eqref{F 3}, \eqref{F 11}, \eqref{F 12} and \eqref{F 13}, we deduce from \eqref{F 10} that
\beq\label{F 14}
\e|\cS_2|\lesssim\bigl(\|(\z_x,J^{-\f12}_\e\z_y)\|_{H^s}+\|J^{\f12}_\e\z\|_{H^s}\bigr)\|J^{\f12}_\e\na\wt\theta\|_{H^{s-1}}\|J^{\f12}_\e\wt p\|_{H^s}.
\eeq

{\bf iii).} Thanks to \eqref{F 1} and \eqref{F 14}, we finally obtain by using \eqref{estimate for wt theta} that
\beno\begin{aligned}
|\cS|&\lesssim\e\bigl(\|(\z_x,J^{-\f12}_\e\z_y)\|_{H^s}+\|J^{\f12}_\e\z\|_{H^s}\bigr)\|J^{\f12}_\e\na\wt\theta\|_{H^{s-1}}\|J^{\f12}_\e\wt p\|_{H^s}\\
&\lesssim\e\bigl(\|(\z_x,J^{-\f12}_\e\z_y)\|_{H^s}+\|J^{\f12}_\e\z\|_{H^s}\bigr)\cdot\bigl(\|J^{\f12}_\e p\|_{H^s}^2+\|J^{\f12}_\e\theta\|_{H^s}^2\bigr),
\end{aligned}\eeno
which along with the definitions of energy functionals (see \eqref{lower order functional} and \eqref{highest order functional}) implies
\beq\label{estimate for S}
|\cS|\lesssim\e E_s(t)^{\f32}.
\eeq

\smallskip

{\it Step 2.2.2. Estimate of $\cT$.} 
 For term $\bigl(\vv V\cdot \na\wt{p} \,|\, \wt{p} \bigr)_{\dH^s}$, we have
\beno\begin{aligned}
\bigl(\vv V\cdot \na\wt{p} \,|\, \wt{p} \bigr)_{\dH^s}
&=\bigl(|D|^s(\vv V\cdot \na\wt{p}) \,|\,|D|^s\wt{p} \bigr)_{L^2}\\
&=\bigl([|D|^s,\vv V]\cdot \na\wt{p}\,|\,|D|^s\wt{p} \bigr)_{L^2}+\bigl(\vv V\cdot \na|D|^s\wt{p} \,|\,|D|^s\wt{p} \bigr)_{L^2}\\
&=\bigl([|D|^s,\vv V]\cdot \na\wt{p}\,|\,|D|^s\wt{p} \bigr)_{L^2}-\f12\bigl(\div\vv V\,|D|^s\wt{p} \,|\,|D|^s\wt{p} \bigr)_{L^2}.
\end{aligned}\eeno
Due to \eqref{commutator D} and $s>3$, we get
\beq\label{T 1}
|\bigl(\vv V\cdot \na\wt{p} \,|\, \wt{p} \bigr)_{\dH^s}|
\lesssim\|[|D|^s,\vv V]\cdot \na\wt{p}\|_{L^2}\|\wt p\|_{\dH^s}
+\|\div\vv V\|_{L^\infty}\|\wt p\|_{\dH^s}^2\lesssim\|\vv V\|_{H^s}\|\na\wt p\|_{H^{s-1}}^2.
\eeq
Similar estimate holds for term $\bigl(\vv V\cdot \na\wt\theta \,|\, \wt\theta\bigr)_{\dH^s}$. Then we obtain
\beno
|\cT|\lesssim\e\|\vv V\|_{H^s}\bigl(\|\na\wt p\|_{H^{s-1}}^2+\|\na\wt\theta\|_{H^{s-1}}^2\bigr),
\eeno
which along with \eqref{estimate for wt theta}, \eqref{lower order functional} and \eqref{highest order functional} gives rise to
\beq\label{estimate for T}
|\cT|\lesssim\e E_s(t)^{\f32}.
\eeq

\smallskip

{\it Step 2.2.3. Estimate of $\cN$.} Thanks to Lemma \ref{lem for nonlinear term 2}, we have
\beno\begin{aligned}
|\cN|&\lesssim\e\|J^{-\f12}_\e N_{\wt{p}}\|_{\dH^s} \|J^{\f12}_\e\wt{p}\|_{\dH^s}
+\e\|J^{-\f12}_\e N_{\wt{\theta}}\|_{\dH^s} \|J^{\f12}_\e\wt\theta\|_{\dH^s}\\
&
\lesssim\e\bigl(\|\z\|_{H^s}^2+\|\vv V\|_{H^s}^2+\|J_\e^{\f12} p\|_{H^s}^2+\|J_\e^{\f12} \theta\|_{H^s}^2\bigr)
\cdot\bigl(\|J^{\f12}_\e\wt{p}\|_{\dH^s}+\|J^{\f12}_\e\wt\theta\|_{\dH^s}\bigr),
\end{aligned}\eeno
which together with \eqref{estimate for wt theta}, \eqref{lower order functional} and \eqref{highest order functional} yields to
\beq\label{estimate for N}
|\cN|\lesssim\e E_s(t)^{\f32}.
\eeq

\smallskip

{\it Step 2.2.4. Energy estimate for $(\wt p,\wt\theta)$.} Combining  estimates  \eqref{estimate for S}, \eqref{estimate for T} and \eqref{estimate for N}, we deduce from \eqref{F 0} that
\beq\label{highest order energy estimate}
\f12\f{d}{dt}\wt{E}_{s,h}(t)\lesssim \e E_s(t)^{\f32}.
\eeq

\medskip

{\bf Step 3. The total energy estimates.} Thanks to \eqref{lower order energy estimate} and \eqref{highest order energy estimate}, we arrive at
\beno
\f12\f{d}{dt}\bigl(2c_5E_{s,l}(t)+\wt{E}_{s,h}(t)\bigr)\lesssim \e E_s(t)^{\f32},
\eeno
which yields to
\beno
2c_5E_{s,l}(t)+\wt{E}_{s,h}(t)\lesssim 2c_5E_{s,l}(0)+\wt{E}_{s,h}(0)+\e t\max_{\tau\in[0,T_0/\e]}E_s(\tau)^{\f32},\quad\forall\, t\in[0,T_0/\e].
\eeno
By virtue of \eqref{equivalence for highest functional}, \eqref{lower order functional} and \eqref{highest order functional}, there exist $C_1,C_2>0$ such that
\beno
E_s(t)\leq C_1 E_s(0)+C_2\e t\max_{\tau\in[0,T_0/\e]}E_s(\tau)^{\f32},\quad\forall\, t\in[0,T_0/\e].
\eeno
This is the desired total energy estimate \eqref{total energy estimate}. The proposition is proved.
\end{proof}

\section{Proof of Theorem \ref{main theorem} for case $a=f=g= -\frac{1}{6},b=d=e=\frac{1}{2},c=-\frac{1}{2}$}
\subsection{ The reduction of \eqref{WTB 2} for case 2.}
For case $a=f=g= -\frac{1}{6},b=d=e=\frac{1}{2},c=-\frac{1}{2}$, system \eqref{WTB 2} reads
\beq\label{WTB case 4}\left\{\begin{aligned}
	&(1-\f\e2  \p_x^2)v_t+(1-\f\e6  \p_x^2)\z_x+\e vv_x+\e wv_y+\f\e2 \z \z_x=0,\quad t>0,\, (x,y)\in\R^2,\\
	&(1-\f\e2  \p_x^2)w_t+(1-\f\e6  \p_x^2)\z_y+\e v w_x+\e w w_y+ \f\e2\z \z_y=0,\\
	&(1-\f{\e}{2} \p_x^2)\z_t+(1-\f{\e}{2} \p_x^2)v_x+(1-\f\e6  \p_x^2)w_y+\e  v \z_x+\e w\z_y+ \f\e2 \z v_x + \f\e2\z w_y=0.
\end{aligned}\right.\eeq

Recall that the eigenvalue of the linearized system of \eqref{WTB case 4} is $\lambda_{2,\pm} = \pm i \Lam_2(\xi)$ with
\beq\label{eigen value 2}
 \Lam_2(\xi) = \left(  \xi_1^2 \cdot \frac{1+\f\e6 \xi_1^2}{1+\f\e2 \xi_1^2} + \xi_2^2\cdot\left(  \frac{1+\f\e6 \xi_1^2}{1+\f\e2 \xi_1^2} \right)^2\right)^{\f12}
 \sim|\xi|.
\eeq

In this section, we shall use the notations
\beq\label{def of J Y K} 
J_\e=J_\e(D_x)=1-\f\e2\p_x^2,\quad Y_\e=Y_\e(D_x)=1-\f\e6\p_x^2,\quad K_\e=J_\e Y_\e^{-1}=1-\f\e2Y_\e^{-1}\p_x^2,
\eeq
and $B(D)=K_\e\Lambda_2(D)$ with
\beq\label{def of B}
B(\xi)=\Bigl(\f{1+\f{\e}{2}\xi_1^2}{1+\f{\e}{6}\xi_1^2}\xi_1^2+\xi_2^2\Bigr)^{\f12}
=\Bigl(K_\e(\xi_1)\xi_1^2+\xi_2^2\Bigr)^{\f12}\sim |\xi|.
\eeq
We remark that the associated operator $B(D)$ plays similar role as $A(D)$ for case 1.

To symmetrize the linear part of system \eqref{WTB case 4}, we introduce the following new unknowns:
\beq\label{new unknowns 2}
p\eqdefa v_x+K^{-1}_\e w_y,\quad\theta\eqdefa \Lambda_2(D)\z.
\eeq
Due to curl-free condition $v_y=w_x$, we get
\beq\label{v w in terms of p theta 2}
\vv V=-\na B^{-2}(D) K_\e p,\quad\z=K_\e B^{-1}(D)\theta,
\eeq
where we used notation $\vv V=(v,w)^T$ for simplicity.

Due to \eqref{def of B}, \eqref{new unknowns 2}, \eqref{v w in terms of p theta 2}, it is easy to get the following equivalent relation.

\begin{lemma}\label{lem for equivalent 1}
Let $(\vv V,\z)=(v,w,\z)$ be smooth enough functions satisfying $v_y=w_x$, and $(p,\theta)$ be defined in \eqref{new unknowns 2}. There holds for any $s\geq0$,
\beq\label{equivalent 1}
\|J^{\f12}_\e p\|_{H^s}+\|J^{\f12}_\e \theta\|_{H^s}\sim\|J^{\f12}_\e\na v\|_{H^s}+\|J^{\f12}_\e\na w\|_{H^s}+\|J^{\f12}_\e\na\z\|_{H^s}.
\eeq
\end{lemma}

Now, we are in the position to derive the evolution equations for $(p,\theta)$. 

{\bf Evolution equation of $p$.} Thanks to the first two equations of \eqref{WTB case 4}, we have
\beno\begin{aligned}
J_\e\bigl(p_t-\Lambda_2(D)\theta\bigr)
&=-\e \p_x\bigl(vv_x+wv_y+\f12 \z \z_x\bigr)-\e K^{-1}_\e\p_y\bigl(v w_x+w w_y+ \f12\z \z_y\bigr)\\
&=-\e \bigl\{\underbrace{vv_{xx}+wv_{yx}+\f12 \z \z_{xx}+K^{-1}_\e\bigl(v w_{xy}+w w_{yy}+ \f12\z \z_{yy}\bigr)}_{I_1}\bigr\}-\e N_{p,1},
\end{aligned}\eeno
where
\beq\label{def of Np1 2}
N_{p,1}\eqdefa v_xv_x+w_xv_y+\f12 \z_x \z_x+K^{-1}_\e\bigl(v_y w_x+w_y w_y+ \f12\z_y \z_y\bigr).
\eeq

For $I_1$, using condition $v_y=w_x$ and relation $K_\e\p_x^2+\p_y^2=-B^2(D)$, we get
\beno\begin{aligned}
I_1&=vv_{xx}+ww_{xx}+\f12 \z \z_{xx}+K^{-1}_\e\bigl(v v_{yy}+w w_{yy}+ \f12\z \z_{yy}\bigr)\\
&=-K^{-1}_\e\bigl(vB^2(D)v+wB^2(D)w+\f12\z B^2(D)\z\bigr)+N_{p,2},
\end{aligned}\eeno
where 
\beq\label{def of Np2 2}
N_{p,2}=-[K^{-1}_\e,v]K_\e v_{xx}-[K^{-1}_\e,w]K_\e w_{xx}-\f12[K^{-1}_\e,\z]K_\e \z_{xx}.
\eeq

Using \eqref{v w in terms of p theta 2} and notation $\vv V=(v,w)^T$, we get 
\beno
I_1=K^{-1}_\e(\vv V\cdot\na K_\e p)-\f12K^{-1}_\e(\z B(D)K_\e\theta)+N_{p,2},
\eeno
Then we derive the evolution equation of $p$ as follows:
\beq\label{equation for p 2}
J_\e\bigl(p_t-\Lambda_2(D)\theta\bigr)=-\e K^{-1}_\e(\vv V\cdot\na K_\e p)+\f\e2K^{-1}_\e(\z B(D)K_\e\theta)+\e N_p,
\eeq
where $N_p=-(N_{p,1}+N_{p,2})$ with $N_{p,1}$ and $N_{p,2}$ being defined in \eqref{def of Np1 2} and  \eqref{def of Np2 2}.

\smallskip

{\bf Evolution equation of $\theta$.} Due to the third equation of \eqref{WTB case 4} and \eqref{new unknowns 2}, we get 
\beno
J_\e\bigl(\theta_t+\Lambda_2(D)p\bigr)
=-\e\Lambda_2(D)\bigl(\vv V\cdot\na\z+ \f12 \z\div\vv V\bigr).
\eeno

Thanks to \eqref{def of J Y K} and \eqref{def of B}, we have
\beno
B^2(D)=-\Delta+\f\e3Y_\e^{-1}\p_x^4,\quad K_\e=J_\e Y_\e^{-1},
\eeno
from which and \eqref{v w in terms of p theta 2}, we deduce that
\beno
\div\vv V=-\Delta B^{-2}(D)K_\e p=K_\e p-\f\e3J_\e Y_\e^{-2}B^{-2}(D)\p_x^4 p.
\eeno
 It hints 
\beno\begin{aligned}
\Lambda_2(D)(\z\div\vv V)&=\Lambda_2(D)(\z K_\e p)-\f\e3 J_\e\Lambda_2(D)\bigl(\z Y_\e^{-2}B^{-2}(D)\p_x^4p\bigr)\\
&\qquad-\f{\e^2}{6}\Lambda_2(D)\bigl([\p_x^2,\z]Y_\e^{-2}B^{-2}(D)\p_x^4p\bigr),
\end{aligned}\eeno
where we used the fact that $J_\e=1-\f\e2\p_x^2$ in the last inequality.

Then we obtain by using \eqref{v w in terms of p theta 2} and $J_\e\Lambda_2(D)=Y_\e B(D)$ that
\beq\label{equation for theta 2}\begin{aligned}
J_\e\bigl(\theta_t+\Lambda_2(D)p\bigr)
&=-\e \Lambda_2(D)(\vv V\cdot\na B^{-1}(D)K_\e\theta)-\f\e2 \Lambda_2(D)(\z K_\e p)\\
&\qquad+\f{\e^2}{6}Y_\e B(D)\bigl[\z Y_\e^{-2}B^{-2}(D)\p_x^4p\bigr]+\e N_\theta,
\end{aligned}\eeq
where 
\beq\label{def of N theta 2}
N_\theta=\f{\e^2}{12}\Lambda_2(D)\bigl([\p_x^2,\z]Y_\e^{-2}B^{-2}(D)\p_x^4p\bigr).
\eeq

Combining \eqref{equation for p 2} and \eqref{equation for theta 2}, we derive the evolution system of $(p,\theta)$ as follows:
\beq\label{equation of p and theta in case 2}
	\left\{ \begin{aligned}
J_\e\bigl(p_t-\Lambda_2(D)\theta\bigr)&=-\e K^{-1}_\e(\vv V\cdot\na K_\e p)+\f\e2K^{-1}_\e(\z B(D)K_\e\theta)+\e N_p,\\
	J_\e\bigl(\theta_t+\Lambda_2(D)p\bigr)
&=-\e \Lambda_2(D)(\vv V\cdot\na B^{-1}(D)K_\e\theta)-\f\e2 \Lambda_2(D)(\z K_\e p)\\
&\qquad+\f{\e^2}{6}Y_\e B(D)\bigl(\z Y_\e^{-2}B^{-2}(D)\p_x^4p\bigr)+\e N_\theta.
	\end{aligned} \right.
\eeq
\begin{remark}
The following Lemma \ref{lem for nonlinear term 1 case 2} displays that $\e N_p$ and $\e N_\theta$ are low-order terms of order $O(\e)$. The quadratic term $\f{\e^2}{6}Y_\e B(D)\bigl(\z Y_\e^{-2}B^{-2}(D)\p_x^4p\bigr)$ is of $O(\e^\f12)$ if we treat it as a lower order term. Then the local time existence on time scalar $O(\e^{-\f12})$ for system \eqref{equation of p and theta in case 2} follows from the classical hyperbolic energy method. To enlarge the existence time scale up to $O(1/\e)$, term $\f{\e^2}{6}Y_\e B(D)\bigl(\z Y_\e^{-2}B^{-2}(D)\p_x^4p\bigr)$ should be treated as a part of the quasilinear terms.
\end{remark}

\begin{lemma}\label{lem for nonlinear term 1 case 2} Let $s>3$ and $(v,w,\z)$ be a smooth enough solution of \eqref{WTB case 4}. There holds
\beq\label{estimate for N p theta 2}
\|N_p\|_{H^s}+\|J^{-\f12}_\e N_\theta\|_{H^s}
\lesssim\|\na v\|_{H^s}^2+\|\na w\|_{H^s}^2+\|\na\z\|_{H^s}^2.
\eeq
\end{lemma}
\begin{proof} We derive the bounds of $N_p$ and $N_\theta$ one by one.

{\bf Bound of $\|N_p\|_{H^s}$.} Firstly, using \eqref{tame} and the fact that $K_\e(\xi_1)=\f{1+\f\e2\xi_1^2}{1+\f\e6\xi_1^2}\sim 1$, we derive from \eqref{def of Np1 2} that
\beno\begin{aligned}
\|N_{p,1}\|_{H^s}&\lesssim\|v_xv_x+w_xv_y+\f12 \z_x \z_x\|_{H^s}+\|K^{-1}_\e\bigl(v_y w_x+w_y w_y+ \f12\z_y \z_y\bigr)\|_{H^s}\\
&\lesssim\|\na v\|_{H^s}^2+\|\na w\|_{H^s}^2+\|\na\z\|_{H^s}^2.
\end{aligned}\eeno

For $N_{p,2}$ in \eqref{def of Np2 2}, using \eqref{commutator K} for $K^{-1}_\e$, we have
\beno\begin{aligned}
\|N_{p,2}\|_{H^s}&\lesssim\|[K^{-1}_\e,v]K_\e v_{xx}\|_{H^s}+\|[K^{-1}_\e,w]K_\e w_{xx}\|_{H^s}+\|[K^{-1}_\e,\z]K_\e \z_{xx}\|_{H^s}\\
&\lesssim \|v_x\|_{H^s}^2+\|w_x\|_{H^s}^2+\|\z_x\|_{H^s}^2.
\end{aligned}\eeno

Then we obtain 
\beq\label{bound of N p 2}
\|N_p\|_{H^s}\lesssim\|\na v\|_{H^s}^2+\|\na w\|_{H^s}^2+\|\na\z\|_{H^s}^2.
\eeq

\smallskip

{\bf Bound of $\|J_\e^{-\f12}N_\theta\|_{H^s}$.} Thanks to \eqref{product estimate 1} and the fact that $\Lambda_2(\xi)\sim|\xi|$, we derive from \eqref{def of N theta 2} that
\beno\begin{aligned}
\|J^{-\f12}_\e N_\theta\|_{H^s}&\lesssim\e^2\|J^{-\f12}_\e\bigl(\p_x^2\z\cdot Y_\e^{-2}B^{-2}(D)\p_x^4p\bigr)\|_{H^{s+1}}+\e^2\|J^{-\f12}_\e\bigl(\p_x\z\cdot Y_\e^{-2}B^{-2}(D)\p_x^5p\bigr)\|_{H^{s+1}}\\
&\lesssim\sqrt\e\cdot\sqrt\e\|J^{-\f12}_\e\p_x^2\z\|_{H^s}\cdot\e\|J^{-\f12}_\e Y_\e^{-2}B^{-2}(D)\p_x^4p\|_{H^s}\\
&\qquad+\sqrt\e\cdot\|J^{-\f12}_\e\p_x\z\|_{H^s}\cdot\e^{\f32}\|J^{-\f12}_\e Y_\e^{-2}B^{-2}(D)\p_x^5p\|_{H^s}\\
&\lesssim\sqrt\e\|\p_x\z\|_{H^s}\|p\|_{H^s}.
\end{aligned}\eeno
Due to \eqref{new unknowns 2}, \eqref{v w in terms of p theta 2} and  condition $v_y=w_x$, it is easy to get 
\beno
\|p\|_{H^s}^2\sim\|\na v\|_{H^s}^2+\|\na w\|_{H^s}^2.
\eeno
Then we obtain
\beq\label{bound of N theta 2}
\|J^{-\f12}_\e N_\theta\|_{H^s}\lesssim\|\na v\|_{H^s}^2+\|\na w\|_{H^s}^2+\|\na\z\|_{H^s}^2.
\eeq

\smallskip

Combining \eqref{bound of N p 2} and \eqref{bound of N theta 2}, we obtain \eqref{estimate for N p theta 2}. The lemma is proved.
\end{proof}

\subsection{Symmetrization of  \eqref{equation of p and theta in case 2} with good unknowns.} To get the large time existence on time scale $O(1/\e)$, we
symmetrize system \eqref{equation of p and theta in case 2} by introducing  good unknowns.

Firstly, for any $r\geq0$, we define Hilbert spaces $J_\e^{\f12}H^r(\R^2)$ and $J_\e^{-\f12}H^r(\R^2)$ as follows:
\beq\label{def of Sobolev space}\begin{aligned}
&J_\e^{\f12}H^r(\R^2)=\{f\in H^r(\R^2)\,|\,J_\e^{\f12}f\in H^r(\R^2)\},\\
&J_\e^{-\f12}H^r(\R^2)=\{f\in\cS'(\R^2)\,|\,J_\e^{-\f12}f\in H^r(\R^2)\},
\end{aligned}\eeq
associated with the norms $\|f\|_{J_\e^{\f12}H^r}\eqdefa\|J_\e^{\f12}f\|_{H^r}$ and $\|f\|_{J_\e^{-\f12}H^r}\eqdefa\|J_\e^{-\f12}f\|_{H^r}$ respectively. We also denote by $J_\e^{\f12}L^2(\R^2)=J_\e^{\f12}H^0(\R^2)$ and $J_\e^{-\f12}L^2(\R^2)=J_\e^{-\f12}H^0(\R^2)$.

\subsubsection{Bounded operators.} In this subsection, we assume that $s>3$, $\e\in(0,1)$, $T^*>0$ and $\z\in C^1\bigl((0,T^*];J_\e^{\f12}H^s(\R^2)\bigr)\cap C\bigl([0,T^*];J_\e^{\f12}H^{s+1}(\R^2)\bigr)$. By virtue of \eqref{tame}, \eqref{product estimate 1} and Sobolev embedding theorem, it is easy to check that there exists an universal constant $C_s>1$ such that 
\beq\label{product J}
\|\z f\|_{X}\leq C_s\|J_\e^{\f12}\z\|_{H^{s+1}}\|f\|_{X},\quad\forall\, f\in X,\,\forall t\in[0,T^*],
\eeq
where 
\beq\label{X}
X=J_\e^{\f12}H^s(\R^2),\,J_\e^{\f12}L^2(\R^2),\,H^{s+1}(\R^2),\,L^2(\R^2),\,J_\e^{-\f12}H^{s+1}(\R^2)\,\,\text{or}\,\,J_\e^{-\f12}L^{2}(\R^2).
\eeq

We assume that 
\beq\label{ansatz 1 a}
\sup_{t\in[0,T^*]}\sqrt\e\|J_\e^{\f12}\z(t,\cdot)\|_{H^{s+1}}\leq\f{1}{C_s},
\eeq
Due to \eqref{product J} and \eqref{ansatz 1 a}, it is easy to check that
\beno
\|J_\e^{\f12}\bigl(\f\e4\z Y^{-1}_\e f\bigr)\|_{H^s}
\leq \f\e4 C_s\|J_\e^{\f12}\z\|_{H^{s+1}}\|J_\e^{\f12}Y^{-1}_\e f\|_{H^s}\leq\f{\sqrt\e}4\|J_\e^{\f12}f\|_{H^s},\quad\forall\, f\in J_\e^{\f12}H^s(\R^2).
\eeno
It turns out that $\f\e4\z Y^{-1}_\e$ is a bounded linear operator from $J_\e^{\f12}H^s(\R^2)$ to $J_\e^{\f12}H^s(\R^2)$ with the norm
\beno
\|\f\e4\z Y^{-1}_\e\|\eqdefa\sup_{\|J_\e^{\f12} f\|_{H^s}=1}\|J_\e^{\f12}\bigl(\f\e4\z Y^{-1}_\e f\bigr)\|_{H^s}\leq\f{\sqrt\e}4.
\eeno
Similarly, $\f\e4\z Y^{-1}_\e$ is also a bounded linear operator on the space   $X$ defined in \eqref{X} with 
\beq\label{operator norm 1}
\|\f\e4\z Y^{-1}_\e\|\eqdefa\sup_{\|f\|_{X}=1}\|\f\e4\z Y^{-1}_\e f\|_{X}\leq\f{\sqrt\e}4<\f14.
\eeq
Here and in what follows, we use the notation $\|\cdot\|$ to denote the operator norm without pointing out the space.

By virtue of \eqref{operator norm 1} and von Neumann theorem,  it appears that
$( 2+ \f\e2\z Y_{\e}^{-1} )^{-1}(=\f12( 1+\f\e4\z Y_{\e}^{-1} )^{-1})$ is a bounded linear operator from $J_\e^{\f12}H^s(\R^2)$ to $J_\e^{\f12}H^s(\R^2)$ and
\beq\label{inverse operator}
( 2+ \f\e2\z Y_{\e}^{-1} )^{-1}=\f12\sum_{n=0}^\infty (-1)^n (\f\e4\z Y_{\e}^{-1})^n
\quad\text{with}\quad
\|( 2+ \f\e2 \z Y_{\e}^{-1} )^{-1}\|\leq\f{1}{2(1-\|\f\e4\z Y^{-1}_\e\|)}
\leq1.
\eeq
Moreover, $(2+\f\e2\z Y_{\e}^{-1} )^{-1}$ is a bounded linear operator on each space stated in \eqref{X} and satisfying \eqref{inverse operator}. By interpolation theory, $(2+\f\e2\z Y_{\e}^{-1} )^{-1}$ is also a bounded linear operator on $J_\e^{\f12} H^r(\R^2),\,J_\e^{\f12} H^{r+1}(\R^2)$, $H^{r+1}(\R^2)$ with $r\in[0,s]$ and satisfying \eqref{inverse operator}.

Denoting by 
\beq\label{def of gamma e}\begin{aligned}
&\Gamma_\e(\z, D_x)\eqdefa( 2+ \f\e2\z Y_{\e}^{-1})^{-1}(\z\cdot),\quad
\gamma_\e(\z,D_x) \eqdefa( 1+ \f\e2\z Y_\e^{-1}) ( 2+ \f\e2\z Y_{\e}^{-1})^{-1} (\z\cdot),
\end{aligned}\eeq
it is easy to get
\beq\label{identity}
\gamma_\e(\z,D_x) +\Gamma_\e(\z, D_x)=\z,\quad \gamma_\e(\z,D_x)=( 1+ \f\e2\z Y_\e^{-1})\Gamma_\e(\z,D_x).
\eeq

\begin{lemma}\label{lem for operator gamma}
Let $s>3$, $r\in[0,s]$, $\e\in(0,1)$, $T^*>0$ and $\z\in C^1\bigl((0,T^*];J_\e^{\f12}H^{s}(\R^2)\bigr)\cap C\bigl([0,T^*];J_\e^{\f12}H^{s+1}(\R^2)\bigr)$ satisfy \eqref{ansatz 1 a}. Then for any $t\in[0,T^*]$,

(1). $\Gamma_\e(\z,D_x)$ and $\gamma_\e(\z,D_x)$ are bounded linear operators on spaces $J_\e^{\f12}H^r(\R^2)$, $H^{r+1}(\R^2)$, $L^2(\R^2)$, $J_\e^{-\f12}H^{r+1}(\R^2)$ and $J_\e^{-\f12}L^2(\R^2)$, and the corresponding operator norms satisfy
\beq\label{operator norm 2}
\|\Gamma_\e(\z,D_x)\|\leq C_s\|J_\e^{\f12}\z\|_{H^{s+1}},
\quad\|\gamma_\e(\z,D_x)\|\leq\f32 C_s\|J_\e^{\f12}\z\|_{H^{s+1}};
\eeq

(2). $\Gamma_\e(\z,D_x)$ and $\gamma_\e(\z,D_x)$ are self-adjoint operators on $L^2(\R^2)$, i.e., 
\beq\label{adjoint}\begin{aligned}
&\bigl(\Gamma_\e(\z,D_x) f\,\big|\,g\bigr)_{L^2}=\bigl( f\,\big|\,\Gamma_\e(\z,D_x)g\bigr)_{L^2},\quad\forall f, g\in L^2(\R^2),\\
&\bigl(\gamma_\e(\z,D_x) f\,\big|\,g\bigr)_{L^2}=\bigl( f\,\big|\,\gamma_\e(\z,D_x)g\bigr)_{L^2},\quad\forall f, g\in L^2(\R^2).
\end{aligned}\eeq
\end{lemma}
\begin{proof}
(1). Thanks to \eqref{product estimate 1} and \eqref{inverse operator}, we get for any $f\in J^{\f12}_\e H^s(\R^2)$, 
\beno
\|J_\e^{\f12}\bigl(\Gamma_\e(\z,D_x)f\bigr)\|_{H^s}\leq\|( 2+ \f\e2\z Y_{\e}^{-1} )^{-1}\|\cdot\|J_\e^{\f12}(\z f)\|_{H^s}\leq C_s\|J_\e^{\f12}\z\|_{H^s}\|J_\e^{\f12}f\|_{H^s},
\eeno
which gives rise to the first inequality of \eqref{operator norm 2}.

Since $\gamma_\e(\z,D_x)=(1+\f\e2\z Y^{-1}_\e)\Gamma_\e(u,D_x)$, using the first inequality of \eqref{operator norm 2} and \eqref{operator norm 1}, we get the second inequality of \eqref{operator norm 2}. Hence, both $\Gamma_\e(\z,D_x)$ and $\gamma_\e(\z,D_x)$ are bounded linear opeators on $J_\e^{\f12}H^s(\R^2)$.

Similar argument also hold on spaces $J_\e^{\f12}L^2(\R^2)$,  $H^{s+1}(\R^2)$, $L^2(\R^2)$, $J_\e^{-\f12}H^{s+1}(\R^2)$ and $J_\e^{-\f12}L^2(\R^2)$. Since $\|J_\e^{-\f12}\z\|_{H^s}\leq\|\z\|_{H^s}\leq\|J_\e^{\f12}\z\|_{H^s}\leq\|J_\e^{\f12}\z\|_{H^{s+1}}$, the inequalities in \eqref{operator norm 2} hold on spaces stated in \eqref{X}. 

Consequently, the standard interpolation theory shows that $\Gamma_\e(\z,D_x)$ and $\gamma_\e(\z,D_x)$ are also bounded on spaces $J_\e^{\f12}H^r(\R^2)$, $H^{r+1}(\R^2)$ and $J_\e^{-\f12}H^{r+1}(\R^2)$ for any $r\in[0,s]$. And the corresponding  operator norms also satisfy \eqref{operator norm 2}. 

\medskip

(2). By virtue of \eqref{def of gamma e} and \eqref{inverse operator}, we have
\beno
\Gamma_\e(\z,D_x)f=\f12\sum_{n=0}^\infty (-1)^n (\f\e4\z Y_{\e}^{-1})^n(\z f).
\eeno
For any $f,g\in L^2(\R^2)$,
\beno\begin{aligned}
\bigl(\Gamma_\e(\z,D_x) f\,\big|\,g\bigr)_{L^2}
&=\f12\sum_{n=0}^\infty(-1)^n\bigl((\f\e4\z Y_{\e}^{-1})^n(\z f)\,\big|\,g\bigr)_{L^2}
=\f12\sum_{n=0}^\infty(-1)^n\bigl(f\,\big|\,\z(\f\e4 Y_{\e}^{-1}(\z\cdot))^n g\bigr)_{L^2}\\
&=\f12\sum_{n=0}^\infty(-1)^n\bigl(f\,\big|\,(\f\e4\z Y_{\e}^{-1})^n(\z g)\bigr)_{L^2}=\bigl( f\,\big|\,\Gamma_\e(\z,D_x)g\bigr)_{L^2},
\end{aligned}\eeno
which shows that $\Gamma_\e(\z,D_x)$ is a self-adjoint operator on $L^2(\R^2)$. Using \eqref{identity}, we see that $\gamma_\e(\z,D_x)$ is also a self-adjoint operator on $L^2(\R^2)$. The proof of the lemma is complete.
\end{proof}

\begin{lemma}\label{lem for commutator gamma}
Let $s>3$, $T^*>0$ and $\z\in C^1\bigl((0,T^*];J_\e^{\f12}H^{s}(\R^2)\bigr)\cap C\bigl([0,T^*];J_\e^{\f12}H^{s+1}(\R^2)\bigr)$ satisfy \eqref{ansatz 1 a}. There hold
\begin{align}
&\bigl\|J_\e^{\f12}\bigl([\p_t,\Gamma_\e(\z,D_x)]f\bigr)\bigr\|_{H^s}\lesssim \|J_\e^{\f12}\z_t\|_{H^s}\|J_\e^{\f12}f\|_{H^s},\quad\forall f\in J_\e^{\f12}H^s(\R^2),\label{commutator gamma 1}\\
&\bigl\|J_\e^{-\f12}\bigl([\p_x,\na\Gamma_\e(\z,D_x)]f\bigr)\bigr\|_{H^s}
\lesssim \e^{-\f12} \|\z\|_{H^{s+1}}\|J_\e^{-\f12}f\|_{H^s},\quad\forall f\in J_\e^{-\f12}H^s(\R^2),\label{commutator gamma 3b}\\
&\bigl\|[\p_x,\na\Gamma_\e(\z,D_x)]f\bigr\|_{H^s}
\lesssim\e^{-\f12}\|J_\e^{\f12}\z\|_{H^{s+1}}\|f\|_{H^s},\quad\forall f\in H^s(\R^2),\label{commutator gamma 3c}\\
&\bigl\|J_\e^{\f12}\bigl([|D|^s,\Gamma_\e(\z,D_x)]f\bigr)\bigr\|_{L^2}\lesssim \|J_\e^{\f12}\z\|_{H^s}\|J_\e^{\f12}f\|_{H^{s-1}},\quad\forall f\in J_\e^{\f12}H^{s-1}(\R^2).\label{commutator gamma 0}
\end{align}
and
\beq\label{commutator gamma X}
\|[\p_x,\Gamma_\e(\z,D_x)]f\|_{X}\lesssim\|J_\e^{\f12}\z\|_{H^{s+1}}\|f\|_{X},\quad\forall f\in X,
\eeq
where $X=J_\e^{\f12}H^r(\R^2),\, H^r(\R^2)$ or  $J_\e^{-\f12}H^r(\R^2)$ for any $r\in[0,s]$.
Moreover, all  estimates also hold for $\gamma_\e(\z,D_x)$.
\end{lemma}
\begin{proof} Thanks to Lemma \ref{lem for operator gamma}, all the commutators appearing in this lemma are well-defined on the corresponding spaces. Let us prove the commutator estimates one by one.

{\bf 1).} Firstly, due to \eqref{inverse operator} and \eqref{def of gamma e}, we have
\beno
[\p_t,\Gamma_\e(\z,D_x)]f=\f12\sum_{n=1}^\infty(-1)^n[\p_t,(\f\e4\z Y_\e^{-1})^n](\z f) + (2+ \f\e2\z Y_{\e}^{-1})^{-1}(\z_t f),
\eeno
and 
\beno\begin{aligned}
&[\p_t,(\f\e4\z Y_\e^{-1})^n]f=\f\e4\z_t Y_\e^{-1}\bigl((\f\e4\z Y_\e^{-1})^{n-1}(\z f)\bigr)+\f\e4\z Y_\e^{-1}\Bigl(\f\e4\z_t Y_\e^{-1}\bigl((\f\e4\z Y_\e^{-1})^{n-2}(\z f)\bigr)\Bigr)\\
&\qquad\qquad\qquad\qquad+\cdots+(\f\e4\z Y_\e^{-1})^{n-1}\bigl(\f\e4\z_tY_\e^{-1}(\z f)\bigr).
\end{aligned}\eeno

Using \eqref{inverse operator} and \eqref{product J},  we get for any $f\in J_\e^{\f12}H^s(\R^2)$,
\beno
\bigl\|J_\e^{\f12}\bigl((2+ \f\e2\z Y_{\e}^{-1})^{-1}(\z_t f)\bigr)\bigr\|_{H^s}
\leq\|(2+ \f\e2\z Y_{\e}^{-1})^{-1}\|\cdot\|J_\e^{\f12}(\z_t f)\|_{H^s}
\leq C_s\|J_\e^{\f12}\z_t\|_{H^s}\|J_\e^{\f12}f\|_{H^s},
\eeno
while using \eqref{operator norm 1} and \eqref{product J}, we get for any $f\in J_\e^{\f12}H^s(\R^2)$,
\beq\label{R 16}\begin{aligned}
\bigl\|J_\e^{\f12}\bigl([\p_t,(\f\e4\z Y_\e^{-1})^n]f\bigr)\bigr\|_{H^s}
&\leq n\|\f\e4\z_t Y_\e^{-1}\|\cdot\|\f\e4\z Y_\e^{-1}\|^{n-1}\|J_\e^{\f12}(\z f)\|_{H^s}\\
&\leq C_s\e\cdot\f{n}{4^{n}}\cdot\|J_\e^{\f12}\z_t\|_{H^s}\|J_\e^{\f12}\z\|_{H^s}\|J_\e^{\f12}f\|_{H^s},\quad n=1,2,\cdots,
\end{aligned}\eeq
where 
\beno
\|\f\e4\z_t Y_\e^{-1}\|\eqdefa\sup_{\|J_\e^{\f12}f\|_{H^s}=1}\f\e4\|J_\e^{\f12}(\z_t Y_\e^{-1}f)\|_{H^s}\leq\f\e4 C_s\|J_\e^{\f12}\z_t\|_{H^s}.
\eeno

Due to ansatz \eqref{ansatz 1 a}, we obtain for any $f\in J_\e^{\f12}H^s(\R^2)$,
\beno
\bigl\|J_\e^{\f12}\bigl([\p_t,\Gamma_\e(\z,D_x)]f\bigr)\bigr\|_{H^s}
\leq C_s\sqrt\e\sum_{n=1}^\infty \f{n}{4^{n}}\cdot\|J_\e^{\f12}\z_t\|_{H^s}\|J_\e^{\f12}f\|_{H^s} + C_s \| J_\e^\f12 \z_t \|_{H^s} \| J_\e^\f12 f  \|_{H^s}.
\eeno
Since $\sum_{n=1}^\infty \f{n}{4^{n}}<+\infty$, we get
\beq\label{R 0}
\bigl\|J_\e^{\f12}\bigl([\p_t,\Gamma_\e(\z,D_x)]f\bigr)\bigr\|_{H^s}
\lesssim \|J_\e^{\f12}\z_t\|_{H^s}\|J_\e^{\f12}f\|_{H^s}.
\eeq

Using \eqref{identity}, we have
\beno
[\p_t,\gamma_\e(\z,D_x)]f=[\p_t,\z]f-[\p_t,\Gamma_\e(\z,D_x)]f,
\eeno
which along with \eqref{R 0} implies that \eqref{commutator gamma 1} also holds for $\gamma_\e(\z,D_x)$.

{\bf 2).} For \eqref{commutator gamma 3b}, using \eqref{inverse operator} and \eqref{def of gamma e}, we have
\beno\begin{aligned}
&[\p_x,\na\Gamma_\e(\z,D_x)]f =\f12\sum_{n=1}^\infty(-1)^n\Bigl([\p_x,\na(\f\e4\z Y_\e^{-1})^n](\z f)+[\p_x,(\f\e4\z Y_\e^{-1})^n](\na\z f)\\
&\qquad\qquad\qquad\qquad+\bigl[\na(\f\e4\z Y_\e^{-1})^n\bigr](\p_x\z f)\Bigr)+(2+ \f\e2\z Y_\e^{-1})^{-1}(\na\p_x\z f),
\end{aligned}\eeno
and 
\beq\label{R 17}\begin{aligned}
\na(\f\e4\z Y_\e^{-1})^n&=\f\e4\na\z Y_\e^{-1}\bigl((\f\e4\z Y_\e^{-1})^{n-1}\bigr)+\f\e4\z Y_\e^{-1}\Bigl(\f\e4\na\z Y_\e^{-1}\bigl((\f\e4\z Y_\e^{-1})^{n-2}\bigr)\Bigr)+\\
&\qquad\cdots+(\f\e4\z Y_\e^{-1})^{n-1}\bigl(\f\e4\na\z Y_\e^{-1}\bigr).
\end{aligned}
\eeq

(i). For $n=1$, it holds
\beno
[\p_x,\na(\f\e4\z Y_\e^{-1})](\z f)=[\p_x,\f\e4\na\z Y_\e^{-1}](\z f)=\f\e4\na\p_x\z Y_\e^{-1}(\z f),
\eeno
which along with \eqref{product estimate 1} and \eqref{ansatz 1 a} shows
\beq\label{R 18}\begin{aligned}
\bigl\|J_\e^{-\f12}\bigl([\p_x,\na(\f\e4\z Y_\e^{-1})](\z f)\bigr)\bigr\|_{H^s}
&\lesssim\e\|J_\e^{-\f12}\na\p_x\z\|_{H^s}\|J_\e^{-\f12}Y_\e^{-1}(\z f)\|_{H^s}\\
&
\lesssim\sqrt\e\|\z\|_{H^{s+1}}\|J_\e^{-\f12}\z\|_{H^s}\|J_\e^{-\f12}f\|_{H^s}\lesssim\|\z\|_{H^s}\|J_\e^{-\f12}f\|_{H^s}.
\end{aligned}\eeq

(ii). For $n\geq 2$, it holds
\beno\begin{aligned}
&\quad\bigl[\p_x,\f\e4\na\z Y_\e^{-1}\bigl((\f\e4\z Y_\e^{-1})^{n-1}\bigr)\bigr](\z f)\\
&=\bigl[\p_x,\f\e4\na\z  Y_\e^{-1}\bigr]\bigl((\f\e4\z Y_\e^{-1})^{n-1}(\z f)\bigr)
+\f\e4\na\z Y_\e^{-1}\bigl(\bigl[\p_x,(\f\e4\z Y_\e^{-1})^{n-1}\bigr](\z f)\bigr)
\end{aligned}\eeno
which along with \eqref{product estimate 1} and \eqref{operator norm 1} implies
\beno\begin{aligned}
&\quad\bigl\|J_\e^{-\f12}\bigl(\bigl[\p_x,\f\e4\na\z Y_\e^{-1}\bigl((\f\e4\z Y_\e^{-1})^{n-1}\bigr)\bigr](\z f)\bigr)\bigr\|_{H^s}\\
&\lesssim\e\|J_\e^{-\f12}\na\p_x\z\|_{H^s}\|J_\e^{-\f12}\bigl((\f\e4\z Y_\e^{-1})^{n-1}(\z f)\bigr)\|_{H^s}+
\e\|J_\e^{-\f12}\na\z\|_{H^s}\|J_\e^{-\f12}\bigl(\bigl[\p_x,(\f\e4\z Y_\e^{-1})^{n-1}\bigr](\z f)\bigr)\|_{H^s}.
\end{aligned}\eeno
Following similar argument as \eqref{R 16}, we get by using \eqref{ansatz 1 a} that
\beno\begin{aligned}
\bigl\|J_\e^{-\f12}\big(\bigl[\p_x,\f\e4\na\z Y_\e^{-1}\bigl((\f\e4\z Y_\e^{-1})^{n-1}\bigr)\bigr](\z f)\bigr)\bigr\|_{H^s}
&\lesssim\f{n}{4^{n-1}}\cdot\e\|\z\|_{H^{s+1}}\|J_\e^{-\f12}\z\|_{H^s}\|J_\e^{-\f12}f\|_{H^s}\\
&\lesssim\f{n}{4^n}\cdot\sqrt\e\|\z\|_{H^s}\|J_\e^{-\f12}f\|_{H^s}.
\end{aligned}\eeno
The same estimate also holds for the remained terms in \eqref{R 17}. Then we obtain
\beq\label{R 19}
\bigl\|J_\e^{-\f12}\bigl([\p_x,\na(\f\e4\z Y_\e^{-1})^n)](\z f)\bigr)\bigr\|_{H^s}\lesssim\f{n^2}{4^n}\cdot\sqrt\e\|\z\|_{H^s}\|J_\e^{-\f12}f\|_{H^s}
\eeq

(iii) For $n\geq 1$, similar argument as \eqref{R 16} gives rise to
\beq\label{R 19 a}
\bigl\|J_\e^{-\f12}\bigl([\p_x,(\f\e4\z Y_\e^{-1})^n)](\z f)\bigr)\bigr\|_{H^s}\lesssim\f{n}{4^n}\cdot\sqrt\e\|\z\|_{H^s}\|J_\e^{-\f12}f\|_{H^s},
\eeq
\beq\label{R 19 b}
\bigl\|J_\e^{-\f12}\bigl(\bigl[\na(\f\e4\z Y_\e^{-1})^n\bigr](\p_x\z f)\bigr)\bigr\|_{H^s}\lesssim\f{1}{4^n}\cdot\e\|\z\|_{H^{s+1}}\|J_\e^{-\f12}f\|_{H^s}
\eeq

(iv).  Using \eqref{inverse operator} and \eqref{product estimate 1}, we have
\beq\label{R 20}\begin{aligned}
&\bigl\|J_\e^{-\f12}\bigl((2+ \f\e2\z Y_{\e}^{-1})^{-1}(\na\p_x\z f)\bigr)\bigr\|_{H^s}
\leq\|(2+ \f\e2\z Y_{\e}^{-1})^{-1}\|\cdot\|J_\e^{-\f12}(\na\p_x\z f)\|_{H^s}\\
&\qquad\lesssim\|\na\p_x\z\|_{H^s}\|J_\e^{-\f12}f\|_{H^s}
\lesssim\e^{-\f12}\|J_\e^{\f12}\z\|_{H^{s+1}}\|J_\e^{-\f12}f\|_{H^s}.
\end{aligned}\eeq

Combining \eqref{R 18}, \eqref{R 19} and \eqref{R 20}, using the fact that $\sum_{n=1}^\infty(\f{n^2}{4^n}+\f{n}{4^n}+\f{1}{4^n})<+\infty$, we deduce
\beno
\bigl\|J_\e^{-\f12}\bigl([\p_x,\na\Gamma_\e(\z,D_x)]f\bigr)\bigr\|_{H^s}
\lesssim \e^{-\f12} \|\z\|_{H^{s+1}}\|J_\e^{-\f12}f\|_{H^s}.
\eeno
This is exactly \eqref{commutator gamma 3b}. Similar arguments give rise to \eqref{commutator gamma 3c}.

{\bf 3).} Estimate \eqref{commutator gamma 0} follows from \eqref{commutator D} and the similar derivation of \eqref{commutator gamma 4} in the following Lemma \ref{lem for commutator gamma a}. 

At the meantime, following similar arguments as in the proof of \eqref{commutator gamma 1}, we could get  \eqref{commutator gamma X} with $X=J_\e^{\f12}H^r(\R^2),\, H^r(\R^2)$ or  $J_\e^{-\f12}H^r(\R^2)$ when $r\in[0,s]$. 
 We omit the details. The lemma is proved.
\end{proof}

\begin{lemma}\label{lem for commutator gamma a}
Under the assumptions of Lemma \ref{lem for commutator gamma}, there hold
\beq\label{commutator gamma 4}
\bigl\|J_\e^{-\f12}\bigl([B(D),\Gamma_\e(\z,D_x)]f\bigr)\bigr\|_{H^s}\lesssim\|J_\e^{\f12}\z\|_{H^{s+1}}\|J_\e^{-\f12}f\|_{H^s},\quad\forall f\in J_\e^{-\f12}H^s(\R^2),
\eeq
\beq\label{commutator gamma 4b}
\bigl\|[B^{-1}(D),\Gamma_\e(\z,D_x)]f\bigr\|_{\dH^s}\lesssim\|\z\|_{H^s}\||D|^{-1}f\|_{H^{s-1}},\quad\forall |D|^{-1}f\in H^{s-1}(\R^2).
\eeq
and for  $r\in[0,s-1]$, there holds
\beq\label{commutator gamma 4a}
\bigl\|J_\e^{\f12}\bigl([B(D),\Gamma_\e(\z,D_x)]f\bigr)\bigr\|_{H^r}\lesssim\|J_\e^{\f12}\z\|_{H^s}\|J_\e^{\f12}f\|_{H^r},\quad\forall f\in J_\e^{\f12}H^r(\R^2),
\eeq
Moreover, estimates \eqref{commutator gamma 4}, \eqref{commutator gamma 4a} and \eqref{commutator gamma 4b} also hold for $\gamma_\e(\z,D_x)$.
\end{lemma}
\begin{proof}
(1). Due to \eqref{def of gamma e} and \eqref{inverse operator}, we first have
\beno\begin{aligned}
&[B(D),\Gamma_\e(\z,D_x)]f=[B(D),(2+\f\e2\z Y_\e^{-1})^{-1}](\z f)+(2+\f\e2\z Y_\e^{-1})^{-1}\bigl([B(D),\z]f\bigr)\\
&\quad=\f12\sum_{n=1}^\infty(-1)^n\Bigl([B(D),(\f\e4\z Y_\e^{-1})^n](\z f)+(\f\e4\z Y_\e^{-1})^n\bigl([B(D),\z]f\bigr)\Bigr) + (2+\f\e2\z Y_\e^{-1})^{-1}\bigl([B(D),\z]f\bigr),
\end{aligned}\eeno
and
\beno\begin{aligned}
&[B(D),(\f\e4\z Y_\e^{-1})^n ](\z f)=[B(D),\f\e4\z Y_\e^{-1}](\f\e4\z Y_\e^{-1})^{n-1}(\z f)+\f\e4\z Y_\e^{-1}\bigl([B(D),\f\e4\z Y_\e^{-1}](\f\e4\z Y_\e^{-1})^{n-2}(\z f)\bigr)\\
&\qquad\qquad\qquad
+\cdots+(\f\e4\z Y_\e^{-1})^{n-1}\bigl([B(D),\f\e4\z Y_\e^{-1}](\z f)\bigr).
\end{aligned}\eeno

Thanks to \eqref{commutator B1}, we have for any $f\in J_\e^{-\f12}H^s(\R^2)$,
\beno
\|J_\e^{-\f12}\bigl([B(D),\f\e4\z Y_\e^{-1}]f\bigr)\|_{H^s}
\lesssim\e\|J_\e^{-\f12}\z\|_{H^{s+1}}\|J_\e^{-\f12}Y_\e^{-1}f\|_{H^s}
\lesssim\e\|J_\e^{\f12}\z\|_{H^{s+1}}\|J_\e^{-\f12}f\|_{H^s},
\eeno
which shows that $[B(D),\f\e4\z Y_\e^{-1}]$ is a bounded linear operator on $J_\e^{-\f12}H^s(\R^2)$ and
\beq\label{R 7}
\|[B(D),\f\e4\z Y_\e^{-1}]\|\lesssim\e\|J_\e^{\f12}\z\|_{H^{s+1}}.
\eeq

Using \eqref{product estimate 1}, \eqref{commutator B1}, \eqref{operator norm 1} and \eqref{R 7}, we get for any $f\in J_\e^{-\f12}H^s(\R^2)$,
\beno\begin{aligned}
\bigl\|J_\e^{-\f12}\bigl([B(D),(\f\e4\z Y_\e^{-1})^n ](\z f)\bigr)\bigr\|_{H^s}
&\leq n\|[B(D),\f\e4\z Y_\e^{-1}]\|\cdot\|\f\e4\z Y_\e^{-1}\|^{n-1}\cdot\|J_\e^{-\f12}(\z f)\|_{H^s}\\
&\lesssim \f{n}{4^n}\cdot\e\|J_\e^{\f12}\z\|_{H^{s+1}}\|J_\e^{-\f12}\z\|_{H^s}\|J_\e^{-\f12}f\|_{H^s},
\end{aligned}\eeno
which along with \eqref{commutator B1} and ansatz \eqref{ansatz 1 a} imply
\beno\begin{aligned}
\|J_\e^{-\f12}[B(D),\Gamma_\e(\z,D_x)]f\|_{H^s}
&\lesssim\sum_{n=1}^\infty\f{n}{4^n}\cdot\sqrt\e\|J_\e^{\f12}\z\|_{H^{s+1}}\|J_\e^{-\f12}f\|_{H^s} + \| J_\e^\f12 \z \|_{H^{s+1}} \| J_\e^{-\f12}f \|_{H^s}\\
& \lesssim\|J_\e^{\f12}\z\|_{H^{s+1}}\|J_\e^{-\f12}f\|_{H^s}.
\end{aligned}\eeno
This is \eqref{commutator gamma 4}. 

(2). Similar argument as previous leads to
\beq\label{B D 1}
\|[B(D),\Gamma_\e(\z,D_x)]f\|_{H^{s-1}}\lesssim\|\z\|_{H^s}\|f\|_{H^{s-1}},\quad\forall f\in H^{s-1}(\R^2). 
\eeq
And noticing that $B(\xi)\sim|\xi|$, we have
\beno
\bigl\|[B^{-1}(D),\Gamma_\e(\z,D_x)]f\bigr\|_{\dH^s}\sim
\bigl\|B(D)\bigl([B^{-1}(D),\Gamma_\e(\z,D_x)]f\bigr\|_{\dH^{s-1}}.
\eeno
 Since 
 \beno
B(D)\bigl([B^{-1}(D),\Gamma_\e(\z,D_x)]f\bigr)=-[B(D),\Gamma_\e(\z,D_x)]B^{-1}(D)f,
 \eeno
using \eqref{B D 1}, we have
\beno
\|[B^{-1}(D),\Gamma_\e(\z,D_x)]f\|_{\dH^s}
\lesssim\|\z\|_{H^s}\|B^{-1}(D)f\|_{H^{s-1}}.
\eeno
This implies \eqref{commutator gamma 4b}.

(3). Due to \eqref{commutator B3}, it is easy to check that $[B(D),\f\e4\z Y_\e^{-1}]$ is a bounded linear operator on $J_\e^{\f12}H^r(\R^2)$ and
\beno
\|[B(D),\f\e4\z Y_\e^{-1}]\|=\sup_{\|J_\e^{\f12}f\|_{H^r}=1}\|J_\e^{\f12}([B(D),\f\e4\z ]Y_\e^{-1}f)\|_{H^r}\lesssim\e\|J_\e^{\f12}\z\|_{H^s}.
\eeno
Similar arguments  hold as in  the proof of \eqref{commutator gamma 4}. We obtain \eqref{commutator gamma 4a}.

By virtue of \eqref{identity}, all estimates stated in the lemma also hold for $\gamma_\e(\z,D_x)$.
The lemma is proved.
\end{proof}

\begin{corollary}\label{cor 1}
Under the assumptions of Lemma \ref{lem for commutator gamma} and $r\in[0,s]$, there hold
\beq\label{commutator gamma 5}
\bigl\|J_\e^{-\f12}\bigl([\p_x^2,\Gamma_\e(\z,D_x)]f\bigr)\bigr\|_{H^r}
\lesssim \e^{-\f12} \|J_\e^{\f12}\z\|_{H^{s+1}}\|f\|_{H^r},\quad\forall f\in H^r(\R^2),
\eeq
\beq\label{commutator gamma 6}
\bigl\|J_\e^{-\f12}\bigl([B(D)Y_\e,\Gamma_\e(\z,D_x)]f\bigr)\bigr\|_{H^s}
\lesssim \|J_\e^{\f12}\z\|_{H^{s+1}}\bigl(\|J_\e^\f12 f\|_{H^{s}} + \sqrt{\e}\|\na f\|_{H^s}\bigr),\quad\forall f\in H^{s+1}(\R^2),
\eeq
\beq\label{commutator gamma 6a}
\bigl\|J_\e^{\f12}\bigl([B(D)Y_\e,\Gamma_\e(\z,D_x)]f\bigr)\bigr\|_{L^2}
\lesssim\|J_\e^{\f12}\z\|_{H^{s+1}}\bigl(\|J_\e^{\f12}f\|_{L^2}+\e\|J_\e^{\f12}f\|_{\dH^2}\bigr),\quad\forall f\in J_\e^{\f12}H^2(\R^2).
\eeq
Moreover, \eqref{commutator gamma 5}, \eqref{commutator gamma 6} and \eqref{commutator gamma 6a} also hold for $\gamma_\e(\z,D_x)$.
\end{corollary}
\begin{proof}
(1). Since 
\beno
[\p_x^2,\Gamma_\e(\z,D_x)]f
=\p_x\bigl([\p_x,\Gamma_\e(\z,D_x)]f\bigr)+[\p_x,\Gamma_\e(\z,D_x)]\p_xf,
\eeno
by using \eqref{commutator gamma X} with $X=H^r(\R^2)$ and $J^{-\f12}H^r(\R^2)$, we have
\beno\begin{aligned}
\bigl\|J_\e^{-\f12}\bigl([\p_x^2,\Gamma_\e(\z,D_x)]f\bigr)\bigr\|_{H^r}
&\lesssim\e^{-\f12}\|[\p_x,\Gamma_\e(\z,D_x)]f\|_{H^r}
+\bigl\|J_\e^{-\f12}\bigl([\p_x,\Gamma_\e(\z,D_x)]\p_x f\bigr)\bigr\|_{H^r}\\
&\lesssim \e^{-\f12}\|J_\e^{\f12}\z\|_{H^{s+1}}\|f\|_{H^r}+\e^{-\f12} \|J_\e^{\f12}\z\|_{H^{s+1}}\cdot \e^{\f12}\|J_\e^{-\f12}\p_xf\|_{H^r},
\end{aligned}\eeno
which hints \eqref{commutator gamma 5}. 

Consequently, we also have
\beno\begin{aligned}
\bigl\|J_\e^{\f12}\bigl([\p_x^2,\Gamma_\e(\z,D_x)]f\bigr)\bigr\|_{H^1}
&\lesssim\|J_\e^{\f12}([\p_x,\Gamma_\e(\z,D_x)]f)\|_{H^2}
+\bigl\|J_\e^{\f12}\bigl([\p_x,\Gamma_\e(\z,D_x)]\p_x f\bigr)\bigr\|_{H^1}.
\end{aligned}\eeno
Using \eqref{commutator gamma X} for  $X=J_\e^{\f12}H^1(\R^2)$ and  $J_\e^{\f12}H^2(\R^2)$, we get 
\beno
\bigl\|J_\e^{\f12}\bigl([\p_x^2,\Gamma_\e(\z,D_x)]f\bigr)\bigr\|_{H^1}
\lesssim\|J_\e^{\f12}\z\|_{H^{s+1}}\bigl(\|J_\e^{\f12}f\|_{H^2}+\|J_\e^{\f12}\p_xf\|_{H^1}\bigr)
\lesssim\|J_\e^{\f12}\z\|_{H^{s+1}}\|J_\e^{\f12}f\|_{H^2},
\eeno
which implies
\beq\label{commutator gamma 5a}
\bigl\|J_\e^{\f12}\bigl([\p_x^2,\Gamma_\e(\z,D_x)]f\bigr)\bigr\|_{H^1}
\lesssim\|J_\e^{\f12}\z\|_{H^{s+1}}\|J_\e^{\f12}f\|_{H^2},\quad\forall f\in J_\e^{\f12}H^2(\R^2).
\eeq

(2). Since $Y_\e=1-\f\e6\p_x^2$ and 
\beno
[B(D)Y_\e,\Gamma_\e(\z,D_x)]f=[B(D),\Gamma_\e(\z,D_x)]Y_\e f
-\f\e6B(D)\bigl([\p_x^2,\Gamma_\e(\z,D_x)]f\bigr),
\eeno
we get by using $B(\xi)\sim|\xi|$ that
\beq\label{R 21}\begin{aligned}
&\bigl\|J_\e^{-\f12}\bigl([B(D)Y_\e,\Gamma_\e(\z,D_x)]f\bigr)\bigr\|_{H^s}
\lesssim\bigl\|J_\e^{-\f12}\bigl([B(D),\Gamma_\e(\z,D_x)]Y_\e f\bigr)\bigr\|_{H^s}\\
&\qquad\qquad
+\e\bigl\|J_\e^{-\f12}\bigl([\p_x^2,\Gamma_\e(\z,D_x)]\na f\bigr)\bigr\|_{H^s}
+\e\bigl\| J_\e^{-\f12}\bigl([\p_x^2,\na\Gamma_\e(\z,D_x)]f\bigr)\bigr\|_{H^s}.
\end{aligned}\eeq

For the last term of \eqref{R 21}, we have
\beno
[\p_x^2,\na\Gamma_\e(\z,D_x)]f=\p_x([\p_x,\na\Gamma_\e(\z,D_x)]f)+[\p_x,\na\Gamma_\e(\z,D_x)]\p_x f.
\eeno
Thanks to \eqref{commutator gamma 3b} and \eqref{commutator gamma 3c}, we get
\beno\begin{aligned}
\e\bigl\| J_\e^{-\f12}\bigl([\p_x^2,\na\Gamma_\e(\z,D_x)]f\bigr)\bigr\|_{H^s}
&\lesssim\sqrt\e\|[\p_x,\na\Gamma_\e(\z,D_x)]f\|_{H^s}+\e\bigl\| J_\e^{-\f12}\bigl([\p_x,\na\Gamma_\e(\z,D_x)]\p_xf\bigr)\bigr\|_{H^s}\\
&\lesssim \|J_\e^{\f12}\z\|_{H^{s+1}}\|f\|_{H^s}+\sqrt\e\|\z\|_{H^{s+1}}\|J_\e^{-\f12}\p_x f\|_{H^s},
\end{aligned}\eeno
which gives rise to
\beq\label{R 22}
\e\bigl\| J_\e^{-\f12}\bigl([\p_x^2,\na\Gamma_\e(\z,D_x)]f\bigr)\bigr\|_{H^s}
\lesssim\|J_\e^{\f12}\z\|_{H^{s+1}}(\|f\|_{H^s}+\sqrt\e \| \nabla f \|_{H^s} ).
\eeq

Using \eqref{commutator gamma 4}, \eqref{commutator gamma 5} and \eqref{R 22}, we obtain 
\beno\begin{aligned}
&\quad\bigl\|J_\e^{-\f12}\bigl([B(D)Y_\e,\Gamma_\e(\z,D_x)]f\bigr)\bigr\|_{H^s}\\
&\lesssim\|J_\e^{\f12}\z\|_{H^{s+1}}\|J_\e^{-\f12} Y_\e f\|_{H^s}
+\sqrt\e\|J_\e^{\f12}\z\|_{H^{s+1}}\|\na f\|_{H^s}
+\|J_\e^{\f12}\z\|_{H^{s+1}}(\|f\|_{H^s}+\sqrt\e \|\nabla f \|_{H^s} )\\
&\lesssim \|J_\e^{\f12}\z\|_{H^{s+1}}\bigl(\|J_\e^\f12 f\|_{H^{s}} + \sqrt{\e}\|\na f\|_{H^s}\bigr).
\end{aligned}\eeno
This is exactly \eqref{commutator gamma 6}. 

(3). Similar proof as \eqref{commutator gamma 6}, using \eqref{commutator gamma 4a}
with $r=0$ and \eqref{commutator gamma 5a}, we have
\beno\begin{aligned}
\bigl\|J_\e^{\f12}\bigl([B(D)Y_\e,\Gamma_\e(\z,D_x)]f\bigr)\bigr\|_{L^2}
&\lesssim\bigl\|J_\e^{\f12}\bigl([B(D),\Gamma_\e(\z,D_x)]Y_\e f\bigr)\bigr\|_{L^2}
+\e\bigl\|J_\e^{\f12}\bigl([\p_x^2,\Gamma_\e(\z,D_x)]f\bigr)\bigr\|_{\dH^1}\\
&\lesssim\|J_\e^{\f12}\z\|_{H^{s+1}}\|J_\e^{\f12} Y_\e f\|_{L^2}
+\e\|J_\e^{\f12}\z\|_{H^{s+1}}\|J_\e^{\f12}f\|_{H^2}\\
&\lesssim\|J_\e^{\f12}\z\|_{H^{s+1}}\bigl(\|J_\e^{\f12}f\|_{L^2}+\e\|J_\e^{\f12}f\|_{\dH^2}\bigr).
\end{aligned}\eeno
This is exactly \eqref{commutator gamma 6a}.
The corollary is proved.
\end{proof}

\subsubsection{Good unknowns.} To symmetrize system \eqref{equation of p and theta in case 2}, we introduce good unknowns as follows:
\begin{align}
	\wt{p} \eqdefa \,&p + \e J_\e^{-1} \bigl(\vv V\cdot \nabla B^{-1}(D)K_\e \theta + \f12 \z \cdot K_\e p  \bigr) - \frac{\e^2}{6} \gamma_\e(\z,D_x) Y_\e^{-2} B^{-2}(D)  \p_x^4 p ,\label{wtp in case 2} \\
	\wt\theta \eqdefa \,& \theta - \e J_\e^{-1} B^{-1}(D) \bigl( \vv V\cdot\na K_\e p - \f12 \z \cdot B(D)K_\e\theta  \bigr).\label{wttheta in case 2}
\end{align}

\begin{lemma}\label{lem for equivalent 2}
Let $s>3$ and $(\vv V,\z)=(v,w,\z)$ be a smooth enough solution of \eqref{WTB case 4} over some time interval $[0,T^*]$ satisfying 
\beq\label{ansatz 2 a}
\sup_{t\in[0,T^*]}\sqrt\e\bigl(\|J_\e^{\f12}\z\|_{H^{s+1}}+\|J_\e^{\f12}\vv V\|_{H^{s+1}}\bigr)\leq\f{1}{C_s}.
\eeq
Then there hold for any $t\in[0,T^*]$,
\beq\label{equivalent 2}\begin{aligned}
&\|J^{\f12}_\e(\wt{p}-p)\|_{H^s}\lesssim \e\|J^{\f12}_\e\z\|_{H^{s+1}}\cdot\|p\|_{H^s}+\e\|\vv V\|_{H^s}\|\theta\|_{H^s},\\
&\|J^{\f12}_\e(\wt\theta-\theta)\|_{\dH^s}
\lesssim\e\|\z\|_{H^s}\|\theta\|_{H^s}+\e\|\vv V\|_{H^s}\|p\|_{H^s}.
\end{aligned}\eeq
Moreover, there holds for any $t\in[0,T^*]$
\beq\label{estimate for wt theta 2}
\|J^{\f12}_\e\wt{p}\|_{H^s}+\|J^{\f12}_\e\na\wt\theta\|_{H^{s-1}}\lesssim\|J_\e^{\f12} p\|_{H^s}+\|J_\e^{\f12} \theta\|_{H^s}.
\eeq
\end{lemma}
\begin{proof}
We divide the proof into three steps.

{\bf Step 1. Bound of $\|J^{\f12}_\e(\wt{p}-p)\|_{H^s}$.} Due to \eqref{wtp in case 2}, we obtain by using \eqref{product estimate 1} and the facts that $B(\xi)\sim |\xi|$ and $K_\e(\xi_1)\sim 1$ that 
\beno
\|J^{-\f12}_\e\left(\vv V\cdot \nabla B^{-1}(D)K_\e \theta \right)\|_{H^s}
\lesssim\|J^{-\f12}_\e\vv V\|_{H^s}\|J^{-\f12}_\e \nabla B^{-1}(D)K_\e \theta\|_{H^s}\lesssim\|\vv V\|_{H^s}\|\theta\|_{H^s},
\eeno
\beno
\|J^{-\f12}_\e\left(\z \cdot K_\e p  \right)\|_{H^s}\lesssim\|J^{-\f12}_\e\z\|_{H^s}\|J^{-\f12}_\e K_\e p\|_{H^s}\lesssim\|\z\|_{H^s}\|p\|_{H^s}.
\eeno

For the last term in \eqref{wtp in case 2}, noticing that $B(\xi)\sim |\xi|$ and $Y_\e(\xi_1)\sim J_\e(\xi_1)=1+\f\e2\xi_1^2$, we obtain by using \eqref{operator norm 2} that
\beno\begin{aligned}
\e\|J^{\f12}_\e\gamma_\e(D,\z) Y_\e^{-2} B^{-2}(D)  \p_x^4 p\|_{H^s}
&\lesssim\|\gamma_\e(D,\z)\|\cdot\e\|J^{\f12}_\e Y_\e^{-2} B^{-2}(D)  \p_x^4 p\|_{H^s}\\
&\lesssim\|J^{\f12}_\e\z\|_{H^{s+1}}\cdot\|p\|_{H^s}
\end{aligned}\eeno

Then we get
\beno
\|J^{\f12}_\e(\wt{p}-p)\|_{H^s}\lesssim\e\|J^{\f12}_\e\z\|_{H^{s+1}}\cdot\|p\|_{H^s}+\e\|\vv V\|_{H^s}\|\theta\|_{H^s}.
\eeno
This is  the first inequality of \eqref{equivalent 2}.

\medskip

{\bf Step 2. Bound of $\|J^{\f12}_\e(\wt\theta-\theta)\|_{\dH^s}$.} Thanks to \eqref{wttheta in case 2} and \eqref{product B2}, we have
\beno\begin{aligned}
&\|J_\e^{-\f12} B^{-1}(D)(\vv V\cdot\na K_\e p)\|_{\dH^s}
\lesssim\|\vv V\|_{H^s}\cdot\|J_\e^{-\f12} B^{-1}(D)\na K_\e p\|_{H^s}
\lesssim\|\vv V\|_{H^s}\|p\|_{H^s},\\
&\|J_\e^{-\f12} B^{-1}(D)(\z \cdot B(D)K_\e\theta)\|_{\dH^s}
\lesssim\|\z\|_{H^s}\cdot\|J_\e^{-\f12}K_\e\theta\|_{H^s}
\lesssim\|\z\|_{H^s}\|\theta\|_{H^s},
\end{aligned}\eeno
which give rise to
\beno
\|J^{\f12}_\e(\wt\theta-\theta)\|_{\dH^s}\lesssim\e\|\z\|_{H^s}\|\theta\|_{H^s}+
\e\|\vv V\|_{H^s}\|p\|_{H^s}
\eeno
This is the second part of \eqref{equivalent 2}.

\medskip

{\bf Step 3. Bound of $\|J_\e^{\f12}\na\wt\theta\|_{H^{s-1}}$.} Using the fact that $B(\xi)\sim|\xi|$, it is easy to get 
\beno\begin{aligned}
\|J_\e^{\f12}\na\wt\theta\|_{L^2}&\lesssim\|J_\e^{\f12}\na\theta\|_{L^2}+\e\|J_\e^{-\f12}\na B^{-1}(D)(\z \cdot B(D)K_\e\theta)\|_{L^2}
+\e\|J_\e^{-\f12}\na B^{-1}(D)(\vv V\cdot \nabla K_\e p)\|_{L^2}\\
&\lesssim\|J_\e^{\f12}\na\theta\|_{L^2}+\e\|\z\|_{L^\infty}\|\na\theta\|_{L^2}+\e\|\vv V\|_{L^\infty}\|\na p\|_{L^2}\\
&\lesssim\|\theta\|_{H^s}+\e\|\z\|_{H^s}\|\theta\|_{H^s}+\e\|\vv V\|_{H^s}\|p\|_{H^s},
\end{aligned}\eeno
from which and the ansatz \eqref{ansatz 2 a}, we deduce that
\beq\label{R 3}
\|J_\e^{\f12}\na\wt\theta\|_{L^2}\lesssim\|p\|_{H^s}+\|\theta\|_{H^s}.
\eeq

Combining \eqref{equivalent 2} and \eqref{R 3}, we obtain \eqref{estimate for wt theta 2} by using \eqref{ansatz 2 a}. It completes the proof of lemma. 
\end{proof}

\subsubsection{Symmetrization of system \eqref{equation of p and theta in case 2}.}
Now, we derive the evolution system for $(\wt{p}, \wt{\theta})$. With $(\wt p,\wt\theta)$, we deduce from \eqref{equation of p and theta in case 2} that
\beq\label{R 1}
J_\e p_t=Y_\e B(D)\wt\theta+\e N_p,
\eeq
and 
\beq\label{R 2}\begin{aligned}
J_\e\theta_t&=-Y_\e B(D)\wt p+\f{\e^2}{6}Y_\e B(D)\bigl(\Gamma_\e(\z,D_x)Y_\e^{-2}B^{-2}(D)\p_x^4p\bigr)+\e N_\theta,
\end{aligned}\eeq
where we used \eqref{def of J Y K} and the fact that $\Gamma_\e(\z, D_x)=\z-\gamma_\e(\z,D_x) $ (see \eqref{identity}) in  \eqref{R 2}. 

{\bf  Evolution equation for $\wt p$. } By the definition of $\wt p$ in \eqref{wtp in case 2}, we have
\beq\label{R 4}
J_\e\wt p_t=\underbrace{J_\e p_t+ \f\e2 \z \cdot K_\e p_t+\e\vv V\cdot \nabla B^{-1}(D)K_\e \theta_t - \frac{\e^2}{6}J_\e\bigl(\gamma_\e(\z,D_x) Y_\e^{-2} B^{-2}(D)\p_x^4 p_t\bigr)}_{I_2}+\e N_{\wt p,1},
\eeq
where 
\beq\label{def of N wtp 1 a}
N_{\wt p,1}\eqdefa\f12 \z_t \cdot K_\e p+\vv V_t\cdot\na B^{-1}(D)K_\e\theta
-\frac{\e^2}{6}J_\e\bigl([\p_t,\gamma_\e(\z,D_x)]Y_\e^{-2} B^{-2}(D)\p_x^4 p\bigr).
\eeq

For $I_2$, using \eqref{R 1}, \eqref{R 2} and the fact $K_\e=J_\e Y^{-1}_\e$, we get 
\beno\begin{aligned}
I_2&=J_\e p_t+ \f\e2 \z \cdot K_\e p_t+\e\vv V\cdot \nabla B^{-1}(D)K_\e \theta_t
- \frac{\e^2}{6}\gamma_\e(\z,D_x) Y_\e^{-2} B^{-2}(D)\p_x^4J_\e p_t\\
&\qquad +\frac{\e^3}{12}[\p_x^2,\gamma_\e(\z,D_x)] Y_\e^{-2} B^{-2}(D)\p_x^4 p_t\\
&=Y_\e B(D)\wt\theta+ \f\e2\z \cdot B(D)\wt\theta
-\e\vv V\cdot\nabla\wt p- \frac{\e^2}{6}\gamma_\e(\z,D_x) Y_\e^{-1} B^{-1}(D)\p_x^4\wt\theta+\e N_{\wt p,2}+\e N_{\wt p,3} ,
\end{aligned}\eeno
where
\beq\label{def of N wtp 2 a}
N_{\wt p,2}\eqdefa\frac{\e^2}{12}[\p_x^2,\gamma_\e(\z,D_x)] Y_\e^{-2} B^{-2}(D)\p_x^4 p_t+\f{\e^2}{6}\vv V\cdot \nabla\bigl(\Gamma_\e(\z,D_x)Y_\e^{-2}B^{-2}(D)\p_x^4 p\bigr),
\eeq
\beq\label{def of N wtp 3 a}
N_{\wt p,3}\eqdefa N_p+\f{\e}{2} \z \cdot Y^{-1}_\e N_p+\e\vv V\cdot\na B^{-1}(D)Y_\e^{-1}N_\theta- \frac{\e^2}{6}\gamma_\e(\z,D_x)Y_\e^{-2} B^{-2}(D)\p_x^4  N_p.
\eeq

Then we deduce from \eqref{R 4} that
\beq\label{eq for wt p 2}\begin{aligned}
&J_\e\wt p_t-\Bigl(Y_\e B(D)+\f\e2\z B(D)
-\frac{\e^2}{6}\gamma_\e(\z,D_x) Y_\e^{-1} B^{-1}(D)\p_x^4\Bigr)\wt\theta=-\e\vv V\cdot\nabla\wt p+\e N_{\wt p},
\end{aligned}\eeq
where $N_{\wt p}=N_{\wt p,1}+N_{\wt p,2}+N_{\wt p,3}$ and $N_{\wt p,1}$, $N_{\wt p,2}$, $N_{\wt p,3}$ are defined in \eqref{def of N wtp 1 a}, \eqref{def of N wtp 2 a} and \eqref{def of N wtp 3 a}.

\smallskip

{\bf  Evolution equation for $\wt\theta$. } By the definition of $\wt\theta$ in \eqref{wttheta in case 2}, we have
\beno\begin{aligned}
J_\e\wt\theta_t=&J_\e\theta_t  + \f\e2B^{-1}(D) \bigl(\z B(D)K_\e\theta_t \bigr) - \e B^{-1}(D) \bigl( \vv V\cdot\na K_\e p_t\bigr)\\
&-\e B^{-1}(D) \bigl( \vv V_t\cdot\na K_\e p\bigr) +\f\e2B^{-1}(D) \bigl(\z_t \cdot B(D)K_\e\theta\bigr)
\end{aligned}\eeno
which yields to
\beq\label{R 5}
J_\e\wt\theta_t=\underbrace{J_\e\theta_t  + \f\e2\z  K_\e\theta_t - \e \vv V\cdot\na K_\e B^{-1}(D) p_t}_{I_3}+\e N_{\wt\theta,1},
\eeq
where 
\beq\label{def of N wttheta 1 a}\begin{aligned}
N_{\wt\theta,1}\eqdefa&-B^{-1}(D) \bigl( \vv V_t\cdot\na K_\e p\bigr) +\f12B^{-1}(D) \bigl( \z_t \cdot B(D)K_\e\theta\bigr)\\
&+\f12[B^{-1}(D),\z]B(D)K_\e\theta_t-[B^{-1}(D),\vv V]\cdot\na K_\e p_t.
\end{aligned}\eeq

For $I_3$, using \eqref{R 1}, \eqref{R 2} and the fact $K_\e=J_\e Y^{-1}_\e$, we get
\beno\begin{aligned}
I_3&=-Y_\e B(D)\wt p-\f\e2\z B(D)\wt p+\f{\e^2}{6}Y_\e B(D)\bigl(\Gamma_\e(\z,D_x)Y_\e^{-2}B^{-2}(D)\p_x^4 p\bigr)\\
&\qquad+\f{\e^3}{12}\z B(D)\bigl(\Gamma_\e(\z,D_x)Y_\e^{-2}B^{-2}(D)\p_x^4 p\bigr)
- \e \vv V\cdot\na\wt\theta+\e N_{\wt\theta,2}\\
\end{aligned}\eeno
where
\beq\label{def of N wttheta 2 a}
N_{\wt\theta,2}\eqdefa N_\theta+\f\e2\z Y^{-1}_\e N_\theta
-\e\vv V\cdot\na B^{-1}(D)Y^{-1}_\e N_p.
\eeq

Since $Y_\e=1-\f\e6\p_x^2$ and $\gamma_\e(\z,D_x)=(1+\f\e2\z Y^{-1}_\e)\Gamma_\e(\z,D_x)$, we have
\beno\begin{aligned}
&Y_\e B(D)\bigl(\Gamma_\e(\z,D_x)f\bigr)
+\f\e2 \z B(D)\bigl(\Gamma_\e(\z,D_x)f\bigr)\\
=&(1+\f\e2\z Y^{-1}_\e)\bigl(\Gamma_\e(\z,D_x)Y_\e B(D)f\bigr)+(1+\f\e2\z Y^{-1}_\e)\bigl([Y_\e B(D),\Gamma_\e(\z,D_x)]f\bigr)\\
=&\gamma_\e(\z,D_x)Y_\e B(D)f+(1+\f\e2\z Y^{-1}_\e)\bigl([Y_\e B(D),\Gamma_\e(\z,D_x)]f\bigr).
\end{aligned}\eeno
Then we get
\beno\begin{aligned}
I_3&=-Y_\e B(D)\wt p-\f\e2\z B(D)\wt p+\f{\e^2}{6}\gamma_\e(\z,D_x)\bigl(Y_\e^{-1}B^{-1}(D)\p_x^4\wt p\bigr)- \e \vv V\cdot\na\wt\theta+\e N_{\wt\theta,2}+\e N_{\wt\theta,3},
\end{aligned}\eeno
where
\beq\label{def of N wttheta 3 a}\begin{aligned}
N_{\wt\theta,3}&\eqdefa -\f{\e}{6}\gamma_\e(\z,D_x)\bigl(Y_\e^{-1}B^{-1}(D)\p_x^4(\wt p-p)\bigr)\\
&\qquad+\f\e6(1+\f\e2\z Y^{-1}_\e)\bigl([Y_\e B(D),\Gamma_\e(\z,D_x)]Y_\e^{-2}B^{-2}(D)\p_x^4 p\bigr).
\end{aligned}\eeq

Hence, we deduce from \eqref{R 5} that
\beq\label{eq for wt theta 2}
J_\e\wt\theta_t+\Bigl(Y_\e B(D)+\f\e2\z B(D)
-\frac{\e^2}{6}\gamma_\e(\z,D_x) Y_\e^{-1} B^{-1}(D)\p_x^4\Bigr)\wt p=-\e\vv V\cdot\nabla\wt\theta+\e N_{\wt\theta},
\eeq
where $N_{\wt\theta}=N_{\wt\theta,1}+N_{\wt\theta,2}+N_{\wt\theta,3}$ and $N_{\wt\theta,1}$, $N_{\wt\theta,2}$, $N_{\wt\theta,3}$ are defined in \eqref{def of N wttheta 1 a}, \eqref{def of N wttheta 2 a} and \eqref{def of N wttheta 3 a}.

Therefore, combining \eqref{eq for wt p 2} and \eqref{eq for wt theta 2}, we obtain a symmetrized evolution system for $(\wt p,\wt\theta)$ as follows:
\beq\label{equations for wt p theta 2}
\left\{\begin{aligned}
&J_\e\wt p_t-\Bigl(Y_\e B(D)+\f\e2\z B(D)
-\frac{\e^2}{6}\gamma_\e(\z,D_x) Y_\e^{-1} B^{-1}(D)\p_x^4\Bigr)\wt\theta=-\e\vv V\cdot\nabla\wt p+\e N_{\wt p},\\
&J_\e\wt\theta_t+\Bigl(Y_\e B(D)+\f\e2\z B(D)
-\frac{\e^2}{6}\gamma_\e(\z,D_x) Y_\e^{-1} B^{-1}(D)\p_x^4\Bigr)\wt p=-\e\vv V\cdot\nabla\wt\theta+\e N_{\wt\theta}.
\end{aligned}\right.\eeq

\begin{lemma}\label{lem for nonlinear term 2 case 2}
Under the assumptions of Lemma \ref{lem for equivalent 2}, there holds
\beq\label{estimate for N wt p theta 2}
\|J_\e^{-\f12}N_{\wt p}\|_{\dH^s}+\|J_\e^{-\f12}N_{\wt\theta}\|_{\dH^s}
\lesssim\|J_\e^{\f12}p\|_{H^s}^2+\|J_\e^{\f12}\theta\|_{H^s}^2.
\eeq
\end{lemma}
\begin{proof}
We divide the proof into four steps.

{\bf Step 1. Estimate of $N_{\wt p}$.} We estimate $N_{\wt p,1}$, $N_{\wt p,2}$ and $N_{\wt p,3}$ one by one.

{\it Step 1.1. Estimate of $N_{\wt p,1}$.} For the first two terms of $N_{\wt p,1}$ in \eqref{def of N wtp 1 a}, using \eqref{product estimate 1} and the facts that $B(\xi)\sim|\xi|$ and $K_\e(\xi_1)\sim 1$, we get
\beno
\|J_\e^{-\f12}(\z_t \cdot K_\e p)\|_{\dH^s}\lesssim\|J_\e^{-\f12}\z_t \|_{H^s}
\|J_\e^{-\f12}K_\e p\|_{H^s}\lesssim\|J_\e^{-\f12}\z_t \|_{H^s}\|p\|_{H^s},
\eeno
\beno
\|J_\e^{-\f12}(\vv V_t\cdot\na B^{-1}(D)K_\e\theta)\|_{\dH^s}\lesssim\|J_\e^{-\f12}\vv V_t \|_{H^s}
\|J_\e^{-\f12}\na B^{-1}(D)K_\e\theta\|_{H^s}\lesssim\|J_\e^{-\f12}\vv V_t \|_{H^s}\|\theta\|_{H^s}.
\eeno

For the last term of $N_{\wt p,1}$, using \eqref{commutator gamma 1}, we have
\beno\begin{aligned}
\e^2\bigl\|J_\e^{\f12}\bigl([\p_t,\gamma_\e(\z,D_x)]Y_\e^{-2} B^{-2}(D)\p_x^4 p\bigr)\bigr\|_{H^s}&\lesssim  \|J_\e^{\f12}\z_t \|_{H^s}\cdot\e^2\|J_\e^{\f12}Y_\e^{-2} B^{-2}(D)\p_x^4 p\|_{H^s}\\
&\lesssim\e\|J_\e^{\f12}\z_t \|_{H^s}\|p\|_{H^s}.
\end{aligned}\eeno

Then we obtain 
\beq\label{estimate for N wtp 1 a}
\|J_\e^{-\f12}N_{\wt p,1}\|_{\dH^s}\lesssim\|J_\e^{\f12}\z_t \|_{H^s}\|p\|_{H^s}
+\|J_\e^{-\f12}\vv V_t \|_{H^s}\|\theta\|_{H^s}.
\eeq

\smallskip

{\it Step 1.2. Estimate of $N_{\wt p,2}$.} For the first term of $N_{\wt p,2}$ in \eqref{def of N wtp 2 a}, using \eqref{commutator gamma 5} with $r=s$, one gets
\beno\begin{aligned}
\e^2\bigl\|J_\e^{-\f12}\bigl([\p_x^2,\gamma_\e(\z,D_x)] Y_\e^{-2} B^{-2}(D)\p_x^4 p_t\bigr)\bigr\|_{\dH^s}
&\lesssim \|J_\e^{\f12}\z\|_{H^{s+1}}\cdot\e^\f32\|Y_\e^{-2} B^{-2}(D)\p_x^4 p_t\|_{H^s}\\
&\lesssim\|J_\e^{\f12}\z\|_{H^{s+1}}\|p_t\|_{H^{s-1}},
\end{aligned}\eeno
where one used the fact that $\|f\|_{H^s}\sim\|f\|_{L^2}+\|f\|_{\dH^s}$ in the last inequality.

For the second term of $N_{\wt p,2}$ in \eqref{def of N wtp 2 a}, using \eqref{product estimate 1} and \eqref{operator norm 2}, one has
\beno\begin{aligned}
&\quad\e^2\bigl\|J_\e^{-\f12}\bigl[\vv V\cdot \nabla\bigl(\Gamma_\e(\z,D_x)Y_\e^{-2}B^{-2}(D)\p_x^4 p\bigr)\big]\bigr\|_{\dH^s}\\
&\lesssim\e^2\|J_\e^{-\f12}\vv V\|_{H^s}\bigl\|J_\e^{-\f12}\bigl(\Gamma_\e(\z,D_x)Y_\e^{-2}B^{-2}(D)\p_x^4 p\bigr)\bigr\|_{H^{s+1}}\\
&\lesssim\sqrt\e\|\vv V\|_{H^s}\cdot\|\Gamma_\e(\z,D_x)\|\cdot\e^{\f32}\|J_\e^{-\f12}Y_\e^{-2}B^{-2}(D)\p_x^4 p\|_{H^{s+1}}\\
&\lesssim\sqrt\e\|\vv V\|_{H^s}\cdot\|J_\e^{\f12}\z\|_{H^{s+1}}\cdot\| p\|_{H^s},
\end{aligned}\eeno
which together with the ansatz \eqref{ansatz 2 a} implies
\beno
\e^2\bigl\|J_\e^{-\f12}\bigl[\vv V\cdot \nabla\bigl(\Gamma_\e(\z,D_x)Y_\e^{-2}B^{-2}(D)\p_x^4 p\bigr)\big]\bigr\|_{\dH^s}\lesssim\|J_\e^{\f12}\z\|_{H^{s+1}}\cdot\| p\|_{H^s}.
\eeno

Thus, we obtain
\beq\label{estimate for N wtp 2 a}
\|J_\e^{-\f12}N_{\wt p,2}\|_{\dH^s}\lesssim\|J_\e^{\f12}\z\|_{H^{s+1}}\|p_t\|_{H^{s-1}}+\|J_\e^{\f12}\z\|_{H^{s+1}}\| p\|_{H^s}.
\eeq

\smallskip

{\it Step 1.3. Estimate of $N_{\wt p,3}$.} Using \eqref{product estimate 1}, \eqref{operator norm 1} and \eqref{operator norm 2}, we deduce from \eqref{def of N wtp 3 a} that
\beno\begin{aligned}
\|J_\e^{-\f12}N_{\wt p,3}\|_{\dH^s}&\lesssim\|J_\e^{-\f12} N_p\|_{\dH^s}+\|\f\e2\z Y^{-1}_\e\|\cdot\|J_\e^{-\f12} N_p\|_{H^s}
+\e\|J_\e^{-\f12}\vv V\|_{H^s}\cdot\|J_\e^{-\f12}\na B^{-1}(D)Y_\e^{-1}N_\theta\|_{H^s}\\
&\quad+ \e^2\|\gamma_\e(\z,D_x)\|\cdot\|J_\e^{-\f12}Y_\e^{-2} B^{-2}(D)\p_x^4  N_p\|_{H^s}\\
&\lesssim\|J_\e^{-\f12} N_p\|_{H^s}+\e\|\vv V\|_{H^s}\|J_\e^{-\f12}N_\theta\|_{H^s}+\e\|J_\e^{\f12}\z\|_{H^{s+1}}\|J_\e^{-\f12} N_p\|_{H^s},
\end{aligned}\eeno
which along with the ansatz \eqref{ansatz 2 a} implies
\beq\label{estimate for N wtp 3 a}
\|J_\e^{-\f12}N_{\wt p,3}\|_{\dH^s}\lesssim
\|J_\e^{-\f12}N_p\|_{H^s}+\|J_\e^{-\f12}N_\theta\|_{H^s}.
\eeq

\smallskip

{\it Step 1.4. Estimate of $N_{\wt p}$.} Combining \eqref{estimate for N wtp 1 a},
\eqref{estimate for N wtp 2 a} and \eqref{estimate for N wtp 3 a}, we obtain 
\beq\label{estimate for N wtp a}\begin{aligned}
\|J_\e^{-\f12}N_{\wt p}\|_{\dH^s}\lesssim&
\|J_\e^{\f12}\z_t \|_{H^s}\|p\|_{H^s}
+\|J_\e^{-\f12}\vv V_t \|_{H^s}\|\theta\|_{H^s}+\|J_\e^{\f12}\z\|_{H^{s+1}}\|p_t\|_{H^{s-1}}
\\
&+\|J_\e^{\f12}\z\|_{H^{s+1}}\|p\|_{H^s}+\|J_\e^{-\f12}N_p\|_{H^s}+\|J_\e^{-\f12}N_\theta\|_{H^s}.
\end{aligned}\eeq

\medskip

{\bf Step 2. Estimate of $N_{\wt\theta}$.} We estimate $N_{\wt\theta,1}$, $N_{\wt\theta,2}$ and $N_{\wt\theta,3}$ one by one.

{\it Step 2.1. Estimate of $N_{\wt\theta,1}$.} For the first two terms of  $N_{\wt\theta,1}$ in \eqref{def of N wttheta 1 a}, using \eqref{product B2}, we get
\beno\begin{aligned}
&\|J^{-\f12}_\e B^{-1}(D) \bigl( \vv V_t\cdot\na K_\e p\bigr)\|_{\dH^s}
\lesssim\|\vv V_t\|_{H^s}\|J^{-\f12}_\e B^{-1}(D)\na K_\e p\|_{H^s}
\lesssim\|\vv V_t\|_{H^s}\|p\|_{H^s},\\
&\|J^{-\f12}_\e B^{-1}(D) \bigl( \z_t \cdot B(D)K_\e\theta\bigr)\|_{\dH^s}
\lesssim\|\z_t\|_{H^s}\|J^{-\f12}_\e K_\e\theta\|_{H^s}
\lesssim\|\z_t\|_{H^s}\|\theta\|_{H^s}.
\end{aligned}\eeno

While for the last two terms of $N_{\wt\theta,1}$ in \eqref{def of N wttheta 1 a}, using \eqref{commutator B2}, we have
\beno\begin{aligned}
&\|J^{-\f12}_\e([B^{-1}(D),\z]B(D)K_\e\theta_t)\|_{\dH^s}
\lesssim\|\na \z\|_{H^{s-1}}\|K_\e\theta_t\|_{H^{s-1}}
\lesssim\|\z\|_{H^s}\|\theta_t\|_{H^{s-1}},\\
&\|J^{-\f12}_\e([B^{-1}(D),\vv V]\cdot\na K_\e p_t)\|_{\dH^s}\lesssim\|\na\vv V\|_{H^{s-1}}\|B^{-1}(D)\na K_\e p_t\|_{H^{s-1}}\lesssim\|\vv V\|_{H^s}\|p_t\|_{H^{s-1}}.
\end{aligned}\eeno

Then we obtain 
\beq\label{estimate for N wttheta 1 a}
\|J_\e^{-\f12}N_{\wt\theta,1}\|_{\dH^s}\lesssim
\|\vv V_t\|_{H^s}\|p\|_{H^s}+\|\z_t\|_{H^s}\|\theta\|_{H^s}+\|\z\|_{H^s}\|\theta_t\|_{H^{s-1}}+\|\vv V\|_{H^s}\|p_t\|_{H^{s-1}}.
\eeq

\smallskip

{\it Step 2.2. Estimate of $N_{\wt\theta,2}$.} Using \eqref{product estimate 1}, we deduce from \eqref{def of N wttheta 2 a} that
\beno\begin{aligned}
&\|J_\e^{-\f12}N_{\wt\theta,2}\|_{\dH^s}\lesssim\|J_\e^{-\f12}N_\theta\|_{\dH^s}+\e\|J_\e^{-\f12}(\z Y^{-1}_\e N_\theta)\|_{\dH^s}+
\e\|J_\e^{-\f12}(\vv V\cdot\na B^{-1}(D)Y^{-1}_\e N_p)\|_{\dH^s}\\
&\lesssim\|J_\e^{-\f12}N_\theta\|_{\dH^s}+\e\|J_\e^{-\f12}\z\|_{H^s}\|J_\e^{-\f12}Y^{-1}_\e N_\theta\|_{H^s}+\e\|J_\e^{-\f12}\vv V\|_{H^s}\|J_\e^{-\f12}\na B^{-1}(D)Y^{-1}_\e N_p\|_{H^s}\\
&\lesssim(1+\e\|\z\|_{H^s})\|J_\e^{-\f12}N_\theta\|_{H^s}
+\e\|\vv V\|_{H^s}\|J_\e^{-\f12}N_p\|_{H^s}. 
\end{aligned}\eeno
Thanks to the ansatz \eqref{ansatz 2 a}, we get
\beq\label{estimate for N wttheta 2 a}
\|J_\e^{-\f12}N_{\wt\theta,2}\|_{\dH^s}\lesssim\|J_\e^{-\f12}N_p\|_{H^s}+\|J_\e^{-\f12}N_\theta\|_{H^s}.
\eeq

\smallskip

{\it Step 2.3. Estimate of $N_{\wt\theta,3}$.}  For the first term of $N_{\wt\theta,3}$ in \eqref{def of N wttheta 3 a},  using \eqref{operator norm 2}, we have
\beno\begin{aligned}
\e\|J_\e^{-\f12}\bigl(\gamma_\e(\z,D_x)\bigl(Y_\e^{-1}B^{-1}(D)\p_x^4(\wt p-p)\bigr)\bigr)\|_{H^s}
&\lesssim\|\gamma_\e(\z,D_x)\|\cdot\e\bigl\|J_\e^{-\f12}Y_\e^{-1}B^{-1}(D)\p_x^4(\wt p-p)\bigr\|_{H^s}\\
&\lesssim\|J_\e^{\f12}\z\|_{H^{s+1}}\cdot\e^{-\f12}\|\wt p-p\|_{H^s}.
\end{aligned}\eeno
Using \eqref{equivalent 2} and  \eqref{ansatz 2 a}, we obtain
\beno\begin{aligned}
&\quad\e\|J_\e^{-\f12}\bigl(\gamma_\e(\z,D_x)\bigl(Y_\e^{-1}B^{-1}(D)\p_x^4(\wt p-p)\bigr)\bigr)\|_{H^s}\\
&\lesssim\sqrt\e\|J_\e^{\f12}\z\|_{H^{s+1}}(\|J^{\f12}_\e\z\|_{H^{s+1}}+\|\vv V\|_{H^s})(\|J_\e^{\f12} p\|_{H^s}+\|J_\e^{\f12} \theta\|_{H^s})\\
&\lesssim\|J_\e^{\f12}\z\|_{H^{s+1}}(\|J_\e^{\f12} p\|_{H^s}+\|J_\e^{\f12}\theta\|_{H^s}).
\end{aligned}\eeno

For the second term of $N_{\wt\theta,3}$ in \eqref{def of N wttheta 3 a},
using \eqref{operator norm 1} and \eqref{commutator gamma 6}, we deduce that
\beno\begin{aligned}
&\quad\e\Bigl\|J_\e^{-\f12}\Bigl((1+\f\e2\z Y^{-1}_\e)\bigl([Y_\e B(D),\Gamma_\e(\z,D_x)]Y_\e^{-2}B^{-2}(D)\p_x^4 p\bigr)\Bigr)\Bigr\|_{\dH^s}\\
&\leq\e(1+\|\f\e2\z Y^{-1}_\e\|)\cdot\|J_\e^{-\f12}\bigl([Y_\e B(D),\Gamma_\e(\z,D_x)]Y_\e^{-2}B^{-2}(D)\p_x^4 p\bigr)\|_{H^s}\\
&\lesssim\|J_\e^{\f12}\z\|_{H^{s+1}}\cdot \left(  \e\|J_\e^\f12Y_\e^{-2}B^{-2}(D)\p_x^4 p\|_{H^{s}} + \e^\f32 \|Y_\e^{-2}B^{-2}(D)\p_x^4 p\|_{H^{s+1}} \right)\\
&\lesssim\|J_\e^{\f12}\z\|_{H^{s+1}}\cdot\| p\|_{H^s}.
\end{aligned}\eeno
 
 Thus, we get
\beq\label{estimate for N wttheta 3 a}
\|J_\e^{-\f12}N_{\wt\theta,3}\|_{\dH^s}
\lesssim\|J_\e^{\f12}\z\|_{H^{s+1}}\bigl(\|J_\e^{\f12}p\|_{H^s}+\|J_\e^{\f12}\theta\|_{H^s}\bigr).
\eeq

\smallskip

{\it Step 2.4. Estimate of $N_{\wt\theta}$.} Combining \eqref{estimate for N wttheta 1 a}, \eqref{estimate for N wttheta 2 a} and \eqref{estimate for N wttheta 3 a}, we obtain 
\beq\label{estimate for N wttheta a}\begin{aligned}
\|J_\e^{-\f12}N_{\wt\theta}\|_{\dH^s}\lesssim&\|\vv V_t\|_{H^s}\|p\|_{H^s}+\|\z_t\|_{H^s}\|\theta\|_{H^s}+\|\z\|_{H^s}\|\theta_t\|_{H^{s-1}}+\|\vv V\|_{H^s}\|p_t\|_{H^{s-1}}\\
&\quad+
\|J_\e^{\f12}\z\|_{H^{s+1}}\bigl(\|J_\e^{\f12}p\|_{H^s}+\|J_\e^{\f12}\theta\|_{H^s}\bigr)
+\|J_\e^{-\f12}N_p\|_{H^s}+\|J_\e^{-\f12}N_\theta\|_{H^s}.
\end{aligned}\eeq

\medskip

{\bf Step 3. Estimates of $(\z_t,\vv V_t)$ and $(p_t,\theta_t)$.} We estimate $(\z_t,\vv V_t)$ and $(p_t,\theta_t)$ one by one.

{\it Step 3.1. Estimates of $(\z_t,\vv V_t)$.} Firstly, due to \eqref{WTB case 4}, we have
\beno\begin{aligned}
\vv V_t&=-K^{-1}_\e\na \z-\e J^{-1}_\e\bigl((\vv V\cdot\na)\vv V\bigr)-\f\e2J^{-1}_\e\bigl(\z\na\z\bigr),\\
&\z_t=-v_x-K^{-1}_\e w_y-\e J^{-1}_\e\bigl(\vv V\cdot\na \z+\f12\z\div\vv V\bigr)
\end{aligned}\eeno
which along with \eqref{product estimate 1} gives rise to
\beno\begin{aligned}
\|\vv V_t\|_{H^s}&\lesssim\|\na \z\|_{H^s}+\e\|\vv V\|_{H^s}\|\na\vv V\|_{H^s}
+\e\|\z\|_{H^s}\|\na\z\|_{H^s},\\
\|J^{\f12}_\e\z_t\|_{H^s}&\lesssim\|J^{\f12}_\e\na\vv V\|_{H^s}+\e\|J^{-\f12}_\e(\vv V\cdot\na \z)\|_{H^s}+\e\|J^{-\f12}_\e(\z\div\vv V)\|_{H^s}\\
&\lesssim\|J^{\f12}_\e\na\vv V\|_{H^s}+\e\|J^{-\f12}_\e\vv V\|_{H^s}\|J^{-\f12}_\e\na\z\|_{H^s}+\e\|J^{-\f12}_\e\z\|_{H^s}\|J^{-\f12}_\e\na\vv V\|_{H^s}.
\end{aligned}\eeno
Using the ansatz \eqref{ansatz 2 a}, we get
\beq\label{R 7a}
\|\vv V_t\|_{H^s}+\|J^{\f12}_\e\z_t\|_{H^s}\lesssim\|J^{\f12}_\e\na\vv V\|_{H^s}
+\|\na \z\|_{H^s}.
\eeq

\smallskip

{\it Step 3.2. Estimates of $(p_t,\theta_t)$.} Thanks to \eqref{R 1} and \eqref{R 2}, we have
\beno\begin{aligned}
&\|p_t\|_{H^{s-1}}+\|\theta_t\|_{H^{s-1}}\lesssim\|J^{-1}_\e Y_\e B(D)\wt\theta\|_{H^{s-1}}
+\|J^{-1}_\e Y_\e B(D)\wt p\|_{H^{s-1}}
\\
&\qquad\qquad+\e^2\bigl\|J^{-1}_\e Y_\e B(D)\bigl(\Gamma_\e(\z,D_x)Y_\e^{-2}B^{-2}(D)\p_x^4 p\bigr)\bigr\|_{H^{s-1}}+\e \|J^{-1}_\e N_p\|_{H^{s-1}}+\e\|J^{-1}_\e N_\theta\|_{H^{s-1}}\\
&\lesssim\|\na\wt\theta\|_{H^{s-1}}+\|\wt p\|_{H^s}+\e^2\|\Gamma_\e(\z,D_x)Y_\e^{-2}B^{-2}(D)\p_x^4 p\|_{H^s}+\e \|J^{-\f12}_\e N_p\|_{H^s}+\e\|J^{-\f12}_\e N_\theta\|_{H^s}\\
&\lesssim\|\na\wt\theta\|_{H^{s-1}}+\|\wt p\|_{H^s}+\e^2\|\Gamma_\e(\z,D_x)\|\cdot\|Y_\e^{-2}B^{-2}(D)\p_x^4\wt p\|_{H^s}+\e \|J^{-\f12}_\e N_p\|_{H^s}+\e\|J^{-\f12}_\e N_\theta\|_{H^s},
\end{aligned}\eeno
which along with \eqref{operator norm 2} and  \eqref{estimate for wt theta 2} implies
\beno\begin{aligned}
\|p_t\|_{H^{s-1}}+\|\theta_t\|_{H^{s-1}}&\lesssim\|\na\wt\theta\|_{H^{s-1}}+\|\wt p\|_{H^s}+\e\|J_\e^{\f12}\z\|_{H^{s+1}}\|\wt p\|_{H^s}+\e \|J^{-\f12}_\e N_p\|_{H^s}+\e\|J^{-\f12}_\e N_\theta\|_{H^s}\\
&\lesssim\|J_\e^{\f12}p\|_{H^s}+\|J_\e^{\f12}\theta\|_{H^s}+\e \|J^{-\f12}_\e N_p\|_{H^s}+\e\|J_\e^{-\f12}N_\theta\|_{H^s}.
\end{aligned}\eeno

Using \eqref{estimate for N p theta 2}, \eqref{equivalent 1} and the ansatz \eqref{ansatz 2 a}, we obtain
\beq\label{R 8}
\|p_t\|_{H^{s-1}}+\|\theta_t\|_{H^{s-1}}\lesssim\|J_\e^{\f12}\na\vv V\|_{H^s}+\|J_\e^{\f12}\na\z\|_{H^s}.
\eeq

\medskip

{\bf Step 4. Estimates of $N_{\wt p}$ and $N_{\wt\theta}$.} Combining \eqref{estimate for N wtp a} and \eqref{estimate for N wttheta a}, using \eqref{R 7a}, \eqref{R 8}, \eqref{estimate for N p theta 2} and \eqref{equivalent 1}, we deduce that
\beno
\|J_\e^{-\f12}N_{\wt p}\|_{\dH^s}+\|J_\e^{-\f12}N_{\wt\theta}\|_{\dH^s}
\lesssim\|J_\e^{\f12}p\|_{H^s}^2+\|J_\e^{\f12}\theta\|_{H^s}^2.
\eeno
This is exactly \eqref{estimate for N wt p theta 2}. The lemma is proved.
\end{proof}

\subsection{The proof of Theorem \ref{main theorem} for case 2}		
The proof is similar to that of case 1. It relies on the continuity argument and the {\it a priori} energy estimate.	We only sketch the proof of {\it a priori} energy estimate. 

Firstly, we define the energy functional $E_s(t)$ as
\beq\label{energy functional for case 2}
E_s(t)\eqdefa\|J_\e^{\f12}\z\|_{H^{s+1}}^2+\|J_\e^{\f12}\vv V\|_{H^{s+1}}^2.
\eeq

We assume that
\beq\label{ansatz a 2}
\sup_{t\in[0,T_0/\e]}E_s(t)\leq 2C_0 E_s(0).
\eeq
Here the constants $C_0>1$ and $T_0>0$ are taken as
\beno
C_0=2C_1,\quad T_0=\f{1}{4\sqrt{C_1}C_2E^{\f12}_s(0)},
\eeno
where $C_1,C_2>1$ are the universal constants appearing in the following Proposition \ref{prop for a priori estimate case 2} which is concerning the {\it a priori} energy estimates.

Taking $\e$ sufficiently small, there holds
\beq\label{ansatz b 2}
\sup_{t\in[0,T_0/\e]}\sqrt\e\bigl(\|J_\e^{\f12}\z\|_{H^{s+1}}+\|J_\e^{\f12}\vv V\|_{H^{s+1}}\bigr)\leq\f{1}{C_s}.
\eeq

Under the ansatz \eqref{ansatz a 2} and \eqref{ansatz b 2}, we derive the {\it a priori} energy estimates stated in the following proposition.

\begin{proposition}\label{prop for a priori estimate case 2}
Let $s>3$ and $(\vv V,\z)=(v,w,\z)$ be a smooth enough solution of \eqref{WTB case 4}  satisfying the ansatz \eqref{ansatz a 2} and \eqref{ansatz b 2} for some $T_0>0$. There holds
\beq\label{total energy estimate a}
E_s(t)\leq C_1E_s(0)+C_2\e t\max_{\tau\in[0,t]} E_s(\tau)^{\f32},\quad\forall\, t\in[0,T_0/\e],
\eeq
where $C_1,C_2>1$ are universal constants.
\end{proposition}
\begin{proof}
We divide the proof into two parts: the {\it lower order} energy estimate in terms of $(v,w,\z)$ and the {\it highest order} energy estimate in terms of $(\wt p,\wt\theta)$.

\smallskip

{\bf Step 1. Energy functionals.} We define the lower order and highest order energy functionals associated to \eqref{WTB case 4} as follows:
\beq\label{lower order functional a}
E_{s,l}(t)\eqdefa\|J^{\f12}_\e K^{\f12}_\e v\|_{H^s}^2+\|J^{\f12}_\e w\|_{H^s}^2
+\|J^{\f12}_\e\z\|_{H^s}^2\sim\|J^{\f12}_\e\vv V\|_{H^s}^2+\|J^{\f12}_\e\z\|_{H^s}^2,
\eeq
\beq\label{highest order functional a}
E_{s,h}(t)\eqdefa\|J^{\f12}_\e\na v\|_{\dH^s}^2+\|J^{\f12}_\e\na w\|_{\dH^s}^2+\|J^{\f12}_\e\na\z\|_{\dH^s}^2\sim\|J^{\f12}_\e p\|_{\dH^s}^2+\|J^{\f12}_\e\theta\|_{\dH^s}^2,
\eeq
where we used \eqref{equivalent 1} in the last equivalence.
Then there holds
\beq\label{equivalent 3}
E_s(t)\sim E_{s,l}(t)+E_{s,h}(t).
\eeq

We also define the highest order energy functional associated to \eqref{equations for wt p theta 2} as
\beno
\wt{E}_{s,h}(t)\eqdefa\|J^{\f12}_\e\wt p\|_{\dH^s}^2+\|J^{\f12}_\e\wt\theta\|_{\dH^s}^2.
\eeno

Due to \eqref{equivalent 2}, there exists $c_1>0$ such that
\beno
\wt{E}_{s,h}(t)\geq\|J^{\f12}_\e p\|_{\dH^s}^2+\|J^{\f12}_\e\theta\|_{\dH^s}^2
-c_1\e^2\bigl(\|J^{\f12}_\e\z\|_{H^s}^2+\|J^{\f12}_\e\vv V\|_{H^s}^2\bigr)\bigl(\|J^{\f12}_\e p\|_{H^s}^2+\|J^{\f12}_\e\theta\|_{H^s}^2\bigr),
\eeno
which along with ansatz \eqref{ansatz b 2} implies
\beno\begin{aligned}
\wt{E}_{s,h}(t)&\geq\|J^{\f12}_\e p\|_{\dH^s}^2+\|J^{\f12}_\e\theta\|_{\dH^s}^2
-c_1\e\bigl(\|J^{\f12}_\e p\|_{H^s}^2+\|J^{\f12}_\e\theta\|_{H^s}^2\bigr)\\
&\geq(1-c_2\e)\bigl(\|J^{\f12}_\e p\|_{\dH^s}^2+\|J^{\f12}_\e\theta\|_{\dH^s}^2\bigr)-c_2\e\bigl(\|J^{\f12}_\e p\|_{L^2}^2+\|J^{\f12}_\e\theta\|_{L^2}^2\bigr),
\end{aligned}\eeno
where $c_2>0$ is an universal constant and the fact that $\|f\|_{H^s}^2\sim\|f\|_{\dH^s}^2+\|f\|_{L^2}^2$ was used in the last inequality.

Thanks to \eqref{equivalent 1}, we have
\beno
\|J^{\f12}_\e p\|_{\dH^s}^2+\|J^{\f12}_\e\theta\|_{\dH^s}^2\sim E_{s,h}(t)
\quad\text{and}\quad
\|J^{\f12}_\e p\|_{L^2}^2+\|J^{\f12}_\e\theta\|_{L^2}^2\lesssim E_{s,l}(t).
\eeno
Then there exist constants $c_3,c_4>0$ such that
\beno
\wt{E}_{s,h}(t)\geq c_3E_{s,h}(t)-c_4E_{s,l}(t).
\eeno

On the other hand, it is easy to check that there exists $c_5>0$ such that
\beno
\wt{E}_{s,h}(t)\leq c_5 E_s(t).
\eeno

Thus, we obtain 
\beq\label{equivalent 4}
c_3E_{s,h}(t)-c_4 E_{s,l}(t)\leq \wt{E}_{s,h}(t)\leq c_5E_{s}(t),\quad\forall\, t\in[0,T_0/\e].
\eeq

{\bf Step 2. The lower order energy estimates.} 
Thanks to \eqref{WTB case 4}, one calculates 
\beno\begin{aligned}
\f12\f{d}{dt}E_{s,l}(t)
&=\bigl(J_\e  v_t\,|\,K_\e v\bigr)_{H^s}+\bigl(J_\e w_t\,|\,w\bigr)_{H^s}+\bigl(J_\e\z_t\,|\,\z\bigr)_{H^s}\\
&=-\e\bigl(\vv V\cdot\na v+\f12\z \z_x\,|\,K_\e v\bigr)_{H^s}-\e\bigl(\vv V\cdot\na w+ \f12\z \z_y\,|\,w\bigr)_{H^s}+\bigl(\vv V\cdot\na\z+ \f12 \z \div\vv V\,|\,\z\bigr)_{H^s}.
\end{aligned}\eeno
 Due to tame estimate \eqref{tame},  we obtain for $s>3$ that
\beno
\f12\f{d}{dt}E_{s,l}(t)\lesssim\e\bigl(\|\vv V\|_{H^s}^2\|\na\vv V\|_{H^s}+\|\z\|_{H^s}\|\na\z\|_{H^s}\|\vv V\|_{H^s}+\|\z\|_{H^s}^2\|\na\vv V\|_{H^s}\bigr).
\eeno
Then by virtue of \eqref{lower order functional a} and \eqref{highest order functional a} and \eqref{equivalent 3}, there exists $c_6>0$ such that
\beq\label{lower order energy a}
\f12\f{d}{dt}E_{s,l}(t)\leq c_6 \e E_s(t)^{\f32}.
\eeq

\smallskip

{\bf Step 3. The highest order energy estimates.} In this step, we derive the highest order energy estimate of \eqref{WTB case 4} via $(\wt p,\wt\theta)$.

{\it Step 3.1. Energy estimate of $(\wt p,\wt\theta)$.} Due to the definition of $\wt{E}_{s,h}(t)$ and system \eqref{equations for wt p theta 2}, one calculates 
\beq\label{R 9}
\f12\f{d}{dt}\wt{E}_{s,h}(t)=\bigl(J_\e\wt p_t\,|\,\wt p\bigr)_{\dH^s}+
\bigl(J_\e\wt\theta_t\,|\,\wt\theta\bigr)_{\dH^s}=\mathcal{S}+\mathcal{T} +\mathcal{N},
\eeq
where 
\beno\begin{aligned}
\mathcal{S} &\eqdefa\Bigl( \bigl(Y_\e B(D)+\f\e2\z B(D)
-\frac{\e^2}{6}\gamma_\e(\z,D_x) Y_\e^{-1} B^{-1}(D)\p_x^4\bigr)\wt\theta \,\Big|\,  \wt{p} \Bigr)_{\dH^s}\\
&\qquad-\Bigl( \bigl(Y_\e B(D)+\f\e2\z B(D)
-\frac{\e^2}{6}\gamma_\e(\z,D_x) Y_\e^{-1} B^{-1}(D)\p_x^4\bigr)\wt p\,\Big|\,  \wt{\theta} \Bigr)_{\dH^s}, \\
\mathcal{T} & \eqdefa -\e \bigl(\vv V\cdot \na\wt{p} \,|\, \wt{p} \bigr)_{\dH^s}-\e \bigl(\vv V\cdot \na \wt\theta \,|\, \wt{\theta} \bigr)_{\dH^s}, \\
\mathcal{N} & \eqdefa \e\bigl(N_{\wt{p}} \,|\, \wt{p} \bigr)_{\dH^s} + \e\bigl( N_{\wt\theta} \,|\, \wt{\theta} \bigr)_{\dH^s}. 
\end{aligned}\eeno

\smallskip

{\it Step 3.2. Estimate of $\cS$.} From the expression of $\cS$, we deduce that
\beno\begin{aligned}
\cS&=\f\e2\bigl\{\bigl(|D|^s(\z B(D)\wt\theta)\,\big|\,|D|^s\wt p\bigr)_{L^2}
-\bigl(|D|^s(\z B(D)\wt p)\,\big|\,|D|^s\wt\theta\bigr)_{L^2}\bigr\}\\
&\quad+\f{\e^2}{6}\bigl\{-\bigl(|D|^s(\gamma_\e(\z,D_x) Y_\e^{-1} B^{-1}(D)\p_x^4\wt\theta)\,\big|\,|D|^s\wt p\bigr)_{L^2}\\
&\qquad
+\bigl(|D|^s(\gamma_\e(\z,D_x) Y_\e^{-1} B^{-1}(D)\p_x^4\wt p)\,\big|\,|D|^s\wt\theta\bigr)_{L^2}\bigr\}\\
&\eqdefa\f\e2\cS_1+\f{\e^2}{6}\cS_2.
\end{aligned}\eeno

{\bf 1). Estimate of $\cS_1$.} Firstly, it is easy to check that 
\beno
\cS_1&=\bigl([|D|^s,\z] B(D)\wt\theta\,\big|\,|D|^s\wt p\bigr)_{L^2}
-\bigl([|D|^s,\z] B(D)\wt p\,\big|\,|D|^s\wt\theta\bigr)_{L^2}-\bigl([B(D),\z]|D|^s\wt\theta\,\big|\,|D|^s\wt p\bigr)_{L^2},
\eeno
which along with \eqref{commutator D} and \eqref{commutator B3a} implies
\beno\begin{aligned}
|\cS_1|&\lesssim\|\z\|_{H^s}\|B(D)\wt\theta\|_{H^{s-1}}\||D|^s\wt p\|_{L^2}
+\|\z\|_{H^s}\|B(D)\wt p\|_{H^{s-1}}\||D|^s\wt\theta\|_{L^2}
\|\z\|_{H^s}\||D|^s\wt\theta\|_{L^2}\||D|^s\wt p\|_{L^2}\\
&\lesssim  \|\z\|_{H^s}\|\wt p\|_{H^s}\|\na\wt\theta\|_{H^{s-1}}.
\end{aligned}\eeno
Thanks to \eqref{estimate for wt theta 2}, \eqref{lower order functional a} and \eqref{equivalent 4}, we obtain
\beq\label{R 10}
|\cS_1|\lesssim E_s(t)^{\f32}.
\eeq

{\bf 2). Estimate of $\cS_2$.}  The first term  of $\cS_2$ can be rewritten as 
\beq\label{R 10a}\begin{aligned}
&\quad-\bigl(|D|^s(\gamma_\e(\z,D_x) Y_\e^{-1} B^{-1}(D)\p_x^4\wt\theta)\,\big|\,|D|^s\wt p\bigr)_{L^2}\\
&=\bigl(\gamma_\e(\z,D_x) Y_\e^{-1} B^{-1}(D)\p_x^3|D|^s\wt\theta)\,\big|\,\p_x|D|^s\wt p\bigr)_{L^2}+\cQ_2(\wt\theta,\wt p),
\end{aligned}\eeq
where 
\beq\label{Q 2}\begin{aligned}
\cQ_2(\wt\theta,\wt p)&=\bigl(|D|^s([\p_x,\gamma_\e(\z,D_x)]Y_\e^{-1} B^{-1}(D)\p_x^3\wt\theta)\,\big|\,|D|^s\wt p\bigr)_{L^2}\\
&\qquad
+\bigl([|D|^s,\gamma_\e(\z,D_x)]Y_\e^{-1} B^{-1}(D)\p_x^3\wt\theta\,\big|\,\p_x|D|^s\wt p\bigr)_{L^2}.
\end{aligned}\eeq

While for the second term of $\cS_2$, there holds
\beno\begin{aligned}
&\quad \bigl(|D|^s(\gamma_\e(\z,D_x) Y_\e^{-1} B^{-1}(D)\p_x^4\wt p)\,\big|\,|D|^s\wt\theta\bigr)_{L^2}\\
&=-\bigl( |D|^s(\gamma_\e(\z,D_x)B^{-1}(D) Y_\e^{-1}\p_x^3\wt p)\,\big|\,\p_x|D|^s\wt\theta\bigr)_{L^2}-\bigl(|D|^s([\p_x,\gamma_\e(\z,D_x)]B^{-1}(D) Y_\e^{-1} \p_x^3\wt p)\,\big|\,|D|^s\wt\theta\bigr)_{L^2}\\ 
&=-\bigl( |D|^s(\gamma_\e(\z,D_x) Y_\e^{-1}\p_x^3\wt p)\,\big|\,B^{-1}(D)\p_x|D|^s\wt\theta\bigr)_{L^2}+\cQ_3(\wt p,\wt\theta),
\end{aligned}\eeno
where 
\beq\label{Q 3}\begin{aligned}
\cQ_3(\wt p,\wt\theta)&=-\bigl(|D|^s([\p_x,\gamma_\e(\z,D_x)]B^{-1}(D) Y_\e^{-1} \p_x^3\wt p)\,\big|\,|D|^s\wt\theta\bigr)_{L^2}\\
&\qquad+\bigl( |D|^s([B^{-1}(D),\gamma_\e(\z,D_x)] Y_\e^{-1}\p_x^3\wt p)\,\big|\,\p_x|D|^s\wt\theta\bigr)_{L^2}.
\end{aligned}\eeq
For term $-\bigl( |D|^s(\gamma_\e(\z,D_x) Y_\e^{-1}\p_x^3\wt p)\,\big|\,B^{-1}(D)\p_x|D|^s\wt\theta\bigr)_{L^2}$, direct calculation shows that
\beno\begin{aligned}
&\quad-\bigl( |D|^s(\gamma_\e(\z,D_x) Y_\e^{-1}\p_x^3\wt p)\,\big|\,B^{-1}(D)\p_x|D|^s\wt\theta\bigr)_{L^2}\\
&=-\bigl(\gamma_\e(\z,D_x) Y_\e^{-1}\p_x|D|^s\wt p\,\big|\,B^{-1}(D)\p_x^3|D|^s\wt\theta\bigr)_{L^2}+\cQ'_4(\wt p, \wt\theta)\\
&=-\bigl(\gamma_\e(\z,D_x)\p_x|D|^s\wt p\,\big|\,Y_\e^{-1}B^{-1}(D)\p_x^3|D|^s\wt\theta\bigr)_{L^2}+\cQ_4(\wt p, \wt\theta)
\end{aligned}\eeno
where
\beq\label{Q 4}\begin{aligned}
\cQ'_4(\wt p, \wt\theta)&=\bigl(|D|^s([\p_x^2,\gamma_\e(\z,D_x)] Y_\e^{-1}\p_x\wt p)\,\big|\,B^{-1}(D)\p_x|D|^s\wt\theta\bigr)_{L^2}\\
&\qquad-\bigl([|D|^s,\gamma_\e(\z,D_x)] Y_\e^{-1}\p_x\wt p)\,\big|\,B^{-1}(D)\p_x^3|D|^s\wt\theta\bigr)_{L^2},\\
\cQ_4(\wt p, \wt\theta)&=\f{\e}{6}\bigl([\p_x^2,\gamma_\e(\z,D_x)] Y_\e^{-1}\p_x|D|^s\wt p\,\big|\,Y_\e^{-1}B^{-1}(D)\p_x^3|D|^s\wt\theta\bigr)_{L^2}+\cQ'_4(\wt p, \wt\theta).
\end{aligned}\eeq

Then we deduce that
\beno\begin{aligned}
&\quad\bigl(|D|^s(\gamma_\e(\z,D_x) Y_\e^{-1} B^{-1}(D)\p_x^4\wt p)\,\big|\,|D|^s\wt\theta\bigr)_{L^2}\\
&=-\bigl(\gamma_\e(\z,D_x)\p_x|D|^s\wt p\,\big|\,Y_\e^{-1}B^{-1}(D)\p_x^3|D|^s\wt\theta\bigr)_{L^2}+\cQ_3(\wt p, \wt\theta)+\cQ_4(\wt p, \wt\theta).
\end{aligned}\eeno
Notice that for any $J^{\f12}_\e\wt p,J^{\f12}_\e\wt\theta\in\dH^s(R^2)$, there holds
\beno
\p_x|D|^s\wt p, \, Y_\e^{-1}B^{-1}(D)\p_x^3|D|^s\wt\theta\in L^2(\R^2).
\eeno
Thanks to Lemma \ref{lem for operator gamma}, we see that $\gamma_\e(\z,D_x)$ is a self-adjoint operator on $L^2(\R^2)$. Thus, using \eqref{adjoint}, we get
\beno\begin{aligned}
&\quad\bigl(|D|^s(\gamma_\e(\z,D_x) Y_\e^{-1} B^{-1}(D)\p_x^4\wt p)\,\big|\,|D|^s\wt\theta\bigr)_{L^2}\\
&=-\bigl(\p_x|D|^s\wt p\,\big|\,\gamma_\e(\z,D_x)Y_\e^{-1}B^{-1}(D)\p_x^3|D|^s\wt\theta\bigr)_{L^2}+\cQ_3(\wt p, \wt\theta)+\cQ_4(\wt p, \wt\theta),
\end{aligned}\eeno
which along with \eqref{R 10a} implies
\beq\label{S 2}
\cS_2=\cQ_2(\wt\theta,\wt p)+\cQ_3(\wt p, \wt\theta)+\cQ_4(\wt p, \wt\theta).
\eeq

Now, we estimate  the r.h.s terms of \eqref{S 2} one by one.

{\bf Estimate of $\cQ_2(\wt\theta,\wt p)$.} Due to {\eqref{commutator gamma X} for $X=J_\e^{-\f12}H^s(\R^2)$} and \eqref{commutator gamma 0}, we deduce from \eqref{Q 2} that
\beno\begin{aligned}
\e|\cQ_2(\wt\theta,\wt p)|&\lesssim\e\|J^{-\f12}_\e([\p_x,\gamma_\e(\z,D_x)]Y_\e^{-1} B^{-1}(D)\p_x^3\wt\theta)\|_{\dH^s}\|J^{\f12}_\e\wt p\|_{\dH^s}\\
&\qquad+\sqrt\e\|J^{\f12}_\e([|D|^s,\gamma_\e(\z,D_x)]Y_\e^{-1} B^{-1}(D)\p_x^3\wt\theta)\|_{L^2}\cdot\sqrt\e\|J^{-\f12}_\e\p_x\wt p\|_{\dH^s}\\
&\lesssim\|J^{\f12}_\e\z\|_{H^{s+1}}\cdot\e\|J^{-\f12}_\e Y_\e^{-1} B^{-1}(D)\p_x^3\wt\theta\|_{H^s}\cdot\|J^{\f12}_\e\wt p\|_{\dH^s}\\
&\qquad 
+\|J^{\f12}_\e\z\|_{H^s}\cdot\sqrt\e\|J^{\f12}_\e Y_\e^{-1} B^{-1}(D)\p_x^3\wt\theta\|_{H^{s-1}}\cdot\|\wt p\|_{\dH^s},
\end{aligned}\eeno
implying
\beq\label{R 11}
\e|\cQ_2(\wt\theta,\wt p)|\lesssim\|\z\|_{H^{s+1}}\|\na\wt\theta\|_{H^{s-1}}\|J^{\f12}_\e\wt p\|_{\dH^s}.
\eeq

{\bf Estimate of $\cQ_3(\wt p, \wt\theta)$.}
Using {\eqref{commutator gamma X} for $X=J_\e^{-\f12}H^s(\R^2)$} and \eqref{commutator gamma 4b}, we derive from \eqref{Q 3} that 
\beno\begin{aligned}
\e|\cQ_3(\wt p,\wt\theta)|&\lesssim 
\|J^{\f12}_\e\z\|_{H^{s+1}}\cdot\e\|J^{-\f12}_\e Y_\e^{-1} B^{-1}(D)\p_x^3\wt p\|_{H^s}\cdot\|J^{\f12}_\e\wt\theta\|_{\dH^s}\\
&\qquad 
+\|J^{\f12}_\e\z\|_{H^{s+1}}\cdot\sqrt\e\||D|^{-1} Y_\e^{-1}\p_x^3\wt p\|_{H^{s-1}}\cdot\sqrt\e\|\p_x\wt\theta\|_{\dH^s},
\end{aligned}\eeno
leading to
\beq\label{R 12}
\e|\cQ_3(\wt p,\wt\theta)|\lesssim\|J^{\f12}_\e\z\|_{H^{s+1}}\|\wt p\|_{H^s}\|J^{\f12}_\e\wt\theta\|_{\dH^s}.
\eeq

{\bf Estimate of $\cQ_4(\wt p, \wt\theta)$.}
By virtue of \eqref{commutator gamma 5} with { $r=0, s$} and \eqref{commutator gamma 0}, we obtain from \eqref{Q 4} that
\beno\begin{aligned}
\e|\cQ_4(\wt p, \wt\theta)|&\lesssim
\e\|J^{-\f12}_\e([\p_x^2,\gamma_\e(\z,D_x)] Y_\e^{-1}\p_x|D|^s\wt p)\|_{L^2}\cdot\e\|J^{\f12}_\e Y_\e^{-1}B^{-1}(D)\p_x^3\wt\theta\|_{\dH^s}\\
&\qquad+\e\|J^{-\f12}_\e([\p_x^2,\gamma_\e(\z,D_x)] Y_\e^{-1}\p_x\wt p)\|_{\dH^s}\cdot\|J^{\f12}_\e B^{-1}(D)\p_x\wt\theta\|_{\dH^s}\\
&\qquad
+\|J^{\f12}_\e([|D|^s,\gamma_\e(\z,D_x)] Y_\e^{-1}\p_x\wt p)\|_{L^2}\cdot\e\|J^{-\f12}_\e B^{-1}(D)\p_x^3\wt\theta\|_{\dH^s}\\
&\lesssim\sqrt{\e}\|J^{\f12}_\e\z\|_{H^{s+1}}\| Y_\e^{-1}\p_x|D|^s\wt p\|_{L^2}\cdot\|J^{\f12}_\e\wt\theta\|_{\dH^s}+\sqrt\e\|J^{\f12}_\e\z\|_{H^{s+1}}\|Y_\e^{-1}\p_x\wt p\|_{H^s}\cdot\|J^{\f12}_\e\wt\theta\|_{\dH^s}\\
&\qquad+\|J^{\f12}_\e\z\|_{H^s}\|J^{\f12}_\e Y_\e^{-1}\p_x\wt p\|_{H^{s-1}}\cdot\|J^{\f12}_\e\wt\theta\|_{\dH^s},
\end{aligned}\eeno
implying
\beq\label{R 13}
\e|\cQ_4(\wt p, \wt\theta)|\lesssim\|J^{\f12}_\e\z\|_{H^{s+1}}\|\wt p\|_{H^s}\|J^{\f12}_\e\wt\theta\|_{\dH^s}.
\eeq

{\bf Estimate of $\cS_2$.} Thanks to \eqref{R 11}, \eqref{R 12} and \eqref{R 13}, we get 
\beno
\e|\cS_2|\lesssim\|J^{\f12}_\e\z\|_{H^{s+1}}\|J^{\f12}_\e\wt p\|_{\dH^s}\|J^{\f12}_\e\na\wt\theta\|_{H^{s-1}},
\eeno
which along with \eqref{estimate for wt theta 2}, \eqref{lower order functional a} and \eqref{equivalent 4} implies
\beq\label{R 14}
|\cS_2|\lesssim E_{s}(t)^{\f32}.
\eeq

{\bf 3). Estimate for $\cS$.} Combining \eqref{R 10} and \eqref{R 14}, we obtain
\beq\label{estimate for S a}
|\cS|\lesssim\e E_{s}(t)^{\f32}.
\eeq

\smallskip

{\it Step 3.3. Estimate of $\cT$.} For the first term of $\cT$, it is easy to check that
\beno\begin{aligned}
\bigl(\vv V\cdot \na\wt{p} \,|\, \wt{p} \bigr)_{\dH^s}
&=\bigl([|D|^s,\vv V]\cdot \na\wt{p} \,|\, |D|^s\wt{p} \bigr)_{L^2}
+\bigl(\vv V\cdot \na |D|^s\wt{p} \,|\, |D|^s\wt{p} \bigr)_{L^2}\\
&=\bigl([|D|^s,\vv V]\cdot \na\wt{p} \,|\, |D|^s\wt{p} \bigr)_{L^2}
-\f12\bigl(\div\vv V |D|^s\wt{p} \,|\, |D|^s\wt{p} \bigr)_{L^2},
\end{aligned}\eeno
which along with \eqref{commutator D} and Sobolev embedding theorem yields to
\beq\label{R 15}
|\bigl(\vv V\cdot \na\wt{p} \,|\, \wt{p} \bigr)_{\dH^s}|
\lesssim\|\vv V\|_{H^s}\|\na\wt p\|_{H^{s-1}}^2.
\eeq
Similar estimate holds for the last term $\bigl(\vv V\cdot \na\wt\theta \,|\, \wt\theta \bigr)_{\dH^s}$. Then we obtain 
\beno
|\cT|\lesssim\e\|\vv V\|_{H^s}(\|\na\wt p\|_{H^{s-1}}^2+\|\na\wt\theta\|_{H^{s-1}}
^2)\lesssim\e\|\vv V\|_{H^s}(\|\wt p\|_{H^s}^2+\|\na\wt\theta\|_{H^{s-1}}
^2),
\eeno
which together with \eqref{estimate for wt theta 2}, \eqref{lower order functional a} and \eqref{equivalent 4} implies
\beq\label{estimate for T a}
|\cT|\lesssim\e E_{s}(t)^{\f32}.
\eeq

\smallskip

{\it Step 3.4. Estimate of $\cN$.} For $\cN$, it is easy to check that
\beno
|\cN|\lesssim\e\|J^{-\f12}_\e N_{\wt{p}}\|_{\dH^s}\|J^{\f12}_\e\wt{p}\|_{\dH^s} +\e\|J^{-\f12}_\e N_{\wt\theta}\|_{\dH^s}\|J^{\f12}_\e\wt\theta\|_{\dH^s},
\eeno
which along with \eqref{estimate for N wt p theta 2} and \eqref{estimate for wt theta 2} implies
\beno
|\cN|\lesssim\e\bigl(\|J^{\f12}_\e p\|_{H^s}+\|J^{\f12}_\e\theta\|_{H^s}\bigr)^3.
\eeno
Due to \eqref{lower order functional a} and \eqref{highest order functional a}, we get
\beq\label{estimate for N a}
|\cN|\lesssim\e E_{s}(t)^{\f32}.
\eeq

\smallskip

{\it Step 3.5. Estimate of $(\wt p,\wt\theta)$.} Combining \eqref{estimate for S a}, \eqref{estimate for T a} and \eqref{estimate for N a}, we deduce from \eqref{R 9} that
\beq\label{highest order energy a}
\f12\f{d}{dt}\wt{E}_{s,h}(t)\lesssim\e E_{s}(t)^{\f32}.
\eeq

\smallskip

{\bf Step 3. The total energy estimate.} Thanks to \eqref{lower order energy a} and \eqref{highest order energy a}, we arrive at
\beno
\f12\f{d}{dt}\bigl(2c_4E_{s,l}(t)+\wt{E}_{s,h}(t)\bigr)\lesssim \e E_s(t)^{\f32},
\eeno
leading to
\beno
2c_4E_{s,l}(t)+\wt{E}_{s,h}(t)\lesssim E_{s,l}(0)+\wt{E}_{s,h}(0)+\e t\max_{\tau\in[0,\f{T_0}{\e}]}E_s(\tau)^{\f32},\quad\forall\, t\in[0,T_0/\e].
\eeno
By virtue of \eqref{equivalent 3}, \eqref{equivalent 4}, \eqref{lower order functional a} and \eqref{highest order functional a}, there exist $C_1,C_2>1$ such that
\beno
E_s(t)\leq C_1 E_s(0)+C_2\e t\max_{\tau\in[0,T_0/\e]}E_s(\tau)^{\f32},\quad\forall\, t\in[0,T_0/\e].
\eeno
This is the desired total  energy estimate \eqref{total energy estimate a}. The proposition is proved.
\end{proof}
\begin{remark}
By virtue of \eqref{total energy estimate a}, taking $C_0=2C_1$ and $T_0=\f{1}{4\sqrt{C_1}C_2E^{\f12}_s(0)}$, one improves the ansatz \eqref{ansatz a 2} by replacing the constant $2C_0$ to $C_0$. Then taking $\e$ sufficiently small, one could improve the ansatz \eqref{ansatz b 2} by replacing the constant $\f{1}{C_s}$ to $\f{1}{2C_s}$. Therefore, the continuity argument is closed and Theorem \ref{main theorem} for case 2 is proven.
\end{remark}

	\section{Final remarks}
This paper is concerned with {\it long time existence}, that is an   important non trivial step  (in order to justify the models on the correct time scales) between local and global existence. The question of global well-posedness versus finite type blow-up is open for weakly transverse Boussinesq systems as well as the possible global existence of solutions  for small initial data (see \cite{KMPP, KM} for results to this last issue for the  isotropic "abcd" Boussinesq systems).

			\vspace{0.5cm}
			
			\noindent {\bf Acknowledgments.}  The work of the second  author was partially  supported by the ANR project ANuI (ANR-17-CE40-0035-02).
			The work
			of the first and the third authors were partially supported by NSF of China under grants 12171019.

		\end{document}